\documentclass[11pt,letterpaper]{amsart}
\usepackage[utf8]{inputenc}

\title{The BGMN conjecture via stable pairs}
\author{Jenia Tevelev}
\author{Sebasti\'an Torres}
\dedicatory{In memory of M. S. Narasimhan}

\usepackage{amsmath}	% v2
\usepackage{amsfonts}	% v2
\usepackage{comment}
\usepackage{amsfonts}

\usepackage{amsthm}
\usepackage{amssymb}
\usepackage{graphicx}
\usepackage{pgfplots}
\usepackage{tikz-cd}
\usepackage[abbrev,nobysame,alphabetic]{amsrefs}
\usepackage{amsmath,calligra,mathrsfs}
\usepackage{mathtools}
\usepackage{multicol}

\usepackage{chngcntr}
\usepackage{pgf}
\usetikzlibrary {positioning}
\usetikzlibrary {arrows,automata}
\usetikzlibrary{shapes}

\newcommand{\ST}[1]{{\color{blue} #1}}

\theoremstyle{plain}
\newtheorem{theorem}{Theorem}[section]
\newtheorem{proposition}[theorem]{Proposition}
\newtheorem{lemma}[theorem]{Lemma}
\newtheorem{corollary}[theorem]{Corollary}

\theoremstyle{definition}
\newtheorem{definition}[theorem]{Definition}

\newtheorem{notation}[theorem]{Notation}
\newtheorem{remark}[theorem]{Remark}

\theoremstyle{remark}

\newtheorem{claim}[theorem]{Claim}
\newtheorem{step}{Step}
\counterwithin*{step}{subsection}

\numberwithin{equation}{section}

\DeclareMathOperator{\spec}{Spec}

\DeclareMathOperator{\git}{/\!\!/}

\DeclareMathOperator{\Hom}{Hom}
\DeclareMathOperator{\Ext}{Ext}
\DeclareMathOperator{\Spec}{Spec}
\DeclareMathOperator{\Pic}{Pic}

\DeclareMathOperator{\codim}{codim}
\DeclareMathOperator{\Id}{Id}
\DeclareMathOperator{\weight}{weight}
\DeclareMathOperator{\weights}{weights}

\DeclareMathOperator{\gr}{gr}
\DeclareMathOperator{\rank}{rk}

\DeclareMathOperator{\Quot}{Quot}
\DeclareMathOperator{\sHom}{\mathscr{H}\text{\kern -3pt {\calligra\large om}}\,}
\DeclareMathOperator{\sExt}{\mathscr{E}\text{\kern -3pt {\calligra\large xt}}\,}

\DeclareMathOperator{\Sym}{Sym}

\DeclareMathOperator{\Bl}{Bl}
\DeclareMathOperator{\pt}{pt}
\DeclareMathOperator{\Res}{Res}

\DeclareMathOperator{\ch}{ch}
\DeclareMathOperator{\td}{td}
\DeclareMathOperator{\sign}{sgn}
\DeclareMathOperator{\supp}{supp}
\DeclareMathOperator{\mult}{mult}
\DeclareMathOperator{\End}{End}
\DeclareMathOperator{\topp}{top}
\newcommand{\win}{w}
\DeclareMathOperator{\cD}{\mathcal{D}}
\DeclareMathOperator{\cG}{\mathcal{G}}
\DeclareMathOperator{\cI}{\mathcal{I}}
\DeclareMathOperator{\cM}{\mathcal{M}}
\DeclareMathOperator{\cE}{\mathcal{E}}
\DeclareMathOperator{\cF}{\mathcal{F}}
\DeclareMathOperator{\cZ}{\mathcal{Z}}
\DeclareMathOperator{\cL}{\mathcal{L}}
\DeclareMathOperator{\cO}{\mathcal{O}}
\DeclareMathOperator{\cP}{\mathcal{P}}
\DeclareMathOperator{\cA}{\mathcal{A}}
\DeclareMathOperator{\cH}{\mathcal{H}}
\DeclareMathOperator{\bP}{{\mathbb{P}}}
\DeclareMathOperator{\bA}{{\mathbb{A}}}
\DeclareMathOperator{\bR}{{\mathbb{R}}}
\DeclareMathOperator{\bB}{{\mathbb{B}}}
\DeclareMathOperator{\bD}{{\mathbb{D}}}

\pgfplotsset{compat=1.17}

\def\bC{\mathbb{C}}

\begin{document}

\maketitle

\begin{abstract}
Let $C$ be a smooth  projective curve of genus $g\ge2$ and let $N$ be the moduli space of stable vector bundles on $C$ of rank $2$ and fixed determinant of odd degree.
We~construct a semi-orthogonal decomposition of $D^b(N)$ conjectured by Narasimhan and by Belmans, Galkin and Mukhopadhyay. It has two blocks for each $i$-th symmetric power of~$C$ for $i=0,\ldots,g-2$ and one block for the $(g-1)$-st symmetric power. We conjecture that the subcategory generated by our blocks has a trivial semi-orthogonal complement, proving  the full BGMN conjecture. Our proof is based on an analysis of wall-crossing between moduli spaces of stable pairs, combining classical vector bundles techniques with the method of windows.
\end{abstract}

\section{Introduction}

Let $C$ be a smooth complex projective curve of genus $g\ge2$. Let $N=M_C(2,\Lambda)$ be the moduli space of stable vector bundles on $C$ of rank $2$ and fixed determinant $\Lambda$ of odd degree. It is a smooth  Fano variety of index~$2$, with $\Pic N=\mathbb{Z}\cdot\theta$ for some ample line bundle $\theta$. 

\begin{theorem}\label{maintheorem}
$D^b(N)$ has a semi-orthogonal decomposition $\langle \cP,\cA\rangle$, where %$P$ is a ``phantom'' subcategory and $\cA$ has a semi-orthogonal decomposition
\begin{equation}\label{SODonN}
\begin{matrix}
  \hskip-6pt\cA=\langle
&\theta^*\otimes\cG_0,
&(\theta^*)^2\otimes\cG_2,
&(\theta^*)^3\otimes\cG_4,
&(\theta^*)^4\otimes\cG_6,
&\ldots,&
\\
&\ldots,
&(\theta^*)^4\otimes\overline{\cG}_7,
&(\theta^*)^3\otimes\overline{\cG}_5,
&(\theta^*)^2\otimes\overline{\cG}_3,
&\theta^*\otimes\overline{\cG}_1,&
\\
&\cG_0,
&\theta^*\otimes\cG_2,
&(\theta^*)^2\otimes\cG_4,
&(\theta^*)^{3}\otimes\cG_6,
&\ldots,&
\\
&\ldots
&(\theta^*)^3\otimes\overline{\cG}_7,
&(\theta^*)^2\otimes\overline{\cG}_5,
&\theta^*\otimes\overline{\cG}_3,
&\overline{\cG}_1
&\rangle\\
\end{matrix}
\end{equation}
Each subcategory $\cG_i\simeq D^b(\Sym^i C)$ (resp. $\overline{\cG}_i\simeq D^b(\Sym^i C)$) is embedded in $D^b(N)$ by a fully faithful Fourier--Mukai functor with kernel given by the $i$-th tensor bundle $\cE^{\boxtimes i}$ (resp. $\overline{\cE}^{\boxtimes i}$) (see Section~\ref{TensorBundles}) 
of the Poincar\'e bundle $\cE$ on $C\times N$ normalized so that $\det\cE_x\simeq\theta$ for every $x\in C$.

There are two blocks isomorphic to
$D^b(\Sym^i C)$ for each $i=0,\ldots,g-2$ and one block isomorphic to $D^b(\Sym^{g-1}C)$, which 
appears on the $1$st or $2$nd line of~\eqref{SODonN}, depending on parity of $g$.
\end{theorem}

The blocks appearing in \eqref{SODonN} cannot be further decomposed \cite{lin}.
%, but there is some flexibility in their order. Most importantly, blocks within each of the four lines are mutually orthogonal (see Corollary~\ref{rearrangement}). 
Remarkably, our decomposition is compatible with  the results of Mu\~{n}oz  \cites{munoz1,munoz2,munoz3} (cf.~\cite[Proposition 6.4.2]{belmans21}), that the operator of the quantum multiplication by $c_1(N)$ on the quantum cohomology $QH^\bullet(N)$ 
has eigenvalues $8\lambda$, where 
$$\lambda=(1-g),\ (2-g)\sqrt{-1},\ (3-g),\ldots,\ (g-3),\ (g-2)\sqrt{-1},\ (g-1)$$
and the eigenspace of $8\lambda$ is isomorphic to  $H^\bullet(\Sym^{g-1-|\lambda|}C)$.
%The four lines of \eqref{SODonN} correspond to the four coordinate rays of the complex plane.
There are many other results, e.g.~  \cites{delbano,lee}, on cohomology and motivic decomposition of $N$ compatible with ~\eqref{SODonN}. This provides an ample evidence towards the expectation that $\cP=0$.
%, i.e.~that Theorem~\ref{maintheorem} provides a complete answer to the BGMN conjecture. 
We hope to address this question in the future, as well as to use our methods to study properties of analogous Fourier--Mukai functors for moduli spaces of vector bundles of higher rank on curves and for moduli spaces of sheaves with $1$-dimensional support on K3 surfaces.

Partial results towards Theorem~\ref{maintheorem} have appeared in the literature. The case $g=2$ is a classical theorem of Bondal and Orlov \cite[Theorem 2.9]{bondal-orlov}, who also proved that $\cP=0$ in that case. Fonarev and Kuznetsov \cite{fonarev} proved that
$D^b(C)\hookrightarrow D^b(N)$ if $C$ is a hyperelliptic curve using 
an explicit description of $N$ due to Desale and Ramanan \cite{desale}.
They also proved that $D^b(C)\hookrightarrow D^b(N)$ for a general curve $C$ by a deformation argument.  Narasimhan proved that $D^b(C)\hookrightarrow D^b(N)$ for all curves \cites{narasimhan1,narasimhan2} using Hecke correspondences. He also showed that one can add the line bundles $\cO$ and $\theta^*$ to $D^b(C)$ to start a semi-orthogonal decomposition of $D^b(N)$.

In \cite{belmans}, Belmans and Mukhopadhyay work with the moduli space $M_C(r,\Lambda)$ of vector bundles of rank $r$ and determinant $\Lambda$, where $r\ge 2$ and $\deg \Lambda=1$. They show that there is a fully faithful embedding $D^b(C)\hookrightarrow D^b(M_C(r,\Lambda))$ provided the genus is sufficiently high. Moreover, they use this embedding to find the start of a semi-orthogonal decomposition of $D^b(M_C(r,\Lambda))$ of the form $\theta^*$, $D^b(C)$, $\cO$, $\theta^*\otimes D^b(C)$, this way extending the decomposition on $N=M_C(2,\Lambda)$ present in \cite{narasimhan2}.
Belmans, Galkin and Mukhopadhyay have conjectured, independently of Narasimhan,  that $D^b(N)$ should have a semi-orthogonal decomposition with blocks $D^b(\Sym^i C)$ (see \cites{oberwolfach, lee}), and have collected additional evidence towards this conjecture in \cite{belmans21}. Lee and Narasimhan \cite{lee-narasimhan}
proved using Hecke correspondences that, if $C$ is non-hyperelliptic and $g\ge 16$, there is a fully faithful functor $D^b(\Sym^2C)\hookrightarrow D^b(N)$ whose image is left semi-orthogonal to the copy of $D^b(C)$ obtained earlier. They also introduced tensor bundles $\cE^{\boxtimes i}$ of the Poincar\'e bundle (see Section~\ref{TensorBundles}), which we discovered independently.
If $D\in\Sym^i C$ is a reduced sum of points $x_1+\ldots+x_i$, the fiber $(\cE^{\boxtimes i})_D$ is a vector bundle on $N$ isomorphic to the tensor product $\cE_{x_1}\otimes\ldots\otimes\cE_{x_i}$. If the points have multiplicities, 
$(\cE^{\boxtimes i})_D$ is a deformation of the tensor product over $\bA^1$ (see Corollary~\ref{standarddeformation}). 

Instead of using Hecke correspondences (although they do make a guest appearance in Section~\ref{,jshBDFjhsbf}), 
we prove Theorem~\ref{maintheorem} by analyzing  Fourier--Mukai 
functors given by tensor bundles $F^{\boxtimes i}$ of the universal bundle $F$ on the moduli space of stable pairs
$(E,\phi)$, where $E$ is a rank-two vector bundle on $C$ with fixed odd determinant line bundle of degree $d$ and $\phi\in H^0(E)$ is a non-zero section.
The stability condition on these spaces depends on a parameter, and we use extensively results of Thaddeus \cite{thaddeus} on wall-crossing. 
If $d=2g-1$ then there is a well-known diagram of flips
\begin{equation}%\label{diagram Mi}
\begin{tikzcd}[cramped,sep=scriptsize]
& \tilde{M}_2\arrow[dl]\arrow[dr] && \tilde{M}_3\arrow[dl]\arrow[dr] && \tilde{M}_{g-1}\arrow[dl]\arrow[dr] & \\
M_1\arrow[d] & & M_2 & & \cdots & & M_{g-1}\arrow[d]\\
M_0 &&  &&  &&  N
\end{tikzcd}
\end{equation}
where $M_0={\mathbb{P}}^{3g-3}$,
$M_1$ is the blow up of $M_0$ in $C$, the rational map $M_{i-1}\dashrightarrow M_{i}$ is a standard flip of projective bundles over $\Sym^{i} C$, and $\xi:\,M_{g-1}\to N$
is a birational Abel--Jacobi map with fiber $\mathbb{P}H^0(E)$ over a stable vector bundle~$E$. Accordingly, $D^b(M_{i})$ has a semi-orthogonal decomposition into $D^b(M_{i-1})$ and several blocks equivalent to $D^b(\Sym^{i}C)$ with torsion supports (see Proposition~\ref{potashnik embedding} or \cite{bfr}). While these decompositions do not descend to $N$ and are not associated with the universal bundle, they are useful. Philosophically,  tensor bundles on $\Sym^iC\times N$ are related to 
exterior powers of the tautological bundle of the universal bundle, which  appear
in the Koszul complex of the  tautological section that vanishes on the flipped  locus. One can  try to connect two Fourier--Mukai functors via  mutations. In practice, this Koszul complex is difficult to analyze except for $M_1$ (see~Section~\ref{M1 section}). We followed another strategy towards proving Theorem~\ref{maintheorem}. 

In order to prove semi-orthogonality in \eqref{SODonN} and full faithfulness of the Fourier--Mukai functors via the Bondal--Orlov criterion, we had to investigate coherent cohomology for a large class of vector bundles.
The main difficulty in this kind of analysis is to find a priori numerical bounds on the class of acyclic vector bundles to get the induction going. 

\begin{definition}
    For an object $\cF$ in the derived category of a scheme $M$, we say that $\cF$ is $\Gamma$-acyclic if $R\Gamma (\cF)=0$. That is, for us $\Gamma$-acyclicity will mean vanishing of \emph{all} cohomology groups, including $H^0(\cF)$. Other authors have used the term \emph{immaculate} for this property (cf. \cite{immaculate}).
\end{definition}
Theorem \ref{maintheorem} then requires the proof of $\Gamma$-acyclicity for several vector bundles. It~is worth emphasizing that the moduli space $N$ depends on the complex structure of the curve~$C$ by a classical theorem of Mumford and Newstead \cite{mumford-newstead} later extended by Narasimhan and Ramanan \cite{narasimhan-ramanan75}.
The uniform shape of Theorem~\ref{maintheorem} is thus a surprisingly strong statement about coherent cohomology of vector bundles on $N$ that does not involve any conditions of the Brill--Noether type.
Our approach utilizes the method of windows into derived categories of GIT quotients of Teleman, Halpern--Leistner, and Ballard--Favero--Katzarkov \cites{teleman,dhl,katzarkov}
to systematically predict behavior of coherent cohomology under wall-crossing.
This dramatically reduces otherwise unwieldy cohomological computations to a few key cases, which can be analyzed using other techniques. Rather unexpectedly, one of the difficult ingredients in the proof is acyclicity of certain line bundles (see Section~\ref{,jshBDFjhsbf}). While cohomology of line bundles on the space of stable pairs was extensively studied in \cite{thaddeus} in order to prove the Verlinde formula, the line bundle that we need is outside of the scope of that paper.

Analogous recent applications of windows to moduli spaces include the proof of the Manin--Orlov conjecture on $\bar M_{0,n}$ by Castravet and Tevelev \cites{castravet1,castravet1.1,castravet2} and analysis of Bott vanishing on GIT quotients by Torres \cite{myself}.

We are grateful to Pieter Belmans for bringing the papers \cites{oberwolfach,belmans21} to our attention,
to Elias Sink for useful comments, 
and to the anonymous referee for sending us a long list of corrections and thoughtful suggestions.
J.T. was supported by the NSF grant DMS-2101726. %and the HSE University Basic Research Program and Russian Academic Excellence Project `5-100'. 
Some results of the paper have first appeared in the PhD thesis \cite{mythesis} of S.T.

%\tableofcontents

\section{Tensor vector bundles}\label{TensorBundles}

Let $C$ be a smooth projective curve  over $\mathbb{C}$. For integers $\alpha\geq 1$ and $1\leq j\leq \alpha$, let $\pi_j:C^\alpha\to C$ be the $j$-th projection  and $\tau:C^\alpha\to \Sym^\alpha C$ the categorical $S_\alpha$-quotient, where $S_\alpha$ is the symmetric group. 
%Since $\tau$ is finite, $\tau_*$ is exact. 
Since $C^\alpha$ is Cohen--Macaulay (in fact smooth), $\Sym^\alpha C$ is smooth, and $\tau$ is equi-dimensional, we conclude that $\tau$ is flat by miracle flatness. Therefore, any base change
$\tau:\,C^\alpha\times M\to \Sym^\alpha C\times M$ is also a finite and flat categorical $S_\alpha$-quotient, where $M$ is any scheme over $\mathbb{C}$. The constructions in this section are functorial in $M$.
In~the following sections, $M$ will be one of the moduli spaces we consider. %The~following lemma follows a similar line of argument to \cite[2.1]{thaddeus}.

\begin{notation}
    For an $S_\alpha$-equivariant vector bundle $\cE$ on $C^\alpha\times M$, we will denote by $\tau_*^{S_\alpha}\cE$ the $S_\alpha$-invariant part of the pushforward $\tau_*\cE$.
\end{notation}

\begin{lemma}
Let $\cE$ be an $S_\alpha$-equivariant locally free sheaf on $C^\alpha\times M$. Then $\tau_*\cE$ and $\tau_*^{S_\alpha}\cE$
are locally free sheaves on $\Sym^\alpha C\times M$.
\end{lemma}

\begin{proof} The scheme
$C^\alpha\times M$ is covered by $S_\alpha$-equivariant affine charts $\Spec R$ and $\tau^*$ is given by the inclusion of invariants $R^{S_\alpha}\subset R$. Since $R$ is a finitely generated and flat $R^{S_\alpha}$-module, it is also a projective $R^{S_\alpha}$-module. Let $E=H^0(\Spec R,\cE)$. Since $E$ is a projective $R$-module, it is a direct summand of $R^s$ for some $s$. It follows that $E$ is a projective $R^{S_\alpha}$-module, i.e.~$\tau_*\cE$ is locally free. Since $E^{S_\alpha}$ is a direct summand of $E$ as an $R^{S_\alpha}$-module, it is also a projective $R^{S_\alpha}$-module. Therefore,
$\tau_*^{S_\alpha}\cE$ is a locally free sheaf as well.
\end{proof}

\begin{definition}\label{tensor products and sums}
    For any vector bundle $\mathcal{F}$ on $C\times M$, we  define the following {\em tensor vector bundles} on $\Sym^\alpha C\times M$,
$$
    \mathcal{F}^{\boxtimes\alpha}=\tau_*^{S_\alpha}\left(\bigotimes_{j=1}^\alpha\pi_j^* \mathcal{F}\right)\quad\hbox{\rm and}\quad
    \overline{\mathcal{F}}^{\boxtimes\alpha}=\tau_*^{S_\alpha}\left(\bigotimes_{j=1}^\alpha\pi_j^* \mathcal{F}\otimes\sign\right),
$$
    where $S_\alpha$ acts on $C^\alpha$ and also permutes the  factors of the corresponding vector bundle on $C^\alpha$. Here  $\sign$ is the sign representation of $S_\alpha$.
\end{definition}

\begin{lemma}\label{tensorfunctorial}
The formation of tensor vector bundles is functorial in $M$, that is, given a morphism $f:\,M'\to M$ and its base changes $C\times M'\to C\times M$ and $\Sym^\alpha C\times M'\to \Sym^\alpha C\times M$, which we also denote by $f$, we have
$$
f^*(\mathcal{F}^{\boxtimes\alpha})=(f^*\mathcal{F})^{\boxtimes\alpha}\quad\hbox{\rm and}\quad
f^*(\overline{\mathcal{F}}^{\boxtimes\alpha})=\overline{(f^*\mathcal{F})}^{\boxtimes\alpha}.
$$
\end{lemma}

\begin{proof}
Since $\tau$ is flat, this follows from cohomology and base change.
\end{proof}

%\begin{remark}
%Similarly, one can define bundles $\mathcal{F}^{\boxplus\alpha}=\tau_*^{S_\alpha}\left(\bigoplus_{j=1}^\alpha\pi_j^* \mathcal{F}\right)$ and $\overline{\mathcal{F}}^{\boxplus\alpha}=\tau_*^{S_\alpha}\left(\bigoplus_{j=1}^\alpha\pi_j^* \mathcal{F}\otimes\sign\right)$.
%In fact, $\mathcal{F}^{\boxplus\alpha}$ is the tautological bundle. 
%\end{remark}

%\begin{remark}
%Since $\cO_{C^\alpha\times M}\otimes\sign$ is the equivariant dualizing sheaf for $\tau$ (see \cite[Lemma 5.8]{ks}), by Grothendieck duality we have
%$
%%\left(\mathcal{F}^{\boxplus\alpha}\right)^\vee&=\tau_*^{S_\alpha}\left(\bigoplus_{j=1}^\alpha\pi_j^* \mathcal{F}^\vee\otimes\sign\right)\\
%%\left(\mathcal{F}^{\boxtimes\alpha}\right)^\vee&=\tau_*^{S_\alpha}\left(\bigotimes_{j=1}^\alpha\pi_j^* \mathcal{F}^\vee\otimes\sign\right).
%%\left(\mathcal{F}^{\boxplus\alpha}\right)^\vee=\overline{\mathcal{F}^\vee}^{\boxplus\alpha},\qquad
%\left(\mathcal{F}^{\boxtimes\alpha}\right)^\vee=\overline{\mathcal{F}^\vee}^{\boxtimes\alpha}.
%$
%\end{remark}

For a divisor $D\in\Sym^\alpha C$ and a vector bundle $\mathcal{G}$ on $\Sym^\alpha C\times M$, we write $\mathcal{G}_D:=\mathcal{G}|_{\{D\}\times M}$. We usually view $\mathcal{G}_D$ as a vector bundle on $M$.

\begin{remark}\label{empty divisor}
    For the empty divisor $D=0$, we have $\mathcal{G}_0\simeq \cO_M$.
\end{remark}

\begin{lemma}\label{differentpoints}
If  $D=\sum\alpha_kx_k$ with $x_k\neq x_l$ for $k\neq l$, then 
we have
%the restrictions of $\mathcal{F}^{\boxplus\alpha}$ and $\mathcal{F}^{\boxtimes\alpha}$ to $\{D\}\times M$ are
\begin{gather}
\begin{aligned}\label{restriction to different points}
    \left(\mathcal{F}^{\boxtimes\alpha}\right)_{D}=\bigotimes\left( \mathcal{F}^{\boxtimes\alpha_k}\right)_{\alpha_kx_k},\qquad
    \left(\overline{\mathcal{F}}^{\boxtimes\alpha}\right)_{D}=\bigotimes\left( \overline{\mathcal{F}}^{\boxtimes\alpha_k}\right)_{\alpha_kx_k}.
\end{aligned}
\end{gather}
%and similar formulas hold for the skew-invariant versions $(\overline{\mathcal{F}}^{\boxplus\alpha})_{D}$ and $(\overline{\mathcal{F}}^{\boxtimes\alpha})_{D}$.
\end{lemma}

\begin{proof}
Indeed, the  quotient $\tau:\,C^{\alpha}\to \Sym^{\alpha} C$  is \'etale-locally near $D\in \Sym^{\alpha} C$ isomorphic to the product of quotients 
$\prod C^{\alpha_k}\to\prod \Sym^{\alpha_k} C$.
Moreover, the stabilizer of the point $D$ under the $S_\alpha$-action is $\prod S_{\alpha_k}$, and $\sign$ restricts to the tensor product of sign representations of $\prod S_{\alpha_k}$.
\end{proof}

Consider the non-reduced scheme $\mathbb{D}_\alpha=\spec \mathbb{C}[t]/t^\alpha$, with maps $\pt \xhookrightarrow{\imath}\mathbb{D}_\alpha\xrightarrow{\rho}\pt$ given by the obvious pullbacks $\mathbb{C}\xrightarrow{\rho^\#}\mathbb{C}[t]/t^\alpha\xrightarrow{\imath^\#}\mathbb{C}$. We still write $\imath$ and $\rho$ for the base changes to $M$ of these morphisms, that is, $M\xrightarrow{\imath}\mathbb{D}_\alpha\times M\xrightarrow{\rho}M$. For a locally free sheaf $\mathcal{F}$ on $\mathbb{D}_\alpha\times M$, we denote by $\mathcal{F}_0=\imath^*\mathcal{F}$ its restriction to $M$.

\begin{definition}
For two vector bundles $\cF$, $\cG$ on a scheme $M$, we will say that $\cF$ is a \emph{deformation of $\cG$ over $\bA^1$} if there is a coherent sheaf $\Tilde{\cF}$ on $M\times \bA^1$, flat over $\bA^1$, with $\tilde{\cF}|_t\simeq \cF$ for $t\neq 0$, while $\tilde{\cF}|_{0}\simeq \cG$.
\end{definition}

%Note that if $\cF$ is a deformation of $\cG$, then by semi-continuity $h^p(\cF)\le h^p(\cG)$ for every $p$. In particular, $\Gamma$-acyclicity of $\cG$ implies that of $\cF$. %No need for this comment, this is basically the second to last remark.

\begin{lemma}\label{deformation lemma}
Every locally free sheaf $\mathcal{F}$ on $\mathbb{D}_\alpha\times M$ is a deformation of $\rho^*\mathcal{F}_0$ over $\mathbb{A}^1$. In particular, $\rho_*\mathcal{F}$ is a deformation of $\mathcal{F}_0^{\oplus\alpha}$  over $\mathbb{A}^1$.
\end{lemma}

\begin{proof}
Let $\lambda:\mathbb{A}^1_s\times\mathbb{D}_\alpha\to\mathbb{D}_\alpha$ be the map defined by its pullback $\lambda^\#:t\mapsto ts$, and also denote by $\lambda$ its base change to $M$. We claim that the locally free sheaf $\lambda^*\mathcal{F}$ gives the required deformation. Indeed, the restriction of $\lambda^*\mathcal{F}$ to $\{s_0\}\in\mathbb{A}^1_s$ is the pullback of $\mathcal{F}$ along the composition $b_{s_0}=\lambda\circ j_{s_0}$
$$
\mathbb{D}_\alpha\times M\xhookrightarrow{j_{s_0}}\mathbb{D}_\alpha\times\mathbb{A}^1_s\times M\xrightarrow{\lambda}\mathbb{D}_\alpha\times M
$$
determined by its pullback $b_{s_0}^\#:t\mapsto s_0t$. When $s_0\neq 0$, $b_{s_0}^*\mathcal{F}\simeq\mathcal{F}$. On the other hand, when $s_0=0$, the map $b_0$ factors as the composition
$$
\mathbb{D}_\alpha\times M\xrightarrow{\rho} M\xrightarrow{\imath}\mathbb{D}_\alpha\times M
$$
so $b_0^*\mathcal{F}=\rho^*\imath^*\mathcal{F}=\rho^*\mathcal{F}_0$, as desired. The last statement follows from projection formula and the fact that $\rho_*\rho^*\mathcal{O}_M\simeq\mathcal{O}_M^{\oplus\alpha}$.
\end{proof}

Suppose $D=\alpha x$ is a fat point, i.e.~a divisor given by a single point $x$ with multiplicity $\alpha$, and let $t$ be a local parameter on $C$ at $x$. 
Note that the notation $\cO_D$ is unfortunately ambiguous, because it can denote both the structure sheaf of the subscheme $D\subset C$ and the skyscraper sheaf of the point $\{D\}\in\Sym^\alpha C$.
When confusion is possible, we denote the latter sheaf by 
$\cO_{\{D\}}$.
Then
\begin{align}\label{the ring B alpha}
\tau^*\mathcal{O}_{\{D\}}\simeq\frac{\mathbb{C}[t_1,\ldots,t_\alpha]}{(\sigma_1,\ldots,\sigma_\alpha)}
\end{align}
is the so-called {\em covariant algebra},
where $\sigma_1,\ldots,\sigma_\alpha$ are the elementary symmetric functions in variables $t_j=\pi_j^*(t)$. Call $\mathbb{B}_\alpha=\spec \tau^*\mathcal{O}_{\{D\}}$. By the Newton formulas, $t_j^\alpha=0$ for every $j=1,\ldots,\alpha$, and in particular, every map $\pi_j:\,\mathbb{B}_\alpha\to C$ factors through~$\mathbb{D}_\alpha$. By abuse of notation, we have a diagram of morphisms
\begin{equation}\label{diagram tau}
\begin{tikzcd}
&\mathbb{B}_\alpha\times M \arrow[r,"\pi_j"]\arrow{dr}[swap]{\tau} &\mathbb{D}_\alpha\times M \arrow[r,"q"]\arrow[d,bend left=20,"\rho"] & C\times M \arrow[dl] \\
& & M\arrow[u,bend left = 20,"\imath"] &
\end{tikzcd}
\end{equation}
%From Lemma \ref{deformation lemma}, we obtain the following.

\begin{corollary}\label{standarddeformation}
Let $D=x_1+\ldots+x_\alpha$ (possibly with repetitions). Then both $\left(\mathcal{F}^{\boxtimes\alpha}\right)_D$ and $\left(\overline{\mathcal{F}}^{\boxtimes\alpha}\right)_D$ 
%(resp., $\left(\mathcal{F}^{\boxplus\alpha}\right)_D$ and $\left(\overline{\mathcal{F}}^{\boxplus\alpha}\right)_D$)
are deformations of $\mathcal{F}_{x_1}\otimes\ldots\otimes\mathcal{F}_{x_\alpha}$
%(resp.,$\bigoplus\mathcal{F}_{x_k}^{\oplus\alpha_k}$)
over $\mathbb A^1$.
\end{corollary}

\begin{proof}
%We only prove the statement for tensor bundles, since this is the only fact we use. The proof for other bundles is entirely analogous and can be obtained by changing $\otimes$ to $\oplus$ everywhere.
By (\ref{restriction to different points}), it suffices to consider the case when $D=\alpha x$. Using the notation as in the diagram (\ref{diagram tau}), the restriction $\left(\mathcal{F}^{\boxtimes\alpha}\right)_D$ can be written as $\tau_*^{S_\alpha}\left(\bigotimes\pi_j^*q^*\mathcal{F}\right)$, by flatness of $\tau$. The construction of Lemma \ref{deformation lemma} commutes with the $S_\alpha$-action, 
so $\tau_*^{S_\alpha}\left(\bigotimes\pi_j^*q^*\mathcal{F}\right)$
is a deformation of $\tau_*^{S_\alpha}\left(\bigotimes\pi_j^*\rho^*\mathcal{F}_x\right)$ over $\bA^1$, since $(q^*\mathcal{F})_0=\mathcal{F}_x=\left.\mathcal{F}\right|_{\{x\}\times M}$. Note $\pi_j^*\rho^*=\tau^*$, so using the projection formula, we get that $\left(\mathcal{F}^{\boxtimes\alpha}\right)_D$ is a deformation of $\left(\bigotimes_{j=1}^\alpha\mathcal{F}_x\right)\otimes\tau_*^{S_\alpha}(\mathcal{O}_{\mathbb{B}_\alpha\times M})$, and similarly, $\left(\overline{\mathcal{F}}^{\boxtimes\alpha}\right)_D$ is a deformation of $\left(\bigotimes_{j=1}^\alpha\mathcal{F}_x\right)\otimes\tau_*^{S_\alpha}(\mathcal{O}_{\mathbb{B}_\alpha\times M}\otimes\sign)$. By flatness of the quotient 
$C^\alpha\to \Sym^\alpha C$, the covariant algebra $\mathcal{O}_{\mathbb{B}_\alpha}$ (\ref{the ring B alpha}) is  the regular representation $\bC[S_\alpha]$ of $S_\alpha$. It follows that it contains the trivial and the sign representations each with multiplicity $1$, and therefore $\tau_*^{S_\alpha}(\mathcal{O}_{\mathbb{B}_\alpha\times M})=\tau_*^{S_\alpha}(\mathcal{O}_{\mathbb{B}_\alpha\times M}\otimes\sign)=\mathcal{O}_M$. This concludes the proof.
\end{proof}

\begin{remark}\label{remark equivariantdeformation}
If we have a $G$-action on $M$ and a $G$-equivariant bundle $\cF$, then the deformations constructed in the proofs of Lemma \ref{deformation lemma} and Corollary \ref{standarddeformation} are also $G$-equivariant, i.e. given by a $G$-equivariant bundle on $\bA^1\times M$. This is because the
map $\lambda: \bA^1_s \times \bD_\alpha \times M \to \bD_\alpha \times M$ is given by the identity on the factor $M$, hence $\lambda$
is $G$-invariant. Thus, the pull-back $\lambda^*\cF$ of a $G$-equivariant sheaf is naturally again a $G$-equivariant sheaf.
\end{remark}

\begin{definition}\label{stable deformation}
    A vector bundle $\mathcal{F}$ on a scheme
     $M$ is said to be a {\em stable deformation} of a vector bundle $\mathcal{G}$ \emph{over $\mathbb A^1$} if there is some vector bundle $\mathcal{K}$ such that $\mathcal{F}\oplus\mathcal{K}$ is a deformation of a direct sum $\mathcal{G}^{\oplus r}$ for some $r>0$.
\end{definition}

\begin{proposition}\label{proposition stable deformation}
Let $D=x+\tilde D$. Then the vector bundle  $\left(\mathcal{F}^{\boxtimes\alpha}\right)_D$ is a stable deformation of the 
vector bundle
$\mathcal{F}_x\otimes \left(\mathcal{F}^{\boxtimes(\alpha-1)}\right)_{\tilde D}$ over $\bA^1$.\end{proposition}

\begin{proof}
By Lemma~\ref{differentpoints}, it suffices to consider the case $D=\alpha x$. Let $W_\alpha=\mathbb{C}^\alpha$ be the tautological representation of $S_\alpha$, which splits as a sum of the trivial and the standard representations, $W_\alpha=\mathbb{C}\oplus V_\alpha$. For any $S_\alpha$-equivariant vector bundle $\mathcal{E}$ on $\mathbb{B}_\alpha\times M$, we  have
\begin{align}\label{rep E splits}
\tau_*^{S_\alpha}(\mathcal{E}\otimes W_\alpha)=\tau_*^{S_\alpha}(\mathcal{E})\oplus \tau_*^{S_\alpha}(\mathcal{E}\otimes V_\alpha).
\end{align}
On the other hand, we have $W_\alpha=\mathbb{C}[S_\alpha/S_{\alpha-1}]$, where $S_{\alpha-1}\hookrightarrow S_\alpha$ is the inclusion given by fixing the $\alpha$-th element. Then, by Frobenius reciprocity, $\tau_*^{S_\alpha}(\mathcal{E}\otimes W_\alpha)=\tau_*^{S_{\alpha-1}}(\mathcal{E})=\rho_*\circ(\pi_\alpha)^{{S_{\alpha-1}}}_*(\mathcal{E})$, where $\pi_\alpha$ is the $\alpha$-th projection. By Lemma \ref{deformation lemma} this bundle is a deformation of $\left((\pi_\alpha)^{S_{\alpha-1}}_*\mathcal{E}\right)_0^{\oplus\alpha}$ over $\bA^1$. Now let $\mathcal{E}$ be $\bigotimes\pi_j^*q^*\mathcal{F}$. Then $\tau_*^{S_\alpha}(\mathcal{E})$ is precisely $\left(\mathcal{F}^{\boxtimes\alpha}\right)_D$ and, by projection formula,
\begin{align*}
    \left((\pi_\alpha)^{S_{\alpha-1}}_*\mathcal{E}\right)_0&=\mathcal{F}_x\otimes\left((\pi_\alpha)_*^{S_{\alpha-1}}\left(\bigotimes_{j=1}^{\alpha-1}\pi_j^*q^*\mathcal{F}\right)\right)_0\\
    &=\mathcal{F}_x\otimes(\pi_\alpha)_*^{S_{\alpha-1}}\bigotimes_{j=1}^{\alpha-1}(\pi_j^*q^*\mathcal{F})|_{t_\alpha=0}=\mathcal{F}_x\otimes \left(\mathcal{F}^{\boxtimes(\alpha-1)}\right)_{(\alpha-1)x}
\end{align*}
since the subscheme $(t_\alpha=0)\subset \mathbb{B}_\alpha$ is isomorphic to $\mathbb{B}_{\alpha-1}$ and the restriction of $\pi_\alpha$ to it is isomorphic to the quotient $\tau$ (for the group $S_{\alpha-1}$).
\end{proof}

\begin{remark}\label{semi-continuityremark}
We will use stable deformations for semi-continuity arguments.
If~$\mathcal{F}$ is a stable deformation of $\mathcal{G}$, $M$ is proper and $H^p(\mathcal{G})=0$, then, by the semi-continuity theorem, $H^p(\mathcal{F})=0$, too. In particular, if $\mathcal{G}$ is $\Gamma$-acyclic, then so is~$\mathcal{F}$.
\end{remark}

\begin{remark}\label{eulerdef} Let $D=x_1+\tilde D$, $\tilde D=x_2+\ldots+x_\alpha$ (possibly with repetitions). Suppose $M$ is proper.
Since $\left(\mathcal{F}^{\boxtimes\alpha}\right)_D$ and 
$\mathcal{F}_{x_1}\otimes \left(\mathcal{F}^{\boxtimes(\alpha-1)}\right)_{\tilde D}$ are both deformations of $\mathcal{F}_{x_1}\otimes\ldots\otimes\mathcal{F}_{x_\alpha}$ over $\bA^1$ by Corollary~\ref{standarddeformation},
they have the same Euler characteristic. Combining this with Remark~\ref{semi-continuityremark}, if $H^p\left(\mathcal{F}_x\otimes \left(\mathcal{F}^{\boxtimes(\alpha-1)}\right)_{\tilde D}\right)=0$ for $p>0$ then  both
$H^p\left(\left(\mathcal{F}^{\boxtimes\alpha}\right)_D\right)=0$ for $p>0$ and 
$H^0\left(\left(\mathcal{F}^{\boxtimes\alpha}\right)_D\right)=H^0\left(\mathcal{F}_x\otimes \left(\mathcal{F}^{\boxtimes(\alpha-1)}\right)_{\tilde D}\right)$. The same results hold for $\left(\overline{\mathcal{F}}^{\boxtimes\alpha}\right)_D$ and $\mathcal{F}_x\otimes
\left(\overline{\mathcal{F}}^{\boxtimes(\alpha-1)}\right)_{\tilde D}$.
\end{remark}

\section{Wall-crossing on moduli spaces of stable pairs}\label{ThaddeusSpaces}

Let $C$ be a smooth projective curve of genus $g\geq 2$ over $\mathbb{C}$. In \cite{thaddeus}, Thaddeus studies moduli spaces of pairs $(E,\phi)$, where $E$ is a rank-two vector bundle on $C$ with fixed determinant line bundle $\Lambda$ and $\phi\in H^0(E)$ is a non-zero section.
We use these results extensively and so, for ease of reference, try to follow the notation in \cite{thaddeus} as closely as possible. 
We always assume that $d=\deg E>0$.
For a given choice of a parameter $\sigma\in\mathbb{Q}$ the following stability condition is imposed: for every line subbundle $L\subset E$, one must have
%\begin{align*}
%\deg L \leq  \frac{d}{2} -\sigma\quad & %\text{ if }\phi\in H^0(L),\\
%\deg L \leq  \frac{d}{2} +\sigma \quad %& \text{ if }\phi\notin H^0(L).
%\end{align*}
$$
\deg L \leq  
\begin{cases}
\frac{d}{2} -\sigma\quad & \text{ if }\phi\in H^0(L),\\
\frac{d}{2} +\sigma \quad & \text{ if }\phi\notin H^0(L).
\end{cases}
$$
Throughout the text, we work with the general assumption  $\sigma\in (0,d/2]$, which guarantees the existence of stable pairs, see \cite[1.3]{thaddeus}. The next lemma follows the ideas of \cite[2.1]{thaddeus}:

\begin{lemma}\label{smooth_stack}
For a given line bundle $\Lambda$ of degree $d$,
the moduli stack $\cM_\sigma(\Lambda)$ of semi-stable pairs is a smooth algebraic stack.
\end{lemma}

\begin{proof}
$\cM_\sigma (\Lambda)$ is a fiber of the morphism
$\cM_\sigma^d\to\Pic^d(C)$, $(E,\phi)\mapsto\det E$, from the stack of semi-stable pairs $(E,\phi)$, where $E$ is a degree $d$ vector bundle.
We first  show that $\cM_\sigma^d$
is smooth. Obstructions to deformations of a morphism of sheaves $\phi$ from a fixed source $\cO_C$ to a varying target $E$ lie in $\Ext^1([\cO_C\mathop{\to}\limits^\phi E], E)$.
The~truncation exact triangle of the complex $[\cO_C\mathop{\to}\limits^\phi E]$ yields an exact sequence
$$
\Ext^1(E,E)\mathop{\to}\limits^\phi \Ext^1(\cO_C,E)\to\Ext^1([\cO_C\mathop{\to}\limits^\phi E], E)\to0.
$$
We claim that the first map is surjective, so obstructions vanish. By~Serre duality, it suffices to prove injectivity of the map of sheaves $E^*(K_C)\mathop{\to}\limits^\phi 
E^*\otimes E(K_C)$ and this follows from $\phi\ne0$ (cf. the proof of \cite[2.1]{thaddeus}). Next we consider obstructions to deformations of $(E,\phi)$ fixing the determinant, which  amounts to studying the map
$\Ext^1(E,E)_0\mathop{\to}\limits^\phi \Ext^1(\cO_C,E)$, where 
$\Ext^1(E,E)_0$ denotes 
traceless endomorphisms. However, this map 
is also surjective because the Serre-dual map is induced by the map of sheaves
$E^*(K_C)\mathop{\to}\limits^\phi 
\cE\! nd(E)_0(K_C)$, where $\cE\! nd(E)_0$ is identified with the quotient of $\cE\! nd(E)$ by the subspace of scalar multiples of the identity. This map is still injective, as a non-zero scalar multiple of the identity cannot have rank~$1$.
\end{proof}

The moduli space $M_\sigma(\Lambda)$ of $S$-equivalence classes of stable pairs exists as a projective variety and, in the case there is no strictly semi-stable locus, it is smooth, isomorphic to the stack $\cM_\sigma(\Lambda)$ and carries a universal bundle $F$ with a universal section $\tilde{\phi}:\cO_{C\times M_\sigma (\Lambda)}\to F$. A salient point is that stable  pairs, unlike stable vector bundles, don't have any automorphisms besides the identity \cite[1.6]{thaddeus}. Note that non-trivial multiples of the identity are not automorphisms, as they do not preserve the section $\phi$.

The spaces $M_\sigma(\Lambda)$ can be obtained as GIT quotients as follows (see \cite[\S1]{thaddeus} for further details). Let $\chi=\chi(E)=d+2-2g$. For $d\gg 0$, every bundle $E$ of rank $2$ and $\det E=\Lambda$ is generated by global sections, and $\chi=h^0(E)$. Then $M_\sigma(\Lambda)$ is a GIT quotient of $U\times \mathbb{PC}^\chi$ by $SL_\chi$, where $U\subset \Quot$ is the locally closed subscheme of the  Quot scheme \cite{groth} corresponding to locally free quotients $\mathcal{O}_C^\chi\twoheadrightarrow E$ inducing an isomorphism $s:\mathbb{C}^\chi\xrightarrow{\sim}H^0(E)$ and such that $\wedge^2 E=\Lambda$. The isomorphism $s$ induces a map $\wedge^2\mathbb{C}^\chi\to H^0(\Lambda)$, and we get an inclusion $U\times\mathbb{PC}^\chi\hookrightarrow \mathbb{P}\Hom\times\mathbb{PC}^\chi$, where we write $\mathbb{P}\Hom$ for $\mathbb{P}\Hom(\wedge^2\mathbb{C}^\chi,H^0(\Lambda))$, and a quotient $s:\mathcal{O}_C^\chi\twoheadrightarrow E$ on the left is sent to the induced map in the first coordinate. Then $M_\sigma(\Lambda)$ can be seen as the GIT quotient of a closed subset of $\mathbb{P}\Hom\times\mathbb{PC}^\chi$ by $SL_\chi$, where the linearization is given by $\mathcal{O}(\chi+2\sigma, 4\sigma)$.

For arbitrary $d$, we pick any effective divisor $D$ on $C$ with $\deg D \gg 0$, and $M_\sigma(\Lambda)$ can be seen as the closed subset of $M_\sigma(\Lambda(2D))$ consisting of pairs $(E,\phi)$ such that $\phi|_D=0$. This way, $M_\sigma(\Lambda)$ is a GIT quotient by $SL_{\chi'}$, with $\chi'=d+2-2g+2\deg D$, of the closed subset $X\subset U'\times \mathbb{PC}^{\chi'}$ determined by the condition that $\phi$ vanishes along $D$ \cite[1.9 \& 1.20]{thaddeus}.
Regardless of the GIT, 
the embedding 
$M_\sigma(\Lambda)\subset
M_\sigma(\Lambda(2D))
$
will play an important role in our induction arguments.

\begin{remark}\label{SL vs PGL}
Scalar matrices in $SL_{\chi'}$ act trivially on $U\times\mathbb{PC}^{\chi'}$, so the action factors through the quotient $SL_{\chi'}\to PGL_{\chi'}$. If we replace $\mathcal{O}(\chi'+2\sigma, 4\sigma)$ by its $\chi'$-th power, this line bundle carries a $PGL_{\chi'}$-linearization and $M_\sigma(\Lambda)$ can also be written as a GIT quotient 
$X\git PGL_{\chi'}$. Moreover, the moduli stack $\cM_\sigma(\Lambda)$
is isomorphic to the corresponding GIT quotient stack~$[X^{ss}/PGL_{\chi'}]$.
\end{remark}

%Thaddeus also describes how $M_\sigma(\Lambda)$ varies with $\sigma$. 
For fixed $\Lambda$ but varying $\sigma$, the spaces $M_\sigma(\Lambda)$ are all GIT quotients of the same scheme, with different stability conditions. The GIT walls occur when $\sigma\in d/2+\mathbb{Z}$, and for $0\leq i\leq v=\lfloor(d-1)/2\rfloor$ we have different GIT chambers with moduli spaces $M_0,M_1,\ldots,M_{v}$, where $M_i=M_i(\Lambda)=M_\sigma(\Lambda)$ for $\sigma \in (\max(0,d/2-i-1),d/2-i)$. These $M_i$ are smooth projective rational varieties of dimension $d+g-2$, see \cite[2.2 \& 3.6]{thaddeus}. Indeed, $M_0=\mathbb{P}H^1(C, \Lambda^{-1})$ is a projective space, $M_1$ is a blow-up of $M_0$ along a copy of $C$ embedded by the complete linear system of $\omega_C \otimes \Lambda$, and the remaining ones are small modifications of $M_1$.
More precisely, for each $0\leq i \leq v= \lfloor(d-1)/2\rfloor$ there are projective bundles $\mathbb{P}W_{i}^+$ and $\mathbb{P}W_{i}^-$ over the symmetric product $\Sym^iC$, of (projective) ranks $d+g-2i-2$, $i-1$, respectively, with embeddings $\mathbb{P}W_{i}^+\hookrightarrow M_i$ and $\mathbb{P}W_{i}^-\hookrightarrow M_{i-1}$, and such that $\mathbb{P}W_i^+$ parametrizes the pairs $(E,\phi)$ appearing in $M_i$ but not in $M_{i-1}$, while $\mathbb{P}W_i^-$ parametrizes those appearing in $M_{i-1}$ but not in $M_i$. 

We have a diagram of flips \eqref{diagram Mi},
where $\tilde{M}_{i}$ is the blow-up of $M_{i-1}$ along $\mathbb{P}W_i^-$ and also the blow-up of $M_i$ along~$\mathbb{P}W_i^+$.
Here $N$ is the moduli space of ordinary slope-semistable vector bundles as in the Introduction and the map $M_{v}\to N$ is an ``Abel-Jacobi'' map with fiber $\mathbb{P}H^0(C,E)$ over a vector bundle $E$. If $d \geq 2g-1$ the Abel-Jacobi map is surjective, and if $d=2g-1$ it is a birational morphism (see \cite[\S 3]{thaddeus} for details). %Our strategy for proving Theorem~\ref{maintheorem} will be to construct the decomposition and to prove its semi-orthogonality step-by-step moving from left to right in the diagram of flips.

\begin{equation}\label{diagram Mi}
\begin{tikzcd}[cramped,sep=scriptsize]
& \tilde{M}_2\arrow[dl]\arrow[dr] && \tilde{M}_3\arrow[dl]\arrow[dr] && \tilde{M}_v\arrow[dl]\arrow[dr] & \\
M_1\arrow[d] & & M_2 & & \cdots & & M_v\arrow[d]\\
M_0 &&  &&  &&  N
\end{tikzcd}
\end{equation}

\begin{notation}
By abuse of notation, we will sometimes write $M_i(d)$ to denote the moduli space $M_i=M_i(\Lambda)$, where $d=\deg \Lambda$. 
\end{notation}

\begin{notation}
In what follows, $v$ will always denote $\lfloor (d-1)/2\rfloor$.
\end{notation}

The Picard group of $M_1=\Bl_CM_0$ is generated by a hyperplane section $H$ in $M_0=\mathbb{P}^{d+g-2}$ and the exceptional divisor $E_1$ of the morphism $M_1\to M_0$. Since the maps $M_i\dashrightarrow M_{i+1}$ are small birational modifications for each $i\geq 1$, there are natural isomorphisms $\Pic M_1\simeq\Pic M_i$, $i\geq 1$. The following notation is taken from \cite[\S 5]{thaddeus}.

\begin{definition}\label{notation O(m,n)}
For each $m$, $n$, we denote the line bundle $\mathcal{O}_{M_1}((m+n)H-nE_1)$ by $\mathcal{O}_{1}(m,n)$, while $\mathcal{O}_i(m,n)$ will denote the image of $\mathcal{O}_{M_1}(m,n)$ under the isomorphism $\Pic M_1\simeq\Pic M_i$.
\end{definition}

\begin{remark}\label{ample cone on Mi}
By \cite[5.3]{thaddeus}, the ample cone of $M_i$ is  bounded by $\cO_i(1,i-1)$ and $\cO_i(1,i)$ for $0<i<v$, while the ample cone of $M_v$ is bounded below by $\cO_v(1,v-1)$ and  contains the cone bounded on the other side by $\cO_v(2,d-2)$. In other words, the ray bounding the cone above has slope at least $(d-2)/2$.
\end{remark}

\begin{remark}\label{restriction of O(m,n)}
For any effective divisor $D$ on $C$ of $\deg D=\alpha$, we have a closed immersion $M_{i-\alpha}(\Lambda (-2D))\hookrightarrow M_i(\Lambda)$,
as the locus of pairs $(E,\phi)$ where the section $\phi$ vanishes along $D$ \cite[1.9]{thaddeus}. The~restriction of $\cO_i (m,n)$ to $M_{i-\alpha}(\Lambda(-2D))$ is $\cO_{i-\alpha}(m,n-m\alpha)$ \cite[5.7]{thaddeus}. If $i-\alpha=0$, the restriction of $\cO_i(m,n)$ to $M_0(\Lambda(-2D))=\bP^r$ is $\cO_{\bP^r}(n+m(1-i))$. This follows from \cite[7.5]{thaddeus} together with the fact that, for an embedding $\bP^{r}=M_0(\Lambda(-2x))\hookrightarrow M_1(\Lambda)$, $\cO_{M_1}(E_1)$ restricts to $\cO_{\bP^r}(-1)$ while $\cO_{M_1}(H)$ restricts to $\cO_{\bP^r}$.
\end{remark}

Suppose $d\gg 0$. Then %%for $\sigma\notin d/2+\mathbb{Z}$ 
the universal bundle $F$ on $M_i\times C$ is the descent from the equivariant vector bundle $\mathcal{F}(1)$ on $X\times C\subset U\times\mathbb{PC}^\chi\times C$, where $\mathcal{O}^\chi\twoheadrightarrow \mathcal{F}$ is the universal quotient bundle over $U\times C$, and the universal section $\tilde{\phi}$ descends from the universal section of~$\mathcal{F}(1)$ \cite[1.19]{thaddeus}. Let $\pi:C\times M_i\to M_i$ be the projection. For every $i\geq 1$, the determinant of cohomology line bundle $\det\pi_!F$ (cf. \cite{det-coh}) descends from $\mathcal{O}(0,\chi)$ on $\mathbb{P}\Hom\times\mathbb{PC}^\chi$ \cite[5.4 \& proof of 5.14]{thaddeus}. On $M_1$, $\det\pi_!F$ corresponds to $\mathcal{O}_{M_1}\bigl((g-d-1)H-(g-d)E_1\bigr)
=\mathcal{O}_1(-1,g-d)$. For $x\in C$, call $F_x=F|_{\{x\}\times M}$. The line bundle $\det F_x=\wedge^2F_x$ does not depend on $x$, and it is the descent of $\mathcal{O}(1,2)$  on $\mathbb{P}\Hom\times\mathbb{PC}^\chi$. It corresponds to $\mathcal{O}_{M_1}(E_1-H)=\mathcal{O}_i(0,-1)$ \cite[5.4 \& proof of 5.14]{thaddeus}.

For arbitrary $d$, consider an embedding $\imath:M_i\hookrightarrow M'= M_\sigma (\Lambda (2D))$, $\deg D\gg 0$, as above, and let $F'$ be the universal bundle on $M'$. Then we have a short exact sequence \cite[1.20]{thaddeus}
\begin{align}\label{usefulseq}
0\to F\to\imath^*F'\to \imath^*F'|_{D\times M_{i}}\to 0.
\end{align}
In particular, $F$ is the descent from an object on $X\times C\subset U'\times \mathbb{PC}^{\chi'}\times C$. The same is true for $\det\pi_!F$ and $\wedge^2F_x$.

\begin{lemma}\label{usefuliso}
$F_x\simeq\imath^*F'_x$ for every $x\in C$.
\end{lemma}

\begin{proof}
We tensor \eqref{usefulseq} with $\cO_{\{x\}\times M_i}$, which gives an exact sequence
$$0\to 
{\mathcal Tor}^1_{C\times M_i}(\imath^*F'|_{D\times M_{i}},\cO_{\{x\}\times M_i})\to
F_x\to\imath^*F'_x\to\qquad\qquad\qquad$$
$$\qquad\qquad\qquad\to \imath^*F'|_{D\times M_{i}}\otimes_{C\times M_i}\cO_{\{x\}\times M_i}\to 0.
$$
If $x\not\in D$ then 
${\mathcal Tor}^1(\imath^*F'|_{D\times M_{i}},\cO_{\{x\}\times M_i})=\imath^*F'|_{D\times M_{i}}\otimes\cO_{\{x\}\times M_i}=0$ and we get $F_x\simeq\imath^*F'_x$. If $x\in D$ then
${\mathcal Tor}^1_C(\cO_D,\cO_x)\simeq
\cO_D\otimes_C\cO_x\simeq\cO_x$, and the sequence splits into two isomorphisms, 
$\imath^*F'_x\simeq F_x$ and
$\imath^*F'_x\simeq\imath^*F'_x$.
\end{proof}

\begin{lemma}\label{Fx=O+O(-1)}
    On $M_0=\bP^r$, $F_x\simeq\cO_{\bP^r}\oplus\cO_{\bP^r}(-1)$.
\end{lemma}

\begin{proof}
    In fact, $F_x$ is a rank-two bundle on $\bP^r$, carrying a nowhere vanishing section, and with determinant $\cO_{\bP^r}(-1)$. Hence, $F_x$ must be isomorphic to $\cO_{\bP^r}\oplus\cO_{\bP^r}(-1)$.
\end{proof}

%%\begin{definition}\label{determinant of cohomology} If $V$ is a vector bundle over $Y\times T$ and $\pi:Y\times T\to T$ is the projection, the determinant of cohomology of $V$ is defined as the line bundle $\det {\pi_!}V$. It can be shown that this definition extends to a morphism from the $K$-group, $K(Y\times T)\to \Pic Y$, so that if $0\to V'\to V\to V''\to 0$ is a short exact sequence of sheaves, then $\det {\pi_!}V\cong\det \pi_!V'\otimes\det\pi_!V''$ \cite{det-coh}. \end{definition}

\begin{definition}\label{notation important line bundles}
We introduce notation for some important line bundles:
$$\psi^{-1}:=\det\pi_!F=\mathcal{O}_i(-1,g-d),$$
$$\Lambda_M:=\wedge^2F_x=\mathcal{O}_i(0,-1),$$
$$\zeta:=\psi\otimes\Lambda_M^{d-2g+1}=\cO_i(1,g-1)
$$
and
$$\theta:=\psi^2\otimes\Lambda_M^\chi=\mathcal{O}_i(2,d-2),$$
where $\chi=d+2-2g$ (cf. \cite[Proposition 2.1]{narasimhan1}).
\end{definition}

\begin{lemma}\label{BasicKoszul}
For a point $x\in C$ and every $i\ge1$, we have 
exact sequences
\begin{align}\label{reduction F dual}
0\to \Lambda_M^{-1} \to F_x^\vee \to \mathcal{O}_{M_i(\Lambda)} \to \mathcal{O}_{M_{i-1}(\Lambda(-2x))} \to 0
\end{align}
and 
\begin{align}\label{reduction F non-dual}
    0\to\mathcal{O}_{M_i(\Lambda)}\to F_x\to \Lambda_M\to \left.\Lambda_M\right|_{M_{i-1}(\Lambda(-2x))}\to 0.
\end{align}
\end{lemma}

\begin{proof}
By Remark~\ref{restriction of O(m,n)}, the zero locus of the section $\phi_x$ of $F_x$ is smooth and has codimension~$2$.
Therefore, the Koszul complex and the dual Koszul complex of $(F_x,\phi_x)$
are exact.
\end{proof}

\begin{definition}
Let $M=M_i(\Lambda)$ be a moduli space in the interior of a GIT chamber, as above, and let $F$ be the universal bundle on $C\times M$. 
We apply the constructions of Section~\ref{TensorBundles} to $F$. In particular, for a divisor $D\in\Sym^\alpha C$, we will denote
$$
%    F_D&=\left(F^{\boxplus\alpha}\right)_D 
%    &F_D^\vee&=\left(\overline{F^\vee}^{\boxplus\alpha}\right)_D
%    \\
    G_D=\left(F^{\boxtimes\alpha}\right)_D\quad\hbox{\rm and}\quad
    \overline{G}_D=\left(\overline{F}^{\boxtimes\alpha}\right)_D.
$$
We write $G_D^\vee$, $\overline{G}_D^\vee$ for their respective duals.
\end{definition}

\begin{lemma}\label{braid lemma}
    We have the following formulas:
    \begin{align*}
        (F^\vee)^{\boxtimes\alpha} \simeq \bigl((\Lambda^{\vee})^{\boxtimes\alpha}\boxtimes\Lambda_M^{-\alpha}\bigr)\otimes F^{\boxtimes\alpha},\\        G_D^\vee\simeq\left(\overline{F^\vee}^{\boxtimes\alpha}\right)_D,\quad \overline{G}_D^\vee\simeq \left((F^\vee)^{\boxtimes\alpha}\right)_D.
%        (F^\vee)^{\boxtimes\alpha} \simeq(F^{\boxtimes\alpha})^\vee,\quad\text{and}\quad
%        \overline{F}^{\boxtimes\alpha} \simeq F^{\boxtimes\alpha}(-\Delta/2)
    \end{align*}
%    where $\cO_{\Sym^\alpha C}(-\Delta/2)=\tau_*^{S_\alpha}(\cO_{C^\alpha}\otimes\sign)$ is a line bundle on $\Sym^\alpha C$ such that $\cO_{\Sym^\alpha C}(-\Delta/2)^{\otimes 2}\simeq\cO_{\Sym^\alpha C}(-\Delta)$, where $\Delta\subset\Sym^\alpha C$ is the diagonal divisor.
\end{lemma}

\begin{proof}
Let us denote
$$
{\widehat{\Lambda^\vee}}^{\boxtimes\alpha}:=\bigotimes_{j=1}^\alpha\pi_j^*(\Lambda^\vee),\qquad\qquad \widehat{F}^{\boxtimes\alpha}:=\bigotimes_{j=1}^\alpha\pi_j^*F,
$$
which are bundles on $C^\alpha$ and $C^\alpha\times M$, respectively. By \cite[Theorem 2.3]{kempfs}, $(\Lambda^\vee)^{\boxtimes\alpha}$ is the descent of $\widehat{\Lambda^\vee}^{\boxtimes\alpha}$, so we have
$$
(F^\vee)^{\boxtimes\alpha}=\tau_*^{S_\alpha}\left(\bigl(\widehat{\Lambda^\vee}^{\boxtimes\alpha}\boxtimes\Lambda_M^{-\alpha}\bigr)\otimes \widehat{F}^{\boxtimes\alpha}\right)\simeq\bigl((\Lambda^\vee)^{\boxtimes\alpha}\boxtimes\Lambda_M^{-\alpha}\bigr)\otimes \tau_*^{S_\alpha}\left(\widehat{F}^{\boxtimes\alpha}\right).
$$
The latter expression is precisely $\bigl((\Lambda^{\vee})^{\boxtimes\alpha}\boxtimes\Lambda_M^{-\alpha}\bigr)\otimes F^{\boxtimes\alpha}$.

We write $\cO_{\Sym^\alpha C}(-\Delta/2):=\tau_*^{S_\alpha}(\cO_{C^\alpha}\otimes\sign)$, a line bundle on $\Sym^\alpha C$ such that $\cO_{\Sym^\alpha C}(-\Delta/2)^{\otimes 2}\simeq\cO_{\Sym^\alpha C}(-\Delta)$, where $\Delta\subset\Sym^\alpha C$ is the diagonal divisor.
The morphism $\tau$ is ramified along $B=\tau^{-1}(\Delta)$ generically of order $2$, so $\cO_{C^\alpha}(B)$ is a relative dualizing sheaf for $\tau$. The equivariant structure on $\cO_{C^\alpha}(B)$ is dual to the equivariant structure of the ideal sheaf $\cO_{C^\alpha}(-B)\subset\cO_{C^\alpha}$. Since the local equation of $B$ is anti-invariant, $\cO_{C^\alpha}(B)\simeq \tau^*\cO_{\Sym^\alpha C}(\Delta/2)\otimes\sign.$ 

\begin{comment}
On $C^\alpha$, we have an $S_\alpha$-equivariant short exact sequence
$$
0\to\widehat{F}^{\boxtimes\alpha}\otimes\sign\to\widehat{F}^{\boxtimes\alpha}(B)\otimes\sign\to\widehat{F}^{\boxtimes\alpha}(B)|_B\otimes\sign\to 0.
$$
Computing $\tau_*^{S_\alpha}$ gives a short exact sequence
$$
0\to \overline{F}^{\boxtimes\alpha}\to F^{\boxtimes\alpha}(\Delta/2)\to F^{\boxtimes\alpha}(\Delta/2)|_\Delta\to 0.
$$
This shows that $\overline{F}^{\boxtimes\alpha}\simeq F^{\boxtimes\alpha}(-\Delta/2)$. 
\end{comment}

By duality,
\begin{align*}
\left((F^\vee)^{\boxtimes\alpha}\right)^{\vee}&\simeq \tau_*^{S_\alpha}\left(\widehat{F}^{\boxtimes\alpha}(B)\right)\simeq \tau_*^{S_\alpha}\left(\widehat{F}^{\boxtimes\alpha}\otimes\sign\right)(\Delta/2)\simeq\overline{F}^{\boxtimes\alpha}(\Delta/2)
%\simeq F^{\boxtimes\alpha}.
\end{align*}

Restrictig to a divisor $D\in \Sym^\alpha C$, we obtain
$$
\left((F^\vee)^{\boxtimes\alpha}\right)_D^{\vee}\simeq\left(\overline{F}^{\boxtimes\alpha}\right)_D
$$
and similarly, arguing with $F^\vee$ in place of $F$, we get
$$
\left(F^{\boxtimes\alpha}\right)_D^{\vee}\simeq\left(\overline{F^\vee}^{\boxtimes\alpha}\right)_D.
$$
This completes the proof.
\end{proof}

\begin{corollary}\label{relation G bar and G}
We have $G_D^\vee\simeq \overline{G}_D\otimes\Lambda_M^{-\deg D}$ and $G_D\simeq \overline{G}_D^\vee\otimes\Lambda_M^{\deg D}$.
\end{corollary}

\begin{proof}
This follows from restricting $(F^\vee)^{\boxtimes\alpha} \simeq \bigl((\Lambda^{\vee})^{\boxtimes\alpha}\boxtimes\Lambda_M^{-\alpha}\bigr)\otimes F^{\boxtimes\alpha}$ to $\{D\}\times M$.
\end{proof}

Consider again the diagram (\ref{diagram Mi}). The wall between two consecutive chambers $M_{i-1}$ and $M_i$ occurs at $\sigma=d/2-i$. The birational transformation $M_{i-1}\dashrightarrow M_i$ is an isomorphism outside of the loci $\mathbb{P}W_i^-\subset M_{i-1}$, $\mathbb{P}W_i^+\subset M_i$, where $W_i^-$ and $W_i^+$ are vector bundles over the symmetric product $\Sym^iC$ of rank $i$ and $d+g-1-2i$, respectively. We have a diagram
\begin{equation*}
\begin{tikzcd}[cramped,sep=scriptsize]
& \tilde{M}\arrow[dl]\arrow[dr] &\\
M_{i-1}=M_{\sigma+\epsilon}\arrow[dr] & & M_{\sigma-\epsilon}=M_i\arrow[dl]\\
& M_\sigma &
\end{tikzcd}
\end{equation*}
where $\tilde{M}$ is both the blow-up of $M_{\sigma+\epsilon}=M_{i-1}$ along $\mathbb{P}W_i^-$ and the blow-up of $M_{\sigma-\epsilon}=M_i$ along $\mathbb{P}W_i^+$. The variety  $M_\sigma$ is singular, obtained from the contraction to $\Sym^iC$ of the exceptional locus $\mathbb{P}W_{i}^-\times_{\Sym^iC}\mathbb{P}W_i^+ \subset \tilde{M}$.

When $d\gg 0$, $M_{\sigma\pm\epsilon}(\Lambda)$ and $M_{\sigma}(\Lambda)$ are obtained as GIT quotients of $U\times\mathbb{PC}^\chi$, with $\chi=d+2-2g$. When $d$ is arbitrary, take an effective divisor $D'$ of large degree, so that $M_\sigma\hookrightarrow M'_\sigma:=M_\sigma(\Lambda(2D'))$, where $M'_\sigma$ is a GIT quotient with a semi-stable locus $X'\subset U'\times\mathbb{PC}^{\chi'}$, $\chi'=d+2-2g+2\deg D'$. The spaces $M_{\sigma\pm\epsilon}(\Lambda)$ and $M_{\sigma}(\Lambda)$ are then GIT quotients by $SL_{\chi'}$ of a closed subset of $U'\times\mathbb{PC}^{\chi'}$ determined by the condition that in the pair $(E',\phi')$, the section $\phi'$ vanishes along $D'$. If we call $\mathcal{L}_\pm$, $\mathcal{L}_0$ the corresponding linearizations, we can write $X\subset X'$, the semi-stable locus of $\cL_0$, as the  union $X=X^{ss}(\mathcal{L}_+)\cup X^{ss}(\mathcal{L}_-)\sqcup Z$, where the locus $Z=X^{u}(\mathcal{L}_+)\cap X^{u}(\mathcal{L}_-)$ corresponds to pairs $(E',\phi')$, such that $E'$ splits as 
$$
E'=L'\oplus K',
$$
with $\deg L'=i+\deg D'$, $\deg K'=d-i+\deg D'$, and $\phi'\in H^0(L')$ vanishes along $D'$ (see \cite[1.4]{thaddeus}). The map $\mathcal{O}_C^{\chi'}\twoheadrightarrow E'$ is then given by a block-diagonal matrix $(\mathcal{O}_C^a\twoheadrightarrow L')\oplus(\mathcal{O}_C^b\twoheadrightarrow K')$, where $a=h^0(L')$, $b=h^0(K')$ and $a+b=h^0(L'\oplus K')=\chi'$. The strictly semi-stable locus $X^{sss}(\cL_0)=X^{u}(\mathcal{L}_+)\cup X^{u}(\mathcal{L}_-)$ consists of the orbits whose closure intersects $Z$ (cf. \cite[Remark 7.4]{potashnik}).

Using techniques from \cite{dhl} and \cite{katzarkov}, we compare the derived categories on either side of the wall $M_\sigma$. We write $M_{\sigma\pm\epsilon}=X\git_{\mathcal{L}_\pm}PGL_{\chi'}$ (cf. Remark \ref{SL vs PGL})
%We write both $M_{\sigma\pm\epsilon}$ $(\epsilon>0)$ as GIT quotients $X\git_{\mathcal{L}_\pm}PGL_{\chi'}$ of the same space $X$ 
and take Kempf--Ness stratifications of the unstable loci $X^{u}(\mathcal{L}_\pm)$ with strata $S^j_{\pm}$ determined by pairs $(Z^j,\lambda^j_\pm)$, where $\lambda^j_-(t)=\lambda^j_+(t)^{-1}$ are one-parameter subgroups and $Z^j$ is the fixed locus of $\lambda^j=\lambda^j_+$ (see \cite[\S 2.1]{dhl} for details).

\begin{remark}\label{unique KN stratum}
From the discussion above, it follows that in this case the KN stratification of the unstable locus in $X$ with respect to $\mathcal{L}_\pm$ has only one stratum $S_\pm$, parametrizing framed extensions as in \cite[(3.2),(3.3)]{thaddeus}. %consisting of the vector bundle $W_i^\pm$ over the locus $Z$ of split bundles described above. 
In the notation of \cite[\S2]{dhl}, the stratum $S_{\pm}$ is determined by the pair $(Z,\lambda)$, where $\lambda=\lambda_+=\mathbb{G}_m$ is the stabilizer of $Z$, and some power of $\lambda$ acts on a split bundle $E'=L'\oplus K'$ by $(t^b,t^{-a})$. 
\end{remark}

\begin{remark}\label{Stacky Z}
Let $\mathfrak{Z}$ be the stack $[Z/\mathbb{L}]$, where $\mathbb{L}$ is the Levi subgroup, i.e. the centralizer of $\lambda$ in $PGL_{\chi'}$. We have a short exact sequence of groups
$1\to\mathbb{G}_m\to \mathbb{L}\to PGL_a\times PGL_b\to 1$ with $\mathbb{G}_m=\lambda$ acting on $Z$ trivially and $[Z/PGL_a\times PGL_b]\simeq \Sym^iC$. Indeed, the action of $PGL_a\times PGL_b$ on $Z$ is free, and each orbit is determined by a divisor $D\in \Sym^iC$, where $D+D'$ is the zero locus of the section $\phi'\in H^0(L')$. Therefore $\mathfrak{Z}\simeq [\Sym^i C/\mathbb{G}_m]$, with the trivial action of $\mathbb{G}_m$.
\end{remark}

For $\sigma=d/2-i$ with $1<i\leq v$, $M_{\sigma\pm\epsilon}$ ($\epsilon>0$) is isomorphic to the corresponding quotient stack, since the action of $PGL_{\chi'}$ is free on the stable locus by \cite[1.6]{thaddeus}. Let $\eta_{\pm}=\weight_{\lambda_\pm} \left.\det\mathcal{N}^\vee_{S_{\pm}/X}\right|_{Z}$. For~any choice of an integer $\win$, $D^b(M_{\sigma\pm\epsilon})$ is equivalent to the window subcategory $G_{\win}^{\pm}\subset D^b([X/PGL_{\chi'}])$ determined by objects having $\lambda_{\pm}$-weights in the range $[\win,\ \win+\eta_\pm)$ for the unique stratum $S_\pm$ (see \cite[Theorem 2.10]{dhl}). If~$\weight_{\lambda}\omega_X|_{Z}=\eta_--\eta_+>0$, we get an embedding $D^b(M_{\sigma+\epsilon})\subset D^b(M_{\sigma-\epsilon})$ (see~\cite[Proposition 4.5]{dhl} and the Remark following it).

%In our case there will be only one KN stratum to consider, so we will drop the index $j$. We compute the widths $\eta_\pm$ of the windows in the corresponding wall-crossing. 

\begin{lemma}\label{eta computation}
In the wall-crossing between the spaces $M_{\sigma+\epsilon}(\Lambda)=M_{i-1}$ and $M_{\sigma-\epsilon}(\Lambda)=M_{i}$, %the unstable locus is a single KN stratum. 
the window has width
$\eta_+=i$, 
$\eta_-=d+g-1-2i$.
%after the standard rescaling of the $1$-parameter subgroup (see the  proof).
\end{lemma}

\begin{proof}

%When $d\gg 0$, $M_{\sigma\pm\epsilon}(\Lambda)$ and $M_{\sigma}(\Lambda)$ are obtained as GIT quotients of $U\times\mathbb{PC}^\chi$, with $\chi=d+2-2g$. When $d$ is arbitrary, take an effective divisor $D'$ of large degree, so that $M_\sigma\hookrightarrow M'_\sigma:=M_\sigma(\Lambda(2D'))$, where $M'_\sigma$ is a GIT quotient with a semi-stable locus $X'\subset U'\times\mathbb{PC}^{\chi'}$, $\chi'=d+2-2g+2\deg D'$. The spaces $M_{\sigma\pm\epsilon}(\Lambda)$ and $M_{\sigma}(\Lambda)$ are then GIT quotients by $SL_{\chi'}$ of a closed subset of $U'\times\mathbb{PC}^{\chi'}$ determined by the condition that in the pair $(E',\phi')$, the section $\phi'$ vanishes along $D'$. If we call $\mathcal{L}_\pm$, $\mathcal{L}_0$ the corresponding linearizations, we can write $X\subset X'$, the semi-stable locus of $\cL_0$, as the  union
%$M_{\sigma\pm\epsilon}(\Lambda)$ as GIT quotients of $X\subset U\times\mathbb{PC}^\chi$, where $X$ is the union of the three semi-stable loci, so that $X=X^{ss}(\mathcal{L}_+)\cup X^{ss}(\mathcal{L}_-)\sqcup Z$, where the strictly semi-stable locus $Z=X^{sss}(\mathcal{L}_0)=X^{u}(\mathcal{L}_+)\cap X^{u}(\mathcal{L}_-)$ corresponds to pairs $(E',\phi')$, where $E'$ splits as $E'=L'\oplus K'$, with $\deg L'=i+\deg D'$, $\deg K'=d-i+\deg D'$, and $\phi'\in H^0(L')$ vanishes along $D'$ (see \cite[1.4]{thaddeus}). The map $\mathcal{O}_C^{\chi'}\twoheadrightarrow E'$ is given by a block-diagonal matrix $(\mathcal{O}_C^a\twoheadrightarrow L')\oplus(\mathcal{O}_C^b\twoheadrightarrow K')$.

We use the notation as in the discussion above, with $M_\sigma\hookrightarrow M'_\sigma:=M_\sigma(\Lambda(2D'))$, $D'$ effective with $\deg D'\gg 0$.
For $\mathcal{L}_\pm$, there is no strictly semi-stable locus and in fact $PGL_{\chi'}$ acts freely on the semi-stable locus \cite[1.6]{thaddeus}, so 
$M_{i-1}=M_{\sigma+\epsilon}(\Lambda)=X\git_{\mathcal{L}_+}SL_{\chi'}$ and $M_i=M_{\sigma-\epsilon}(\Lambda)=X\git_{\mathcal{L}_-}SL_{\chi'}$
are isomorphic to the quotient stacks $[{X}^{ss}(\mathcal{L}_\pm)/PGL_{\chi'}]$ (cf. Remark \ref{SL vs PGL}).
By Lemma~\ref{smooth_stack}, both $[X/PGL_{\chi'}]$ and $[X'/PGL_{\chi'}]$ are smooth quotient stacks of dimension $d+g-2$ and $d+g-2+2\deg D'$, respectively, and thus $X$ and $X'$ are both smooth and $X\subset X'$ is a local complete intersection cut out precisely by the $2\deg D'$ conditions imposed by the vanishing of a section along~$D'$.

Recall that the unique KN stratum of $X^{u}(\cL_\pm)$ is determined by $(Z,\lambda)$ (cf. Remark \ref{unique KN stratum}), where for a pair $(E',\phi')\in Z$, the bundle $E'=L'\oplus K'$ is acted on by (some power of) $\lambda=\mathbb{G}_m$ by $(t^b,t^{-a})$. 
%The KN stratification of the unstable locus in $X$ with respect to $\mathcal{L}_\pm$ has a unique stratum $S_\pm$, consisting of the vector bundle $W_i^\pm$ over $Z$. The stabilizer of $Z$ is $\lambda=\mathbb{G}_m$, acting on $L'\oplus K'$ by $(t^b,t^{-a})$, where $a=h^0(L')$, $b=h^0(K')$, and one can show that 
We will first compute the weights with respect to this action, and later rescale according to the parametrization that describes the whole one-parameter subgroup.
By \cite[Lemma 7.6]{potashnik} and its proof,
the $\lambda$-weights of $\mathcal{N}^\vee_{S_\pm/ X'}$ on $Z$ are all $\pm (a+b)=\pm \chi'$ or $0$.
%(see \cite[\S 7]{potashnik}). 
Then the weights of $\mathcal{N}_{S_\pm/X}^\vee $ are all $\pm \chi'$, and $\eta_{\pm}=\weight_{\lambda_\pm}\left.\det \mathcal{N}^\vee_{S_{\pm}/X}\right|_Z$ is just the codimension of $S_\pm\subset X$.
Since $S_\pm$ is the bundle $W_i^\pm$ on $Z$, we have $\codim (S_\pm\subset X)=\rank W_i^\mp$, so that $\eta_+=i\chi'$ and $\eta_-=(d+g-1-2i)\chi'$.

As a one-parameter subgroup of $PGL_{\chi'}$, $\lambda$ is given by sending $t\mapsto diag(s^b,\ldots,s^b,s^{-a},\ldots,s^{-a})$, where $s^{\chi'}=t$. Note that this is well defined, since when replacing $s$ by $\xi s$, with $\xi$ a $\chi'$-th root of unity, the matrix $\lambda(t)$ gets scaled by $\xi^b=\xi^{-a}$. Therefore, all weights computed above need to be rescaled by $1/\chi'$. This gives the formulas in the statement.
\end{proof}

Using this we obtain the following result.

\begin{proposition}\label{potashnik embedding}
For $1\leq i \leq \frac{d+g-1}{3}$
(resp., $i\geq \frac{d+g-1}{3}$)
there is an admissible
embedding $D^b(M_{i-1})\hookrightarrow D^b(M_i)$ (resp.,  $D^b(M_{i})\hookrightarrow D^b(M_{i-1})$). When $1<i\leq \frac{d+g-1}{3}$, the admissible embedding can be chosen to be the window subcategory $G_0^+\subset D^b(M_i)$ determined by the range of weights $[0,i)\subset [0,d+g-1-2i)$ (cf. \cite{dhl}) 
%after the standard rescaling of the $1$-parameter subgroup $\lambda$ 
and moreover there is a semi-orthogonal decomposition 
\begin{align}\label{sod bfr}
    D^b(M_i)=\langle D^b(M_{i-1}),D^b(\Sym^iC),\ldots,D^{b}(\Sym^iC)\rangle
\end{align}
with $\mu=d+g-3i-1$ copies of $D^b(\Sym^iC)$
given by the fully faithful images of functors $Rj_*\bigl(L\pi^*(\cdot)\otimes^L\mathcal{O}_\pi (l)\bigr):D^b(\Sym^iC)\to D^b(M_i)$ for $l=0,\ldots,\mu-1$, where $\pi:\mathbb{P}W_i^+\to \Sym^iC$ is the projection and $j:\mathbb{P}W_i^+\hookrightarrow M_i$ the inclusion.
\end{proposition}

The semi-orthogonal decomposition \eqref{sod bfr} follows from \cite{bfr}, as the birational transformation between $M_{i-1}$ and $M_i$ is a standard flip of projective bundles over $\Sym^iC$. Here we provide an alternative proof for this case. We also note that \cite[Corollary 8.1]{potashnik} shows the admissible embeddings $D^b(M_{i-1})\hookrightarrow D^b(M_i)$ when $i$ is in the specified range. %Maybe also cite \cite{bfr} here.

As explained in the introduction, Proposition~\ref{potashnik embedding} does not provide a semi-orthogonal decomposition with Fourier--Mukai functors associated with Poincar\'e bundles and it is not used in our paper. However, we find this result relevant.

\begin{proof}
If $i=1$, this follows from Orlov's blow-up formula \cite{orlov}. Let $i>1$.
From Lemma \ref{eta computation}, $\weight_\lambda\omega_X|_Z=\eta_--\eta_+=(d+g-1-3i)$.
%after the standard rescaling. 
By \cite[Proposition 4.5 and Remark 4.6]{dhl}, and since $M_{\sigma\pm\epsilon}\simeq [{X}^{ss}(\mathcal{L}_\pm)/PGL_{\chi'}]$, we get a window embedding $D^b(M_{\sigma+\epsilon})\subset D^b(M_{\sigma-\epsilon})$ if $\eta_+\leq\eta_-$ and the other way around if $\eta_+\geq \eta_-$. 
Moreover, if $G_{\win}^+= D^b(M_{\sigma+\epsilon})$ is a window, determined by the range of weights $[\win,\ \win+\eta_+)\subset [\win,\ \win+\eta_-)$,
%(after the standard rescaling), 
then \cite[Theorem 2.11]{dhl} and \cite[Theorem 1]{katzarkov} give semi-orthogonal blocks $D^b(\mathfrak{Z})_k$, so that 
\begin{align}\label{sod katzarkov SymC}
D^b(M_{\sigma-\epsilon})=\langle G_{\win}^+,D^b(\mathfrak{Z})_{\win},\ldots,D^b(\mathfrak{Z})_{\win+\mu-1}\rangle,
\end{align}
where $\mu=\eta_--\eta_+$ and $\mathfrak{Z}=[Z/\mathbb{L}]$ is the quotient stack by the Levi subgroup. 
%Here $\mathfrak{Z}=[Z/L]$, where $L$ is the Levi subgroup, i.e. the centralizer of $\lambda$ in $PGL_{\chi'}$.%, acting on $Z$. We have a short exact sequence of groups $1\to\mathbb{G}_m\to L\to PGL_a\times PGL_b\to 1$ with $\mathbb{G}_m=\lambda$ acting on $Z$ trivially and $[Z/PGL_a\times PGL_b]\simeq \Sym^iC$. Then $\mathfrak{Z}\simeq [\Sym^i C/\mathbb{G}_m]$, with the trivial action of $\mathbb{G}_m$, 
By Remark \ref{Stacky Z}, 
$D^b(\mathfrak{Z})=D_{\mathbb{G}_m}^b(\Sym^iC)$, so the blocks in (\ref{sod katzarkov SymC}) are given by the fully faithful images of $Rj_*\bigl(L\pi^*(\cdot)\otimes^L\mathcal{O}_\pi (l)\bigr):D^b(\Sym^iC)\to D^b(M_i)$ for $l\in[\win,\ \win+\mu)$, where $\pi:\mathbb{P}W_i^+\to \Sym^iC$ is the projection and $j:\mathbb{P}W_i^+\hookrightarrow M_i$ the inclusion. Taking $\win=0$ gives the claim.
\end{proof}

\begin{corollary}\label{potashnik embedding corollary}
If $d\le 2g-1$, then $D^b(M_{i-1})\subset D^b(M_i)$ for any $1\leq i\leq v$.
\end{corollary}

\begin{proof}
In this case $i \leq (d-1)/2\leq g-1$, so the inequality $i < (d+g-1)/3$ holds for every $i$.
\end{proof}

Consider an object $G$ in $D^b([X/PGL_{\chi'}])$ descending to some objects on $D^b(M_{i-1})$ and $D^b(M_i)$. We can use windows to determine when such object can ``cross the wall''. Namely, if the weights of $G$ are in the required range, cohomology groups will be the same on either side. By abuse of notation, we often denote in the same way both the object on $D^b([X/PGL_{\chi'}])$ and the objects it descends to in $M_{\sigma\pm\epsilon}(\Lambda)$.

\begin{theorem}\label{windows wall-crossing}
Let $\sigma=d/2-i$, $1<i\leq v$. If $A$, $B$ are objects in $D^b([X/PGL_{\chi'}])$, with $\lambda=\lambda_+$-weights satisfying the inequalities
\begin{align}\label{weights within eta}
1+2i-d-g<\weight_\lambda B|_Z-\weight_\lambda A|_Z <i
\end{align}
then
$R\Hom_{M_{\sigma+\epsilon}}(A,B)=R\Hom_{M_{\sigma-\epsilon}}(A,B)$.
In particular, if 
$
1+2i-d-g<\weight_\lambda B|_Z <i
$
then 
$
R\Gamma_{M_{i-1}}(B)=R\Gamma_{M_{i}}(B)
$.
\end{theorem}

\begin{proof}
By Lemma \ref{eta computation}, \eqref{weights within eta} is equivalent to inequalities  
$$-\eta_-<\weight_\lambda B|_Z-\weight_\lambda A|_Z <\eta_+,$$
so the Quantization Theorem \cite[Theorem 3.29]{dhl} implies that
\begin{align*}
R\Hom_{M_{\sigma+\epsilon}}(A,B)=R\Hom_{[X/PGL_{\chi'}]}(A,B)=R\Hom_{M_{\sigma-\epsilon}}(A,B).
\end{align*}
Indeed, the first equality follows directly from \cite[Theorem 3.29]{dhl} applied on $M_{\sigma+\epsilon}$, while the second is the same theorem applied on $M_{\sigma-\epsilon}$, using the fact that
$
\weight_{\lambda_-} B|_Z-\weight_{\lambda_-} A|_Z=-(\weight_\lambda B|_Z-\weight_\lambda A|_Z)
$.
\end{proof}

We finish this section with the computation of all weights that we need in order to construct the semi-orthogonal decompositions.

\begin{theorem}\label{weights computations of all bundles}
The objects of the form $F_x$, $\Lambda_M$, $\psi$, $\zeta$, $G_D$ on both $M_{i-1}$ and $M_i$ are the descents of objects $\Tilde{F}_x$, $\Tilde{\Lambda}_M$, $\Tilde{\psi}$, $\Tilde{\zeta}$, $\Tilde{G}_D$ on $D^b([X/PGL_{\chi'}])$ that
%after the standard rescaling (see the proof of Lemma~\ref{eta computation}), 
have $\lambda$-weights
\begin{align*}
&\weights_\lambda \Tilde{F}_x|_Z=\{0,-1\}\\
&\weight_\lambda \Tilde{\Lambda}_M|_Z=-1\\
&\weight_\lambda \Tilde{\psi}|_Z=d+1-g-i\\
&\weight_\lambda \Tilde{\zeta}|_Z=g-i\\
&\weights_\lambda \Tilde{G}_D|_Z=\{0,-1,\ldots,-\deg D\}.
%\item $\weight_\lambda \Tilde{\theta}|_Z=d-2i$
\end{align*}
\end{theorem}

\begin{proof}
Let $\sigma=d/2-i$ and embed $\imath:M_\sigma(\Lambda)\hookrightarrow M'_\sigma =M_\sigma (\Lambda (2D'))$ for an effective divisor $D'$, $\deg D'\gg 0$, as usual.
%The spaces $M'_{\sigma\pm\epsilon}=M_\sigma(\Lambda(2D'))$ are GIT quotients of $X'$ by $SL_{\chi'}$ and
Recall that the universal bundle $F'$ on $C\times M'_{\sigma\pm\epsilon}$ is the descent of $\mathcal{F}'(1)$ on $C\times X'\subset C\times U'\times\mathbb{PC}^{\chi'}$, where $\mathcal{F}'$ is the universal family on $C\times U'$ \cite[1.19]{thaddeus}.
%The $\sigma$-strictly semi-stable locus $Z'\subset X'$ corresponds to split bundles $L'\oplus K'$ together with a section $\phi'\in H^0(L')$, and the action of $SL_{\chi'}$ on $H^0(E')$ is given by $(t^{b},t^{-a})$, where $a=h^0(L')$, $b=h^0(K')$ and $a+b=h^0(L'\oplus K')=\chi'$.
Let us compute the $\lambda$-weights of $\mathcal{F}'_x(1)$ on the $\sigma$-strictly semi-stable locus, for a point $x\in C$. The fiber of $\mathcal{F}'_x$ over $L'\oplus K'$ is $L'_x\oplus K'_x$, which is acted on with weights $b$ in the first component and $-a$ in the second. Since the $\lambda$-weight of $\mathcal{O}_{\mathbb{PC}^{\chi'}}(1)$ over the section $(\phi',0)$ is $-b$, the weights of $\mathcal{F}'_x(1)$ are $0$ and $-a-b=-\chi'$. By Lemma \ref{usefuliso}, we have $F_x \simeq \imath^*F'_x$. Hence, $F_x$ also is the descent of an object with weights $0$ and $-\chi'$.

The bundle $\det {\pi_!}F'$ descends from $\det {\pi_!}\mathcal{F}'(1)$. On the fiber of $\pi_!\mathcal{F}'$ over $L'\oplus K'$, $\lambda$ acts on $H^0(L')\oplus H^0(K')$ with weights $b$ and $-a$, with multiplicities $h^0(L')=a$ and $h^0(K')=b$, respectively. Taking tensor product with $\mathcal{O}_{\mathbb{PC}^{\chi'}}(1)$ shifts each weight by $-b$, and then taking the determinant we get $\weight_\lambda \det {\pi_!}\mathcal{F}'(1)|_{Z'}=0\cdot a+(-a-b)\cdot b=-b\chi'$.
For $\det F'_x$, which is the descent of $\det\mathcal{F}'_x(1)$, we see that $\lambda$ acts with weights $b,-a$ on $L'_x\oplus K'_x$ and then shifting by $-b$ and taking determinants we get $\weight_\lambda\det\mathcal{F}'_x(1)|_{Z'}=-a-b=-\chi'$.

Now for the universal bundle $F$ on $C\times M_{\sigma\pm\epsilon}(\Lambda)$, we use the short exact sequence \eqref{usefulseq}.
%\begin{align*}
%0\to F\to\imath^*F'\to \imath^*F'|_{D\times M_{\sigma\pm\epsilon}}\to 0.
%\end{align*}
From this we see that $\Lambda_M=\det F_x\simeq\det F'_x$ is the descent of an object with $\lambda$-weight equal to $-\chi'$. Also, since $\det\pi_!F'|_{D'\times M_{\sigma\pm\epsilon}}=\det\bigoplus_{x\in D'}F'_x=(\det F'_x)^{\deg D'}$, we obtain that $\psi^{-1}=\det \pi_!F=\det {\pi_!}F'\otimes (\det F'_x)^{-\deg D'}$ %(cf. \cite[Proposition 8]{det-coh} and the discussion before that), which
is the descent of an object with $\lambda$-weight equal to $-b\chi'+\deg D' \chi'$. Recall $\deg L'=i+\deg D'$, $\deg K'=d-i+\deg D'$ (see the discussion before Remark \ref{unique KN stratum}), so by Riemann-Roch $b=h^0(K')=d-i+\deg D'+1-g$ and the weight of ${\psi}$ is $-\chi'(\deg D'-b)=\chi'(d+1-g-i)$. As for $\zeta=\psi\otimes\Lambda_M^{d-2g+1}$, the weights must be $(d+1-g-i-(d-2g+1))\chi'=(g-i)\chi'$. Rescaling everything by $1/\chi'$ as in Lemma \ref{eta computation}, we get the weights as in the statement.

Finally, we consider  $G_D$. Let $D=x_1+\ldots+x_\alpha$. 
Since by construction tensor bundles are functorial in $M$,
the bundle $G_D$ is the descent of a vector bundle $(\cE^{\boxtimes \alpha})_D$ on $X$, where $M=X\git SL_{\chi'}$ and $\cE$ descends to~$F$. By Lemma \ref{standarddeformation}, $(\cE^{\boxtimes \alpha})_D$ is a deformation of $\cE_{x_1}\otimes\ldots\otimes\cE_{x_\alpha}$, and the deformation can be chosen to be $SL_{\chi'}$-equivariant (see Remark \ref{remark equivariantdeformation}). Therefore, $(\cE^{\boxtimes \alpha})_D$ has the same weights as the tensor product $\cE_{x_1}\otimes\ldots\otimes\cE_{x_\alpha}$, i.e. $0,-1,\ldots,-\alpha$.
%We claim that $G_D$ has the same weights as the tensor product $F_{x_1}\otimes\ldots\otimes F_{x_\alpha}$, i.e.~$0,-1,\ldots,-\alpha$ (after rescaling). By Lemma~\ref{differentpoints}, we can assume that $D=\alpha x$ is a fat point. Fix a point $z$ in the strictly semi-stable locus. As discussed above, $\cE|_{C\times\{z\}}$ is a split vector bundle $L\oplus K$  on $C$ and $\lambda$ acts on its summands with weights $0$, $-1$ (after the standard rescaling). 
%As in Corollary \ref{standarddeformation}, consider the diagram
%\begin{equation}\label{shortdiagram}
%\begin{tikzcd}
%&\mathbb{B}_\alpha \arrow[r,"\pi_j"]\arrow{dr}[swap]{\tau} %&\mathbb{D}_\alpha \arrow[r,"q"]\arrow[d,bend %left=20,"\rho"] & C %\arrow[dl] 
%\\
%& & \{z\} \arrow[u,bend left = 20,"\imath"] 
%&
%\end{tikzcd}
%\end{equation}
%Consider the diagram \eqref{diagram tau} with $M=\{z\}$. The fiber of the vector bundle $\left(\mathcal{E}^{\boxtimes \alpha}\right)_{D}$ at $z\in X$ can be computed as $\left(\mathcal{E}^{\boxtimes \alpha}\right)_{D,z}=\tau_*^{S_\alpha}\bigotimes\pi_j^*q^*(L\oplus K)$. The one-parameter subgroup $\lambda$ acts on all schemes in \eqref{diagram tau} trivially and on $L\oplus K$ with weights $0$ and $-1$. The $\cO_{\bD_\alpha}$-module $q^*(L\oplus K)$ is free, and although the trivialization is not canonical, it preserves the weight structure, i.e. can be chosen to be $\lambda$-equivariant. Thus $\left(\mathcal{E}^{\boxtimes \alpha}\right)_{D,z}$ has the same weights as $(L\oplus K)^{\otimes\alpha}_x$.
\end{proof}

\begin{remark}\label{weights of line bundles}
Observe that $\cO_i(1,0)=\psi\otimes\Lambda_M^{d-g}$ and $\cO_i(0,1)=\Lambda_M^{-1}$, so we can use the previous theorem to see that in general, a line bundle $\cO_i(m,n)$ is the descent on both $M_{i-1}$ and $M_i$ of an object having $\lambda$-weight $m (1-i)+n$ on the strictly semi-stable locus of the wall.
\end{remark}

\section{Acyclic vector bundles on $M_i$ -- easy cases}

In order to prove Theorem \ref{maintheorem}, we will first construct fully faithful functors $\Phi_\alpha^i:D^b(\Sym^\alpha C)\hookrightarrow D^b(M_i)$ for $1\leq\alpha\leq i$ and show that, after suitable twists, the essential images of these functors are semi-orthogonal to each other in the required way (see Theorem \ref{fully faithfulness}, Definition \ref{definition ABCD} and Theorem \ref{the sod on Mi} below). By means of Bondal-Orlov's criterion \cite{bondal-orlov}, this reduces to the computation of $R\Gamma$ for a large class of vector bundles on $M_i$. In particular, we will need to prove $\Gamma$-acyclicity for several of these vector bundles.

\begin{theorem}\label{vanishing with Z on Mi}
Let $d>2$ and $1\leq i\leq v$. Let $D=x_1+\ldots+x_\alpha$, $D'=y_1+\ldots+y_\beta$ (possibly with repetitions). Suppose
$$\deg D-g<t<
d-\deg D'-i-1.$$
Then
\begin{equation}\label{eq easy vanishing}
R\Gamma_{M_i(d)}\left(\bigotimes_{k=1}^\alpha F_{x_k}^\vee\otimes \bigotimes_{k=1}^\beta F_{y_k}\otimes\Lambda_M^t\otimes \zeta^{-1}\right)=0.
\end{equation}
\end{theorem}

\begin{remark}\label{remark easy vanishing}
By Corollary \ref{standarddeformation} and semi-continuity, the same vanishing holds if in \eqref{eq easy vanishing} we replace $\bigotimes_{k=1}^\alpha F_{x_k}^\vee$ by either $G_D^\vee$ or $\overline{G}_D^\vee$ and $\bigotimes_{k=1}^\beta F_{y_k}$ by either $G_{D'}$ or $\overline{G}_{{D}'}$.
\end{remark}

We start with a lemma.

\begin{lemma}\label{O(-kH+lE)}
$R\Gamma_{M_1(d)}(\mathcal{O}_{M_1(d)}(-kH+lE_1))=0$ for $0<k\leq d+g-2$ and $0\leq l\leq d+g-4$. In particular, taking $t=k=l$ we get $R\Gamma_{M_1(d)}(\Lambda_M^t)=0$ for $0< t \leq d+g-4$.
\end{lemma}

\begin{proof}
Consider the short exact sequence
\begin{align}\label{short ex seq from lemma}
0\to\mathcal{O}_{M_1(d)}\to\mathcal{O}_{M_1(d)}(E_1)\to\mathcal{O}_\pi(-1)\to 0,
\end{align}
where $E_1=\mathbb{P}W_1^+$ and $\pi:E_1\to C$ is the $\mathbb{P}^r$-bundle, $r=d+g-4$. $\mathcal{O}_{M_1(d)}(-kH)$ is $\Gamma$-acyclic provided $0<k\leq d+g-2=\dim M_1(d)$. Then twisting (\ref{short ex seq from lemma}) by $\mathcal{O}_{M_1(d)}(-kH)$ and taking a long exact sequence in cohomology gives $\Gamma$-acyclicity of $\mathcal{O}_{M_1(d)}(-kH+E_1)$ for such $k$. Similarly, twisting by powers of $\mathcal{O}_{M_1(d)}(E_1)$ and using induction, we get that $R\Gamma_{M_1(d)} (\mathcal{O}_{M_1(d)}(-kH+lE))=0$ as well, since $\mathcal{O}_\pi(-l)$ is $\Gamma$-acyclic for $0< l\leq d+g-4$.
\end{proof}

We will prove Theorem \ref{vanishing with Z on Mi} by induction, starting with the base case $i=1$.

\begin{lemma}\label{vanishing with Z on M1}
The statement of Theorem~\ref{vanishing with Z on Mi} holds for $i=1$.
\end{lemma}

\begin{proof}
Let $\alpha=\deg D$, $\beta=\deg D'$.
We are given that
$\alpha-g<t<d-\beta-2$.
We do induction on $\alpha+\beta$. If $\alpha=\beta=0$, %then $G_D= G_{D'}\simeq\cO_M$ (cf. Remark \ref{empty divisor}), so 
we have to check that $\Lambda_M^t\otimes \zeta^{-1}=-(t+g)H+(g+t-1)E_1$ is $\Gamma$-acyclic on $M_1(d)$. By Lemma $\ref{O(-kH+lE)}$, this holds provided
$ 0<t+g \leq d+g-2$ and $0\leq g+t-1 \leq d+g-4$,
which is true by hypothesis.

If $\alpha>0$, we write $D=\tilde{D}+x_\alpha$.
%for some point $x\in C$. Then by Proposition \ref{proposition stable deformation} $G_D^\vee$ is a stable deformation of $F_x^\vee\otimes G_{\tilde{D}}^\vee$ over $\bA^1$, so it suffices by semi-continuity to show $\Gamma$-acyclicity of $F_x^\vee\otimes G_{\tilde{D}}^\vee\otimes G_{D'}\otimes \Lambda_M^t\otimes \zeta^{-1}$.
Consider the exact sequence (\ref{reduction F dual}) from Lemma~\ref{BasicKoszul} and twist it by $U:=\bigotimes_{k=1}^{\alpha-1} F_{x_k}^\vee\otimes \bigotimes_{k=1}^\beta F_{y_k}\otimes\Lambda_M^t\otimes \zeta^{-1}$ to get 
\begin{align}\label{twisted seq}
0\to \Lambda_M^{-1}\otimes U \to \bigotimes_{k=1}^\alpha F_{x_k}^\vee\otimes \bigotimes_{k=1}^\beta F_{y_k}\otimes\Lambda_M^t\otimes \zeta^{-1} \to U \to U|_{M_0(d-2)} \to 0.
\end{align}
The restriction of $F_y$ to $M_0(d-2)=\mathbb{P}^{r}$, $r=d+g-4$, is equal to $\mathcal{O}_{\mathbb{P}^r}\oplus \mathcal{O}_{\mathbb{P}^r}(-1)$ by Lemma \ref{Fx=O+O(-1)}. Therefore, we see that the restriction of the bundle
%by Corollary \ref{standarddeformation}, we have that $\left.G_{\tilde{D}}^\vee\otimes G_{D'}\otimes \Lambda_M^t\otimes \zeta^{-1}\right|_{M_0(d-2)}$ 
$\bigotimes_{k=1}^{\alpha-1} F_{x_k}^\vee\otimes \bigotimes_{k=1}^\beta F_{y_k}\otimes\Lambda_M^t\otimes \zeta^{-1}$ to ${M_0(d-2)}$
%is a deformation over $\bA^1$ of 
is a sum of bundles $\bigoplus \mathcal{O}_{\mathbb{P}^r}(s_j)\otimes \mathcal{O}_{\mathbb{P}^r}(1-t-g)$, with $-\beta\leq s_j\leq \alpha-1$ (cf. Remark \ref{restriction of O(m,n)}). These are all $\Gamma$-acyclic on $\mathbb{P}^{d+g-4}$, since by hypothesis
\begin{align}\label{ineq on P^r}
    \alpha-t-g <0,\quad 
    -\beta+1-t-g\geq-(d+g-4).
\end{align}
%Then $\left.G_{\tilde{D}}^\vee\otimes G_{D'}\otimes \Lambda_M^t\otimes \zeta^{-1}\right|_{M_0(d-2)}$ is $\Gamma$-acyclic too, by semi-continuity.

The other two terms from the sequence \eqref{twisted seq} are $\bigotimes_{k=1}^{\alpha-1} F_{x_k}^\vee\otimes \bigotimes_{k=1}^\beta F_{y_k}\otimes\Lambda_M^t\otimes \zeta^{-1}$ and $\bigotimes_{k=1}^{\alpha-1} F_{x_k}^\vee\otimes \bigotimes_{k=1}^\beta F_{y_k}\otimes\Lambda_M^{t-1}\otimes \zeta^{-1}$
%$G_{\tilde{D}}^\vee\otimes G_{D'}\otimes \Lambda_M^t\otimes \zeta^{-1}$ and $G_{\tilde{D}}^\vee\otimes G_{D'}\otimes \Lambda_M^{t-1}\otimes \zeta^{-1}$. 
We observe that they both satisfy the inequalities of the hypothesis, so by induction they are $\Gamma$-acyclic on $M_1(d)$.
%, and we conclude that so is $G_D^\vee\otimes G_{D'}\otimes\Lambda_M^t\otimes \zeta^{-1}$, as desired.

Similarly, if $\beta>0$ we write $D'=\tilde{D'}+y_\beta$ and use the exact sequence
\begin{align*}%\label{reduction F non-dual}
    0\to\mathcal{O}_{M_1(d)}\to F_{y_\beta}\to \Lambda_M\to \left.\Lambda_M\right|_{M_0(d-2)}\to 0,
\end{align*}
twisted with $\bigotimes_{k=1}^{\alpha} F_{x_k}^\vee\otimes \bigotimes_{k=1}^{\beta-1} F_{y_k}\otimes\Lambda_M^t\otimes \zeta^{-1}$. The resulting term on the right is
%a deformation over $\bA^1$ of 
a sum $\bigoplus \mathcal{O}_{\mathbb{P}^r}(s_j)\otimes \mathcal{O}_{\mathbb{P}^r}(-t-g)$, with $-\beta+1\leq s_j\leq \alpha$, and it is again $\Gamma$-acyclic by the same inequalities (\ref{ineq on P^r}). Finally, the remaining two terms 
%$G_D^\vee\otimes G_{\tilde{D'}}\otimes\Lambda_M^t\otimes \zeta^{-1}$ and $G_D^\vee\otimes G_{\tilde{D'}}\otimes\Lambda_M^{t+1}\otimes \zeta^{-1}$ 
are $\Gamma$-acyclic by induction, and we conclude that $R\Gamma_{M_1(d)}(\bigotimes_{k=1}^{\alpha} F_{x_k}^\vee\otimes \bigotimes_{k=1}^\beta F_{y_k}\otimes\Lambda_M^t\otimes \zeta^{-1})=0$.
\end{proof}

\begin{proof}[Proof of Theorem~\ref{vanishing with Z on Mi}]
Let $\alpha=\deg D$ and $\beta=\deg D'$.
We do induction on $i$. If $i=1$, this is Lemma \ref{vanishing with Z on M1}. Let $i>1$ and suppose the statement holds for $i-1$. For $t$ in the given range, we have
$$
R\Gamma_{M_{i-1}(d)}\left(\bigotimes_{k=1}^{\alpha} F_{x_k}^\vee\otimes \bigotimes_{k=1}^\beta F_{y_k}\otimes\Lambda_M^t\otimes \zeta^{-1}\right)=0
$$
by induction hypothesis.
Consider the wall-crossing between $M_{i-1}$ and $M_i$. Here, the bundle $\bigotimes_{k=1}^{\alpha} F_{x_k}^\vee\otimes \bigotimes_{k=1}^\beta F_{y_k}\otimes\Lambda_M^t\otimes \zeta^{-1}$ descends from an object with weights
$
%\weights_\lambda \Tilde{G}_{D}^\vee\otimes \Tilde{G}_{D'}\otimes\Tilde{\Lambda}_M^t\otimes \Tilde{\zeta}^{-1} = 
\{-\beta-t+i-g,\ldots, \alpha-t+i-g\}
$
(see Theorem \ref{weights computations of all bundles}). Our hypothesis guarantees that $\alpha-t+i-g<i=\eta_+$ and $-\beta-t+i-g>1+2i-d-g=-\eta_-$, that is, all these weights live in the range $(-\eta_-,\ \eta_+)$. By Theorem \ref{windows wall-crossing} this implies
$
R\Gamma_{M_i(d)}(\bigotimes_{k=1}^{\alpha} F_{x_k}^\vee\otimes \bigotimes_{k=1}^\beta F_{y_k}\otimes\Lambda_M^t\otimes \zeta^{-1})=R\Gamma_{M_{i-1}(d)}(\bigotimes_{k=1}^{\alpha} F_{x_k}^\vee\otimes \bigotimes_{k=1}^\beta F_{y_k}\otimes\Lambda_M^t\otimes \zeta^{-1})=0,
$
as desired.
\end{proof}

\section{A fully faithful embedding $D^b(C)\subset D^b(M_1)$}\label{M1 section}

The following Theorem \ref{Poincare is fully faithful on M1} is a special case of Theorem \ref{fully faithfulness}, and will be needed for our proof of the latter. Namely, the result of Theorem \ref{Poincare is fully faithful on M1} will be used in Sections \ref{hard section} and \ref{fully section}, in results that are necessary for Theorem \ref{fully faithfulness}. While Theorem \ref{Poincare is fully faithful on M1} could be avoided by including it as a step of a more complicated inductive proof, we find it more convenient to prove it first, both to make the inductions less cumbersome and to introduce some ideas that will help understand the general picture.

We assume that $v\ge1$, i.e. $d\ge3$. As before, let $E_1\subset M_1$ be the exceptional locus of the blow-up $M_1\to M_0$ along $C\subset M_0$. By Orlov's blow-up formula \cite{orlov}, we have a fully faithful functor $\Psi:D^b(C)\hookrightarrow D^b(M_1)$, corresponding to the Fourier--Mukai transform given by $\mathcal{O}_{Z}(E_1)$, where $Z=C\times_C E_1$. 
Now consider the Fourier--Mukai transform 
$$\Phi_F=Rp_*(Lq^*(\cdot)\otimes^L F):D^b(C)\to D^b(M_1)$$
determined by the universal bundle $F$ on $C\times M_1$.

\begin{theorem}\label{Poincare is fully faithful on M1}
The functor $\Phi_F$ is fully faithful.
\end{theorem}

We need a few constructions and lemmas first. Observe that $Z=C\times_C E_1$ is supported precisely on the zero locus of the universal section $\tilde{\phi}:\mathcal{O}_{C\times M_1} \to F$. %, and it is a local complete intersection. 
Indeed, pairs $(E,\phi)$ in $\mathbb{P}W_1^+=E_1$ parametrize extensions 
\begin{align*}
0\to\mathcal{O}_C(x)\to E\to \Lambda(-x)\to 0
\end{align*}
with the canonical section $\phi\in H^0(C,\mathcal{O}_C(x))$ vanishing on $x\in C$ \cite[3.2]{thaddeus}, and in fact $\tilde{\phi}$ has no zeros outside this locus, since $M_1\backslash E_1$ consists of extensions $0\to \mathcal{O}_C\to E\to\Lambda\to 0$ together with a (constant) section $\phi\in H^0(C,\mathcal{O}_C)$ \cite[3.1]{thaddeus}. Since $Z$ has codimension $2$, we have a Koszul resolution
\begin{align}\label{resolution Z}
\left[\wedge^2 F^\vee\to F^\vee \xrightarrow{\tilde{\phi}}\mathcal{O}_{C\times M_1}\right]\xrightarrow{\sim} \mathcal{O}_Z.
\end{align}

\begin{lemma}\label{RGamma(lambda-1)=0}
$R\Gamma_{M_1} (\Lambda_M^{-1})=0$.
\end{lemma}

\begin{proof}
Recall $\Lambda_M^{-1}=\mathcal{O}_{M_1}(H-E_1)$. We have an exact sequence
\begin{align*}%\label{s.e.s. for lambda inverse}
0\to \mathcal{O}_{M_1}(H-E_1)\to\mathcal{O}_{M_1}(H)\to\mathcal{O}_{E_1}(H)\to 0,
\end{align*}
so it suffices to show that $j^*:H^i(M_1,\cO_{M_1}(H))\xrightarrow{\sim}H^i(E_1,\cO_{E_1}(H))$ for every $i$, where $j:E_1\hookrightarrow M_1$ is the inclusion. For each $i$, consider the commutative diagram
\begin{equation}\label{some diagram}
\begin{tikzcd}
& H^i(M_1,\cO_{M_1}(H)) \arrow[r,"j^*"] & H^i(E_1,\cO_{E_1}(H))   \\
& H^i(M_0,\cO_{M_0}(H)) \arrow[u,"\pi^*"]\arrow[r,"\imath^*"] & H^i(C,\cO_{C}(H))\arrow[u,"q^*"]
\end{tikzcd}
\end{equation}
where $\imath:C\hookrightarrow M_0=\bP^{d+g-2}$ is the inclusion, $\pi:M_1=\Bl_CM_0\to M_0$ is the blow-up along $C$, and $q=\pi|_{E_1}:E_1\to C$, which is a $\bP^r$-bundle. Hence, both vertical arrows in \eqref{some diagram} are isomorphisms. Indeed, these pullbacks are fully faithful at the level of derived categories. Moreover, $\imath:C\hookrightarrow M_0$ is the embedding by the complete linear system $|\omega_C\otimes\Lambda|$ \cite[3.4]{thaddeus}. Therefore, $\cO_C(H)\simeq \omega_C\otimes\Lambda$ and $\imath^*:H^0(M_0,\cO_{M_0}(H))\to H^0(C,\cO_{C}(H))$ is an isomorphism. For $i>0$, $H^i(M_0,\cO_{M_0}(H))=0$ because $M_0$ is a projective space. On the other hand, since Since $\deg \omega_C\otimes\Lambda >\deg \omega_C$, we also have $H^i(C,\cO_{C}(H))=0$ for $i>0$. In summary, the two vertical maps and the lower horizontal map in the commutative diagram are isomorphisms for all $i$. Hence, the same holds for the upper horizontal map.
%Since the embedding $$C\xhookrightarrow{|\omega_C\otimes \Lambda|}M_0=\mathbb{P}^{d+g-2}$$ is given by a complete linear system \cite[3.4]{thaddeus}, the image of $C$ is not contained in any hyperplane and thus $H^0(M_1,\mathcal{O}_{M_1}(H-E_1))=0$. Now use the exact sequence
%\begin{align}\label{s.e.s. for lambda inverse}
%0\to \mathcal{O}_{M_1}(H-E_1)\to\mathcal{O}_{M_1}(H)\to\mathcal{O}_{E_1}(H)\to 0.
%\end{align}
%Observe that $H^p(M_1,\mathcal{O}_{E_1}(H))=H^p(C,\omega_C\otimes\Lambda)$, because $C\hookrightarrow M_0$ is given by $|\omega_C\otimes\Lambda|$. Since $\deg \omega_C\otimes\Lambda >\deg \omega_C$, we have $H^1(C,\omega_C\otimes\Lambda)=0$. On the other hand, $H^0(M_1,\mathcal{O}_{M_1}(H))=\mathbb{C}^{d+g-1}$ and $H^{>0}(M_1,\mathcal{O}_{M_1}(H))=0$. Taking a long exact sequence in cohomology from (\ref{s.e.s. for lambda inverse}) we get
%\begin{align}\label{short}
%0\to H^0(M_1,\mathcal{O}_{M_1}(H))\to H^0(C,\omega_C\otimes\Lambda)\to H^1(M_1,\Lambda_M^{-1})\to 0
%\end{align}
%and $H^p(M_1,\Lambda_M^{-1})=0$ for $p\neq 1$. By Riemann-Roch, $H^0(C,\omega_C\otimes\Lambda)=\mathbb{C}^{d+g-1}$, so the first map in (\ref{short}) is an isomorphism and $H^1(M_1,\Lambda_M^{-1})=0$ as well.  
\end{proof}

\begin{lemma}\label{acyclicty of Fx, Fx*}
Let $x\in C$. Then $R\Gamma_{M_1}(F_x^\vee)=0$, while $R\Gamma_{M_1} (F_x)=\mathbb{C}$, with $H^0(M_1,F_x)=\mathbb{C}$ given by restriction of the universal section $\tilde{\phi}$ of $F$ to $\{x\}\times M_1$.
\end{lemma}

\begin{proof}
Consider the resolution (\ref{resolution Z}) and restrict to $\{x\}\times M_1$ to get
\begin{align}\label{resolution O}
\left[\Lambda_M^{-1}\to F_x^\vee\to\mathcal{O}_{M_1} \right]\xrightarrow{\sim}\mathcal{O}_{\mathbb{P}_x^r}
\end{align}
where $\mathbb{P}_x^r=M_0(\Lambda(-2x))$ is the fiber over $x\in C\subset M_0$ along the blow-up $\pi:M_1\to M_0$. We twist by $\Lambda_M=\mathcal{O}_{M_1}(E_1-H)$ to get
\begin{align}\label{resolution O(-1)}
\left[\mathcal{O}_{M_1}\xrightarrow{\tilde{\phi}} F_x\to\Lambda_{M}\right]\xrightarrow{\sim} \mathcal{O}_{\mathbb{P}_x^r}(-1),
\end{align}
using that $F_x^\vee\otimes\Lambda_M= F^\vee_x \otimes(\wedge^2F_x)\simeq F_x$ and that $\mathcal{O}_{M_1}(H)$ restricts trivially to $\mathcal{O}_{\mathbb{P}_x^r}$
(see Lemma~\ref{BasicKoszul} for a generalization of \eqref{resolution O} and \eqref{resolution O(-1)}). It is well-known that $R\Gamma(\cO_{\bP^r_x}(-1))=0$. By Lemma 4.2, we also have $R\Gamma (\Lambda_M)=0$. Hence, by \eqref{resolution O(-1)}, $\Tilde{\phi}$ induces an isomorphism $R\Gamma (\cO_{M_1}) \simeq R\Gamma (F_x)$. As $M_1$ is a blow up of a projective space along a smooth center, we get $R\Gamma (F_x) \simeq R\Gamma (\cO_{M_1}) \simeq \mathbb{C}$, with $H^0(M_1,F_x)=\mathbb{C}$ given by restriction of $\Tilde{\phi}$ to $\{x\}\times M_1$.

%Since the right hand side of \eqref{resolution O(-1)} is $\Gamma$-acyclic we see that, by hypercohomology spectral sequence and Lemma~\ref{O(-kH+lE)}, $H^0(M_1,F_x)=\mathbb{C}$, $H^{k}(M_1,F_x)=0$ for $k>0$. %In other words, $R\Gamma (F_x)=\mathbb{C}$. Further, the isomorphism $H^0(M_1,\mathcal{O}_{M_1})\xrightarrow{\tilde{\phi}} H^0(M_1,F_x)$ from (\ref{resolution O(-1)}) is provided  by the universal section.

To show that $R\Gamma_{M_1}(F_x^\vee)=0$, we apply $R\Gamma$ to (\ref{resolution O}). We already know $R\Gamma_{M_1}(\Lambda_M^{-1})=0$ by Lemma \ref{RGamma(lambda-1)=0}, so it suffices to show that that the restriction map $H^i(M_i,\cO_{M_i}) \to H^i(\bP^r_x,\cO_{\bP^r_x})$ is an isomorphism for every $i$. For $i > 0$, both vector spaces vanish, because we have a projective space and a blow-up of a projective space. For $i = 0$, we have an isomorphism of one-dimensional
vector spaces because this is just restriction of constant sections.
%We know that $R\Gamma_{M_1} (\mathcal{O}_{M_1})=R\Gamma_{\mathbb{P}^r} (\mathcal{O}_{\mathbb{P}^r})=\mathbb{C}$, while $R\Gamma_{M_1}(\Lambda_M^{-1})=0$ by Lemma \ref{RGamma(lambda-1)=0}. Then the claim will be proved if we show that $H^0(M_1,F_x^\vee)=0$, as this would also imply that $H^p(M_1,F_x^\vee)$ vanishes for $p>0$. Any global section $s\in H^0(M_1,F_x^\vee)$, composed with $F_x^\vee\to \mathcal{O}_{M_1}$ gives a constant section $\mathcal{O}_{M_1}\to\mathcal{O}_{M_1}$ vanishing along the locus $Z$, hence identically~$0$. By exactness of $0\to \Lambda_{M}^{-1}\to F_x^\vee\to\mathcal{O}_{M_1}$, the section $s:\mathcal{O}_{M_1}\to F_x^\vee$ must lift to a section $\mathcal{O}_{M_1}\to \Lambda_M^{-1}$. But $\Lambda_M^{-1}$ has no global sections by Lemma \ref{RGamma(lambda-1)=0}.
\end{proof}

\begin{proof}[Proof of Theorem \ref{Poincare is fully faithful on M1}]

By Bondal-Orlov's criterion \cite{bondal-orlov}, in order to show full faithfulness of $\Phi_F$ we only need to consider the sheaves $\Phi_F(\mathcal{O}_x)=F_x$ for closed points $x\in C$.
On the other hand, consider the functor $\Psi$ from Orlov's blow-up formula, with Fourier--Mukai kernel $\mathcal{O}_Z(E_1)$, $Z=C\times_C E_1$. We can compute $\Psi(\mathcal{O}_x)=\Phi_{\mathcal{O}_Z(E_1)}(\mathcal{O}_x)$ for a point $x\in C$ using \eqref{resolution Z} as follows. As before, let $\mathbb{P}_x^r=M_0(\Lambda(-2x))$ denote the fiber over $x\in C\subset M_0$ along the blow-up. The fact that $\mathcal{O}_{M_1}(H)$ restricts trivially to this fiber implies that both $\Lambda_M$ and $\mathcal{O}_{M_1}(E_1)$ restrict to $\mathcal{O}_{\mathbb{P}_x^r}(-1)$ there. Now we restrict \eqref{resolution Z} to $\{x\}\times M_1$ and twist it by $\Lambda_M$ to get $\Phi_{\mathcal{O}_Z(E_1)}(\mathcal{O}_x)\simeq\left[\mathcal{O}_{M_1}\to F_x\to\Lambda_{M}\right]
\simeq\mathcal{O}_{\mathbb{P}_x^r}(-1)$,
as in (\ref{resolution O(-1)}).
Since we already know that $\Psi$ is fully faithful, we have
\begin{align}\label{RGamma orlov}
\Hom_{D^b(M_1)}\bigl(\Psi(\mathcal{O}_x),\Psi(\mathcal{O}_y)[k]\bigr)=\begin{cases}
0 \quad &\text{if }x\neq y\\
0 \quad &\text{if }x=y \text{ and }k\neq 0,1\\
\mathbb{C} &\text{if }x=y \text{ and }k= 0,1.
\end{cases}
\end{align}
But $R\Hom_{D^b(M_1)}\bigl(\Psi(\mathcal{O}_x),\Psi(\mathcal{O}_y)\bigr)\simeq R\Gamma \circ R\cH\!om\bigl(\Psi(\mathcal{O}_x),\Psi(\mathcal{O}_y)\bigr)$ can also be obtained as follows: take $R\cH\!om\bigl(\Psi(\mathcal{O}_x),\Psi(\mathcal{O}_y)\bigr)\simeq \Psi(\mathcal{O}_x)^\vee\otimes^L\Psi(\mathcal{O}_y)$ as an inner tensor product obtained from the double complex
\begin{equation} \label{double complex Fx}
\begin{tikzcd}[cramped,sep=scriptsize]
\mathcal{O}_{M_1}\arrow[r] & F_x^\vee\otimes\Lambda_M \arrow[r] &\Lambda_M\\
\Lambda_M^{-1}\otimes F_y \arrow[r]\arrow[u] &F_x^\vee\otimes F_y \arrow[r]\arrow[u] &F_y \arrow[u]\\
\Lambda_M^{-1}\arrow[r]\arrow[u] & F_x^\vee \arrow[u]\arrow[r] &\mathcal{O}_{M_1}\arrow[u]
\end{tikzcd}
\end{equation}
which produces the total complex 
\begin{gather}\label{long complex Fx}
\begin{aligned}
&\left[\Lambda_M^{-1}\to F_x^\vee\oplus F_y^\vee\to\mathcal{O}_{M_1}^{\oplus 2}\oplus(F_x^\vee\otimes F_y)\to F_x\oplus F_y \to \Lambda_M\right] \\
\simeq &\phantom{[}\Psi(\mathcal{O}_x)^\vee\otimes^L\Psi(\mathcal{O}_y),
\end{aligned}
\end{gather}
again using $F_x= F_x^\vee\otimes \Lambda_M$. Recall that our descriptions of $\Psi(\mathcal{O}_y)$ and $\Psi(\mathcal{O}_x)^\vee$ were obtained from the Koszul resolution \eqref{resolution O(-1)} and its dual. In particular, the maps $\cO_{M_1}\to F_x^\vee\otimes\Lambda_M=F_x$ and $\cO_{M_1}\to F_y$ appearing in \eqref{double complex Fx} correspond to the restriction of the universal section $\Tilde{\phi}$ to $\{x\}\times M_1$ and $\{y\}\times M_1$, respectively.

The hypercohomology $R\Gamma$ of (\ref{long complex Fx}) can be computed by taking the spectral sequence with first page $E_1^{p,q}=H^q(X,\cF^p) \Rightarrow H^{p+q} (X,\cF^{\bullet})$. On the other hand, we know that $R\Gamma$ of this complex is given by (\ref{RGamma orlov}). We will combine these to show that
\begin{align}\label{abzdhadjatej}
R\Gamma (F_x^\vee\otimes F_y)=\begin{cases}
0 \quad &\text{if }x\neq y\\
\mathbb{C}\oplus\mathbb{C}[-1] &\text{if }x=y.
\end{cases}
\end{align}

By Lemma \ref{O(-kH+lE)}, $R\Gamma (\Lambda_M)=0$, and by Lemma \ref{RGamma(lambda-1)=0} $R\Gamma (\Lambda_M^{-1})=0$. Also, Lemma \ref{acyclicty of Fx, Fx*} computes hypercohomology of both $F_x$ and $F_x^\vee$. Summing up, applying $R\Gamma$ to (\ref{long complex Fx}) yields a spectral sequence $E_1^{p,q}$ of the form
\begin{equation*}
\begin{tikzcd}[cramped, sep=scriptsize]
\vdots & \vdots & \vdots & \vdots & \vdots\phantom{.} \\
0\arrow[r] &0\arrow[r] & H^1(F_x^\vee\otimes F_y) \arrow[r] & 0\arrow[r] & 0\phantom{.}\\
0\arrow[r] &0\arrow[r] & H^0(\mathcal{O}_{M_1})^{\oplus 2}\oplus H^0 (F_x^\vee\otimes F_y) \arrow[r] & H^0(F_x) \oplus H^0(F_y) \arrow[r] & 0,
\end{tikzcd}
\end{equation*}
where the map $H^0(\mathcal{O}_{M_1})^{\oplus 2}\to H^0(F_x) \oplus H^0(F_y)$ is the isomorphism $\mathbb{C}^2\xrightarrow{\sim}\mathbb{C}^2$ given by the universal section in each coordinate, by Lemma \ref{acyclicty of Fx, Fx*} and the discussion above. Since this spectral sequence converges to (\ref{RGamma orlov}), we obtain \eqref{abzdhadjatej}.
\end{proof}

\section{Acyclicity of powers of $\Lambda_M^\vee$}\label{,jshBDFjhsbf}

The goal of the present section is to prove the following generalization of Lemma~\ref{RGamma(lambda-1)=0}:

\begin{theorem}\label{full vanishing of Lambda*}
Suppose $2<d\leq 2g+1$ and $1\le k\leq l\leq v$. Then
$$
R\Gamma_{M_l(d)}(\Lambda_M^{-k})=0.
$$
\end{theorem}

$\Gamma$-acyclicity of these negative powers of $\Lambda_M$ will be crucial for the cohomology computations in the upcoming sections.

\begin{lemma}\label{noglobalsections}
Under the assumptions of Theorem~\ref{full vanishing of Lambda*},
$
H^0(M_l(d),\Lambda_M^{-k})=0.
$
\end{lemma}

\begin{proof}
Since $M_l$ is isomorphic to $M_1$ in codimension~$1$, it suffices to prove that
$
H^0(M_1,\Lambda_M^{-k})=
H^0(M_1,kH-kE_1)=0.
$
Recall that $M_1$ is the blow-up of $\bP^r$ in $C$ embedded by a complete linear system of $K_C+\Lambda$, $r=d+g-2$, $E_1$ is the exceptional divisor and $H$ is a hyperplane divisor. The claim is that there is no hypersurface $D\subset \bP^{r}$ of degree $k$ that vanishes along $C$ with multiplicity $\ge k$.
We argue by contradiction.
Choose $r+1$ points $p_1,\ldots,p_{r+1}\in C$ in linearly general position.
Then $D$ vanishes at these points with multiplicity $\ge k$. Let $R$ be a rational normal curve passing through   $p_1,\ldots,p_{r+1}$. Let $\tilde R$ and $\tilde D$ be the proper transforms of $R$ and $D$ in $\Bl_{p_1,\ldots,p_{r+1}}\bP^r$. Then $\tilde D\cdot\tilde R\le kr-k(r+1)<0.$
It follows that $R\subset D$. But we can choose $R$ passing through a general point of $\bP^r$, which is a contradiction.
\end{proof}

\begin{lemma}\label{windowforward}
Under the assumptions of Theorem~\ref{full vanishing of Lambda*},
if 
$
R\Gamma_{M_k(d)}(\Lambda_M^{-k})=0,
$
then
$
R\Gamma_{M_l(d)}(\Lambda_M^{-k})=0.
$
\end{lemma}

\begin{proof}
By Theorem~\ref{weights computations of all bundles}, in the wall between $M_{l-1}$ and $M_l$, $\Lambda_M^{-k}$ descends from an object of weight $k$, with $-\eta_-<k<\eta_+$ when $k<l\leq v$, that is, $1+2l-d-g<k<l$ for $l$ in that range. This way, $0=R\Gamma_{M_k}(\Lambda_M^{-k})=R\Gamma_{M_l}(\Lambda_M^{-k})$ for $l\geq k$ by Theorem~\ref{windows wall-crossing}
\end{proof}

\begin{definition}
For $0\leq\alpha\leq i$, we introduce the following loci:
\begin{align*}
    E_i^\alpha &:= \{(E,s)\mid Z(s)\subset C \text{ has degree }\geq\alpha\}\subset M_i,\\
    \cD_i^\alpha &:=\{(D,E,s)\mid s|_D=0\}\subset \Sym^\alpha C\times M_i,\\
    R_i^\alpha&:=\{(D,E,s)\mid s|_D=0\text{ and }Z(s)\text{ has degree }\geq\alpha+1\}\subset \cD_i^\alpha,
\end{align*}
where $Z(s)$ denotes the zero locus subscheme of the section $s$.
\end{definition}

Note that $E_i^i$ is precisely $\mathbb{P}W_i^+$ \cite[proof of 3.2]{thaddeus}, while $E_i^1=E_i$ is the proper transform of $E_1$ under the birational equivalence given by (\ref{diagram Mi}). Recall $\cO(E_i)=\cO_i(1,-1)$ according to Definition \ref{notation O(m,n)}.
For a divisor $D\in \Sym^\alpha C$, we observe that the fiber $(\cD_i^\alpha)_D$ along the projection $\Sym^\alpha C\times M_{i}\to \Sym^\alpha C$ is isomorphic to $M_{i-\alpha}(\Lambda (-2D))$, see Remark~\ref{restriction of O(m,n)} or
\cite[1.9]{thaddeus}. Similarly, $(R_i^\alpha)_D\simeq E_{i-\alpha}(\Lambda(-2D))$. In particular, $\cD_i^\alpha$ is smooth, and we have a diagram
\begin{equation}\label{normalization of Ei}
\begin{tikzcd}
& R_i^\alpha \arrow[r,hook]\arrow[d] & \cD_i^\alpha \arrow[d,"\nu"]  \\
& E_i^{\alpha+1}\arrow[r,hook] & E_i^\alpha
\end{tikzcd}
\end{equation}
%\begin{equation}\label{normalization of Ei}
%    \cD_i^\alpha\xrightarrow[]{\nu} E_i^\alpha
%\end{equation}
where $\nu$ is the normalization morphism.

\begin{lemma}\label{commutative diagram conductor}
We have the following commutative diagram
\begin{equation}\label{conductor}
    \begin{tikzcd}
    0 \arrow[r] & \nu_*\cO_{D_i^\alpha}(-R_i^\alpha) \arrow[r] &\nu_*\cO_{D_i^\alpha}\arrow[r] &\nu_*\cO_{R_i^\alpha} \arrow[r] &0\\
    0\arrow[r] &\mathcal{I}_{E_i^{\alpha+1}}\arrow{u}[sloped]{\sim}\arrow[r] &\cO_{E_i^\alpha}\arrow[u,hook]\arrow[r,"\beta"] &\cO_{E_i^{\alpha+1}}\arrow[u,hook]\arrow[r] &0
    \end{tikzcd}
\end{equation}
where $\mathcal{I}_{E_i^{\alpha+1}}\simeq \nu_*\cO_{D_i^\alpha}(-R_i^\alpha)$ is the conductor sheaf of the normalization (\ref{normalization of Ei}) and $R_i^\alpha$ (resp.~$E_i^{\alpha+1}$)
is a conductor subscheme in $\cD_i^\alpha$ (resp.~$E_i^\alpha$).
\end{lemma}

\begin{proof}
From the flipping diagram \eqref{diagram Mi},
$E_\alpha^\alpha\subset M_\alpha$ is the projective bundle $\bP W_\alpha^+$ and $E_{\alpha+1}^{\alpha}
\subset M_{\alpha+1}$ is isomorphic to $E_\alpha^\alpha$ away from $E_{\alpha+1}^{\alpha+1}\simeq\bP W_{\alpha+1}^+$.

\begin{claim}
$E_{\alpha+1}^{\alpha}$ has a multicross singularity 
generically along $E_{\alpha+1}^{\alpha+1}$
(concretely, this means that a general section of $E_{\alpha+1}^{\alpha}$ that intersects  $E_{\alpha+1}^{\alpha+1}$
in a point is \'etale locally isomorphic to the union of coordinate axes in $\bA^{\alpha+1}$).
\end{claim}

Given the claim, and since multicross singularities are semi-normal \cite{LV},  $E_{\alpha+1}^\alpha$ has semi-normal singularities in codimension~$1$. For $i>\alpha+1$, $E_i^\alpha$ is isomorphic to $E_{\alpha+1}^\alpha$ in codimension~$2$, and so also has semi-normal singularities in codimension~$1$.
Next we argue by induction on $\alpha$ that $\cD_i^\alpha \to E_i^\alpha$ has reduced conductor subschemes $E_{i}^{\alpha+1}\subset E_i^\alpha$ and $R_i^\alpha \subset \cD_i^\alpha$ and $E_{i}^{\alpha+1}$ is Cohen-Macaulay and semi-normal, and in particular that we have a commutative diagram \eqref{conductor}.

Indeed, $E^1_i\subset M_i$ is Cohen--Macaulay as a hypersurface in a smooth variety. Suppose $E_i^\alpha$ is Cohen--Macaulay. Since it is semi-normal in codimension~$1$ by the above, it is semi-normal everywhere \cite[Corollary 2.7]{greco}. Therefore, its conductor subschemes in $E_i^\alpha$ and $\cD_i^\alpha$ are both reduced  \cite[Lemma 1.3]{traverso} and all of their associated primes have height~$1$ in $E_i^\alpha$ and $\cD_i^\alpha$, respectively \cite[Lemma 7.4]{greco}. It follows that these conductor subschemes are equal to $E_i^{\alpha+1}$ and 
$R_i^\alpha$, respectively. Finally,
$R_i^\alpha\subset \cD_i^\alpha$ is Cohen--Macaulay as a hypersurface in a smooth variety and therefore
$E_i^{\alpha+1}\subset E_i^{\alpha}$ is also Cohen--Macaulay \cite[Theorem 2.2]{roberts}, and we can proceed with induction.

It remains to prove the claim.
We analyze the flipping diagram \eqref{diagram Mi} between the spaces $M_\alpha$ and $M_{\alpha+1}$, where 
$M_\alpha$ contains projective bundles 
$\bP W_{\alpha+1}^-$ (over $\Sym^{\alpha+1}C$) and $\bP W_{\alpha}^+\simeq E_\alpha^\alpha$ (over $\Sym^{\alpha}C$) of dimensions $2\alpha+1$ and $d+g-2-\alpha$, respectively. 
What is their intersection over a point $D'\in\Sym^{\alpha+1}C$, for simplicity a reduced sum of points?
By \cite[3.3]{thaddeus}, $\bP W_{\alpha+1}^-$ parametrizes pairs $(E,\phi)$ that appear in extensions $$0\to L\to E\to\Lambda\otimes L^{-1}\to0$$
with $\deg L=d-\alpha-1$ and $\phi\not\in H^0(L)$. Projecting $\phi$ to $\Lambda\otimes L^{-1}$ gives a non-zero vector $\gamma\in H^0(\Lambda\otimes L^{-1})$ with $Z(\gamma)=D'$, so that
$\Lambda\otimes L^{-1}=\cO(D')$, where $\deg D'=\alpha+1$ (this gives the map from $\bP W_{\alpha+1}^-$ to $\Sym^{\alpha+1}C$). Moreover, at $D'$ the section lifts to a section of $\cO_{D'}\otimes L\simeq \cO_{D'}\otimes \Lambda(-D')$, and this vector $p\in H^0(\cO_{D'}\otimes \Lambda(-D'))$ 
(determined uniquely up to a scalar) determines $(E,\phi)$ uniquely \cite[3.3]{thaddeus}.

The same pair $(E,\phi)$ belongs to $\bP W_\alpha^+$ if it can be given by an extension
$$0\to \cO(D)\to E\to\Lambda(-D)\to0$$
with $\phi\in H^0(\cO(D))$ and $\deg D=\alpha$ \cite[3.2]{thaddeus}.
Since $\phi$ vanishes at $D$ and its image in $\cO(D')$ vanishes at $D'$, we have $D\subset D'$.
Since we assume that $D'$ is a reduced divisor, there are exactly $\alpha+1$ choices for $D$. 
Since $p$ has to vanish at points of $D\subset D'$, there is exactly one vector 
$p\in H^0(\cO_{D'}\otimes \Lambda(-D'))$ 
(up to a multiple) that works for a given choice of $D$. Moreover, in this way we get a basis of 
$H^0(\cO_{D'}\otimes \Lambda(-D'))\simeq\bC^{\alpha+1}$.
It follows that, over $D'\in\Sym^{\alpha+1}C$, 
$\bP W_{\alpha+1}^-$ and $\bP W_{\alpha}^+\simeq E_\alpha^\alpha$
intersect in $\alpha+1$ reduced points which form a basis of the projective space
$(\bP W_{\alpha+1}^-)_{D'}\simeq\bP^\alpha$.

The strict transform of $\bP W_\alpha^+$ in $M_{\alpha+1}$ is $E_{\alpha+1}^\alpha$, which contains the bundle $\bP W_{\alpha+1}^+$ of dimension $d+g-3-\alpha$ (the flipped locus). After the flip, linearly independent intersection points in
$(\bP W_{\alpha+1}^-)_{D'}\cap \bP W_\alpha^+$ become linearly independent normal directions of branches of $E_{\alpha+1}^\alpha$ along $\bP W_{\alpha+1}^+$, i.e. $E_{\alpha+1}^\alpha$ has a multicross singularity in codimension~$1$, as claimed. We illustrate the geometry of $M_{\alpha}$, $M_{\alpha+1}$ and the common resolution $\tilde M_{\alpha+1}$ in Figure~\ref{FlipPicture}.
\end{proof}

\begin{figure}[htbp]
\includegraphics[width=4in]{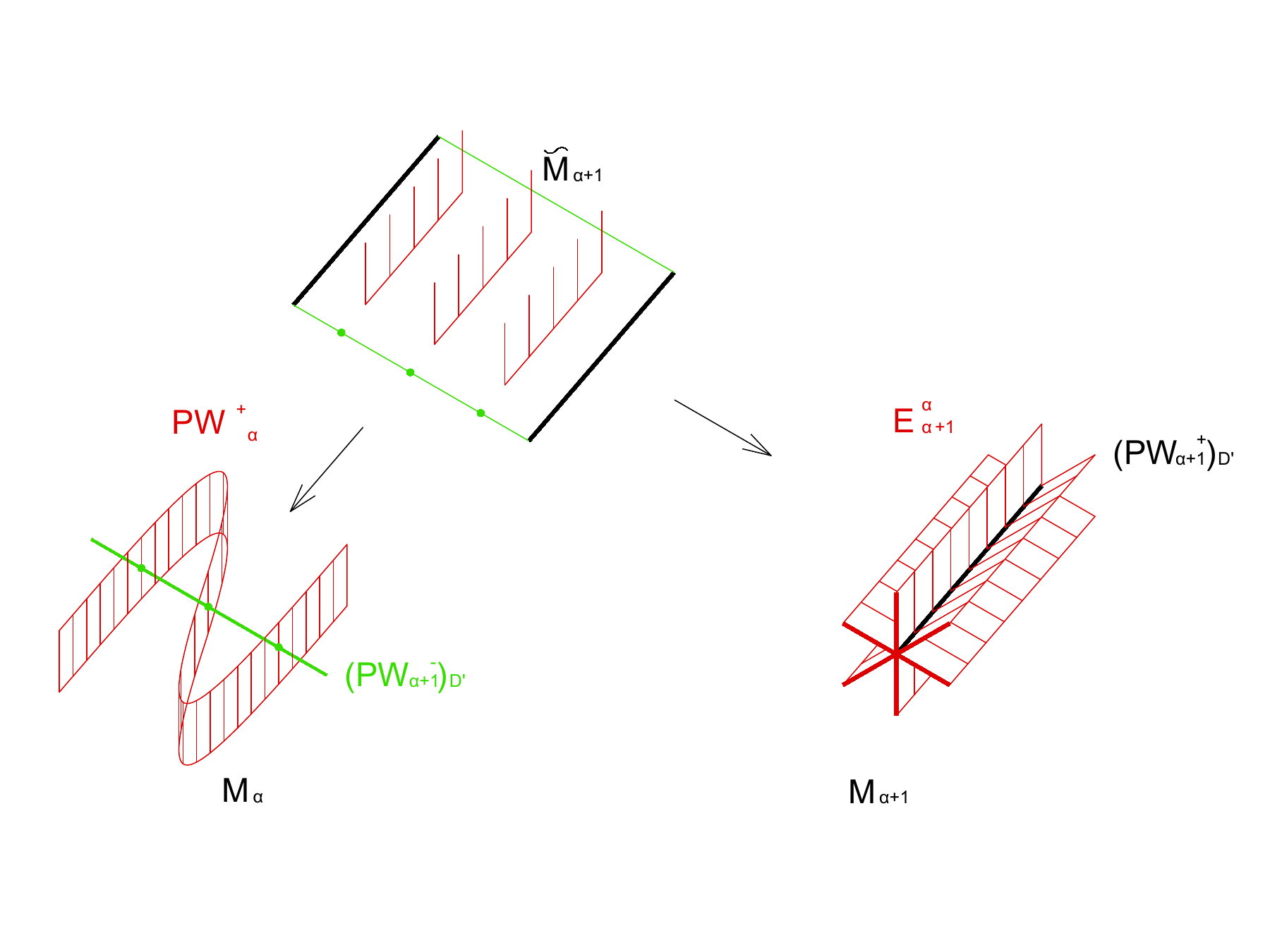}
\caption{Common resolution $\tilde M_{\alpha+1}$ of $M_{\alpha}$ and $M_{\alpha+1}$.}
\label{FlipPicture}
\end{figure}

\begin{corollary}\label{induction corollary}
If the claim of Theorem~\ref{full vanishing of Lambda*}
is proved for  $1\leq k=l\leq i-1$, then, for $1\leq\alpha\leq i-1$, $R\Gamma_{M_i} (\cO_{E_i^\alpha}(1,i-1))\simeq R\Gamma_{M_i} (\cO_{E_i^{\alpha+1}}(1,i-1))$ via $R\Gamma(\beta)$.
\end{corollary}

\begin{proof}
Twisting by $\cO_i(1,i-1)$ and applying $R\Gamma$ to the bottom sequence in (\ref{conductor}), we see that it suffices to show $\mathcal{I}_{E_i^{\alpha+1}}(1,i-1)\simeq\nu_*\cO_{\cD_i^\alpha}(-R_i^\alpha)(1,i-1)$ is $\Gamma$-acyclic. But $\nu$ is a finite map, so this is equivalent to $\Gamma$-acyclicity of $\cO_{\cD_i^\alpha}(-R_i^\alpha)(1,i-1)$. Using the Leray spectral sequence for the fibration $p:\cD_i^\alpha\to \Sym^\alpha C$, it suffices to prove that 
$R\Gamma \bigl(\cO_{\cD_i^\alpha,D}(-R_{i,D}^\alpha)(1,i-1)\bigr)=0.$ 
Under the isomorphism $(\cD_i^\alpha)_D\simeq M_{i-\alpha}(\Lambda(-2D))$, $R_i^\alpha\subset \Sym^\alpha C\times M_i$ restricts to $E_{i-\alpha}^1$ on $M_{i-\alpha}(\Lambda(-2D))$, while $\cO_i (m,n)$ on $M_i(\Lambda)$ restricts to $\cO(m,n-m\alpha)$ on $M_{i-\alpha}(\Lambda(-2D))$ (cf. Remark \ref{restriction of O(m,n)}). Therefore,
$$
R\Gamma_{M_i(d)} \bigl(\cO_{\cD_i^\alpha,D}(-R_{i,D}^\alpha)(1,i-1)\bigr)=R\Gamma_{M_{i-\alpha}(d-2\alpha)}(\Lambda_M^{\alpha-i})
$$
which is zero by hypothesis. \end{proof}

\begin{lemma}\label{cohomology of (1,i-1)}
Suppose $d\leq 2g+1$. Then for $1\leq i\leq d+1-g$, $i\le v$ we have $H^p\bigl(M_i(d),\cO_i(1,i-1)\bigr)=0$ for any $p>0$.
\end{lemma}

\begin{proof}
Recall that $
\omega_{M_k}=\cO_{M_k}(-3,4-d-g)
$ for every $1 \leq k \leq v$ (see \cite[6.1]{thaddeus}). First, we see that there is some $i \leq k \leq v$ such that the bundle $\cO_{M_k}(1,i-1)\otimes \omega_{M_k}^{-1}=\cO_{M_k}(4,d+g+i-5)$ is big and nef. By the description of the ample cones in Remark \ref{ample cone on Mi}, it suffices to check that $(4,d+g+i-5) \in \bR^2$ lies in the closed cone bounded below by the ray through $(1,i-1)$ and above by the ray through $(2,d-2)$. Considering the slopes, this is equivalent to
$
i-1\leq \frac{d+g+i-5}{4}\leq \frac{d-2}{2}.
$
The inequality on the left is equivalent to $3i\leq d+g-1$, which is guaranteed by the fact that $i\leq v=\lfloor(d-1)/2\rfloor$ and $d\leq 2g+1$. The other inequality is  equivalent to $i\leq d+1-g$, which is given as a hypothesis. Therefore, there is some $k\geq i$, $k\le v$ such that $\cO_{M_k}(1,i-1)\otimes \omega_{M_k}^{-1}$ is big and nef. By  the Kawamata--Viehweg vanishing theorem, $H^p(M_k,\cO_k(1,i-1))=0$ for $p>0$.

Now, we claim that in fact
\begin{equation}\label{wall-crossing (1,i-1)}
R\Gamma_{M_i}(\cO_i(1,i-1))=R\Gamma_{M_{i+1}}(\cO_{i+1}(1,i-1))=\ldots=R\Gamma_{M_k}(\cO_{k}(1,i-1)).
\end{equation}
Indeed, in the wall-crossing between $M_{l-1}$ and $M_l$, there are windows of width $\eta_+=l$ and $\eta_-=d+g-1-2l$ and $\cO_l(1,i-1)$, $\cO_{l-1}(1,i-1)$ both descend from the same object, that has $\lambda$-weight $i-l$ (see Proposition \ref{weights computations of all bundles} and Remark \ref{weights of line bundles}). By Theorem~\ref{windows wall-crossing}, we will have $R\Gamma_{M_{l-1}}(\cO_{l-1}(1,i-1))=R\Gamma_{M_l}(\cO_l(1,i-1))$ whenever 
\begin{equation}\label{ineq i-l}
    1+2l-d-g<i-l<l.
\end{equation}
But (\ref{ineq i-l}) holds for any $i< l\leq k$, because then $i<2l$, while $3l\leq 3(d-1)/2<i+d+g-1$ provided $d\leq 2g+1$. Therefore, (\ref{wall-crossing (1,i-1)}) holds and in particular $H^p(M_i,\cO_i(1,i-1))=0$ for $p>0$.
\end{proof}

\begin{remark}
Suppose  that $d\leq 2g+1$. Then (\ref{ineq i-l}) holds for $l\in (i/2,v]$, and the same reasoning shows that 
$
R\Gamma_{M_i}(\cO_i(1,i-1))=R\Gamma_{M_l}(\cO_l(1,i-1))
$
for every $\lfloor i/2\rfloor \leq l\leq v$. In particular, under the same hypotheses of Lemma \ref{cohomology of (1,i-1)},  $\cO_l(1,i-1)$ has no higher cohomology whenever $\lfloor i/2 \rfloor\leq l \leq v$.
\end{remark}

\begin{definition}
    Let $L_i$ be the line bundle on $\Sym^iC$ defined by
    \begin{equation}\label{Li}
        L_i=\det\nolimits^{-1}\pi_!\Lambda (-\Delta)\otimes\det\nolimits^{-1}\pi_!\cO(\Delta),
    \end{equation}
    where $\Delta\subset\Sym^iC\times C$ is the universal divisor, cf. \cite[6.5]{thaddeus}.
    To emphasize  the degree $d$, sometimes we denote this line bundle by $L_i(d)$.
\end{definition}

\begin{lemma}\label{KodairaOnSym}
$H^p(\Sym^iC,L_i(d))=0$
if $p>0$, $1\le i\le d-g$.
\end{lemma}

\begin{proof}
By \cite[7.5]{thaddeus} (see also \cite{macdonald}), and mixing notation for line bundles and divisors,
\begin{equation}\label{divisorsSym}
L_i(d)=(d-2i)\eta+2\sigma\quad\hbox{\rm and}\quad K_{\Sym^iC}=(g-i-1)\eta+\sigma,
\end{equation}
where 
$\eta=p_0+\Sym^{i-1}C\subset \Sym^i C$
is an ample divisor for any fixed $p_0\in C$
and $\sigma\subset \Sym^i C$ is a pull-back of a theta-divisor via the Abel--Jacobi map, in particular $\sigma$ is nef. It follows that 
$
L_i(d)-K_{\Sym^iC}=(d-i-g+1)\eta+\sigma
$
is ample if 
$i\le d-g$ and the result follows by Kodaira vanishing theorem.
\end{proof}

\begin{lemma}\label{base case d}
    Suppose $i+g\leq d\leq 2g+1$. Then $\chi\bigl(M_i(d),\cO_i(1,i-1)\bigr)=\chi(\Sym^iC,L_i(d)).$
\end{lemma}

\begin{proof}
Since $i\leq d-g$, we can use Lemma~\ref{cohomology of (1,i-1)} %\cite[6.2]{thaddeus} 
together with \cite[7.8]{thaddeus} to compute $\chi(\cO_i(1,i-1))=$
%$h^0(\cO_i(1,i-1))$, which is equal to $\chi(\cO_i(1,i-1))=$
\begin{gather}
\begin{aligned}\label{residue Mi}
    =&\underset{t=0}{\Res}\left\{\frac{(1-t^3)^{2i-d-1}(1-t^2)^{2d+1-2i-2g}}{t^{i+1}(1-t)^{d+g-1}}(1-5(1-t)t^2-t^5)^gdt\right\}\\
    =&\underset{t=0}{\Res}\left\{\frac{(1+t)^{2d+1-2i-2g}(1+3t+t^2)^g(1-t)}{t^{i+1}(1+t+t^2)^{d+1-2i}}dt\right\}.
\end{aligned}
\end{gather}
On the other hand, we use Hirzebruch--Riemann--Roch theorem to compute, using the formulas 
$$
\ch (L_i)=e^ {(d-2i)\eta+2\sigma}, \quad\td (\Sym^iC)=\left(\frac{\eta}{1-e^{-\eta}}\right)^{i-g+1}\exp \left(\frac{\sigma}{e^\eta-1}-\frac{\sigma}{\eta}\right)
$$
(see \cite[\S 7]{thaddeus}) and notation from the proof of Lemma~\ref{KodairaOnSym}, that
$$
\chi (L_i)=\underset{\eta=0}{\Res}\left\{\frac{e^{\eta(d-2i)}}{(1-e^{-\eta})^{i-g+1}}\left(2+\frac{1}{e^\eta-1}\right)^gd\eta\right\},
$$
where we have used \cite[7.2]{thaddeus} with
$$
A(\eta)=e^{\eta(d-2i)}\left(\frac{\eta}{1-e^{-\eta}}\right)^{i-g+1},\quad B(\eta)=2+\frac{1}{e^\eta-1}-\frac{1}{\eta}.
$$
If we let $u(\eta)=e^\eta-1$, then $u$ is biholomorphic near $\eta=0$, with $u(0)=0$, $u'(0)=1$, so we can do a change of variables $u=e^\eta-1$, $du=e^\eta d\eta$ to obtain
\begin{align}\label{residue Li}
\chi(L_i)=\underset{u=0}{\Res}\left\{\frac{(1+u)^{d-i-g}(2u+1)^g}{u^{i+1}}du\right\}.
\end{align}
Next, we apply an {\em ad hoc} change of variables 
$$
u=\frac{t}{t^2+t+1},\quad du=\frac{1-t^2}{(t^2+t+1)^2}dt
$$
to (\ref{residue Li}) and we get precisely (\ref{residue Mi}) after some algebraic manipulations. 
\end{proof}

For what follows we need some geometric constructions. Fix a point $p_0\in C$ and consider a subvariety
$M_{i-1}(d-1)\subset M_i(d+1)$ of codimension~$2$ as in Remark~\ref{restriction of O(m,n)}, with $D=p_0$. Let $B$ be the blow-up of $M_i(d+1)$ in $M_{i-1}(d-1)$ with exceptional divisor $\cE$.

Consider the $\bP^1$-bundle $\bP F_{p_0}$ over $M_i(d+1)$ that parametrizes triples $(E,\phi,l)$, where  $\phi$ is a non-zero section of $E$ and $l\subset E_{p_0}$ is a line, subject to the usual stability condition (see Section~\ref{ThaddeusSpaces}) that
for every line subbundle $L\subset E$, one must have
\begin{equation}\label{qrwgwRGwg}
\deg L\le\begin{cases}
    i+\frac{1}{2}\quad & \text{ if }\phi\in H^0(L),\\
d-i+\frac{1}{2} \quad & \text{ if }\phi\notin H^0(L).
\end{cases}
\end{equation}

\begin{lemma}\label{blowup lemma}
With the notation as above, the blow-up $B$ of $M_i(d+1)$ in $M_{i-1}(d-1)$ is isomorphic to the following locus: 
$$Z=\{(E,\phi,l)\,:\,\phi(p_0)\in l\}\subset \bP\!F_{p_0}.$$
\end{lemma}

\begin{proof}
Indeed, the projection of $Z$ onto $M_i(d+1)$ is clearly
an isomorphism outside of $M_{i-1}(d-1)$, since the latter is precisely the locus where $\phi(p_0)=0$.
Over $M_{i-1}(d-1)$, the fiber of this projection is $\bP^1$. By the universal property of the blow-up,
it suffices to check that $Z$ is the blow-up of 
$M_i(d+1)$ in $M_{i-1}(d-1)$ locally near $(E,\phi)\in M_{i-1}(d-1)$, where we can trivialize $F_{p_0}\simeq\cO\oplus\cO$. Its universal section can be written as $s=(a,b)$, where $a,b\in\cO$ is a regular sequence (its vanishing locus is $M_{i-1}(d-1)$ locally near $(E,\phi)$). Then $Z$ is locally given by the equation 
$ay-bx=0$, where $[x:y]$ are homogeneous coordinates of the $\bP^1$-bundle $\bP F_{p_0}$ given by the trivialization $F_{p_0}\simeq\cO\oplus\cO$. Thus $Z$ is indeed isomorphic to the blow-up $B$.
\end{proof}

Now we can prove the main result of this section.

\begin{proof}[Proof of Theorem~\ref{full vanishing of Lambda*}]

By Lemma \ref{windowforward}, it suffices to prove that $R\Gamma_{M_i}(\Lambda^{-i}_
M)$ is zero for every $i = 1,\ldots,v$, which we will do by induction on $i$. The
base case $i = 1$ is Lemma \ref{RGamma(lambda-1)=0}.
%We proceed by induction on $k$. 
%Suppose $R\Gamma_{M_k} (\Lambda^{-k})$ is zero for $1\leq k<i$. 
%By Lemma~\ref{windowforward}, it suffices to show that $R\Gamma_{M_i}(\Lambda^{-i})=0$.
Recall that $\cO_{M_i}(E_i) = \cO_i(1,-1)$. Twist the tautological short exact sequence for $E_i\subset M_i$ by $\cO_i(1,i-1)$ to get 
$$0\to \Lambda_M^{-i}\to \cO_i(1,i-1)\mathop{\to}^\gamma \cO_{E_i}(1,i-1)\to 0.$$
It suffices to prove that $R\Gamma_{M_i}(\cO_i(1,i-1))\simeq R\Gamma_{E_i}(\cO_{E_i}(1,i-1))$ via $R\gamma$.
By the induction hypothesis, we can apply Corollary \ref{induction corollary} to see that 
$$
R\Gamma(\cO_{E_i}(1,i-1))\simeq\ldots\simeq R\Gamma (\cO_{E_i^i}(1,i-1))=R\Gamma (\cO_{\mathbb{P}W_i^+}(1,i-1)).
$$
But $\mathcal{O}_{\mathbb{P}W_i^+}(1,i-1)$ restricts trivially to each fiber of $\mathbb{P}W_i^+$. Arguing as in  \cite[6.5]{thaddeus}, where an analogous statement is proved for $\mathcal{O}_{\mathbb{P}W_i^-}(1,i-1)$ (but using \cite[(3.2)]{thaddeus} instead of \cite[(3.3)]{thaddeus}), the restriction 
$\mathcal{O}_{\mathbb{P}W_i^+}(1,i-1)$ is a pull-back of the line bundle $L_i$  on $\Sym^i C$ defined in (\ref{Li}). 
Alternatively, it is clear that $\mathcal{O}_{\mathbb{P}W_i^+}(1,i-1)$ and 
$\mathcal{O}_{\mathbb{P}W_i^-}(1,i-1)$ are pull-backs of the same line bundle on $\Sym^i C$ because these projective bundles are contracted to their base $\Sym^i C$ by birational morphisms from 
$M_i(d)$ and $M_{i-1}(d)$ to the 
(singular) GIT quotient $M_\sigma(d)$, where $\sigma=\frac{d}{2}-i$ is the slope of the wall between the moduli spaces $M_i(d)$ and $M_{i-1}(d)$. Furthermore, $\cO_i(1,i-1)$ is a pull-back of a line bundle from that GIT quotient.

This implies that $R\Gamma (\cO_{\mathbb{P}W_i^+}(1,i-1))\simeq R\Gamma (\Sym^i C,L_i)$. Therefore, it suffices to show that
\begin{equation}\label{eq to prove}
R\Gamma_{M_i(d)}(\cO_i(1,i-1))\simeq R\Gamma_{\Sym^iC}(L_i(d))
\end{equation}
via the composition of functors as above.

\begin{claim}\label{claim base d}
    If $d\geq i+g$, then \eqref{eq to prove} holds.
\end{claim}

\begin{proof}
In this case $H^p(M_i,\cO_i(1,i-1))=H^p(\Sym^iC,L_i)=0$ for $p>0$ by Lemmas \ref{cohomology of (1,i-1)} and \ref{KodairaOnSym}. Using this together with the fact that $\Lambda_M^{-i}=\cO_i(0,i)$ has no global sections by Lemma~ \ref{noglobalsections}, it suffices to prove that
$h^0(M_i,\cO_i(1,i-1))=h^0(\Sym^iC,L_i)$
or, equivalently, that
$\chi(M_i,\cO_i(1,i-1))=\chi(\Sym^iC,L_i).$ Thus, Lemma \ref{base case d} proves the Claim.
\end{proof}

We now proceed by a downward induction on $d$, starting with any $d$ such that $d\geq i+g$. For such $d$, we have the result by the Claim above.

Next we perform a step of the downward induction assuming the theorem holds for degree $d+1$. As above, we fix a point $p_0\in C$ and consider the subvariety
$M_{i-1}(d-1)\subset M_i(d+1)$ of codimension~$2$ described in Remark~\ref{restriction of O(m,n)}. Let $\cI\subset \cO_{M_i(d+1)}$ be its ideal sheaf.
As in the proof of Lemma~\ref{KodairaOnSym},
we denote the divisor $p_0+\Sym^{i-1}C\subset \Sym^{i}C$ by $\eta$ and, by abuse of  notation, we denote its pull-back to the projective bundle $\bP W_i^+$ by $\eta$ as well. 
Note that $M_{i-1}(d - 1) \cap \bP W^+_{i} = \bP W^+_{i-1}$.
%, which as a locus in $M_{i-1}(d-1)$ is nothing but $\bP W_{i-1}^+$,
To summarize, we have a commutative diagram of sheaves on $M_i(d+1)$ with exact rows, where we suppress closed embeddings from notation.
\begin{equation}\label{blow-upideal}
    \begin{tikzcd}
    0 \arrow[r] & \cI \arrow[d]\arrow[r] &\cO_{M_i(d+1)}\arrow[d,"\beta"]\arrow[r] &\cO_{M_{i-1}(d-1)} \arrow[d,"\gamma"]\arrow[r] &0\\
    0\arrow[r] &\cO_{\bP W_i^+}(-\eta)\arrow[r] &\cO_{\bP W_i^+}\arrow[r] &\cO_{\bP W_{i-1}^+}\arrow[r] &0
    \end{tikzcd}
\end{equation}
We tensor \eqref{blow-upideal} with $\cO(1,i-1)$. Recall that the restriction of $\cO(1,i-1)$ to $M_{i-1}(d-1)$ is $\cO(1,i-2)$, to 
$\bP W_i^+$ is the pull-back of $L_i(d+1)$ from $\Sym^iC$, and to $\bP W_{i-1}^+$ is the pull-back of 
$L_{i-1}(d-1)$ from $\Sym^{i-1}C$. 
By inductive hypothesis on $i$, the arrow $\gamma$ in \eqref{blow-upideal} gives an isomorphism in cohomology after tensoring with $\cO(1,i-1)$. The same is true for $\beta$ by our inductive assumption on $d$.
%and Claim \ref{claim base d}, it follows that the arrows $\beta$ and $\gamma$ in \eqref{blow-upideal} give an isomorphism in cohomology after tensoring with $\cO(1,i-1)$. 
By the 5-lemma, we conclude that we have an isomorphism
\begin{equation}\label{some eq 1}
R\Gamma(\cI(1,i-1))\simeq R\Gamma\bigl(\cO_{\bP W_i^+}(-\eta)(1,i-1)\bigr).
\end{equation}
As $\cO_{\bP W_i^+}(1, i - 1)$ is the pull-back of $L_i(d+1)$, it follows that $\cO_{\bP W_i^+}(-\eta)(1, i-1)$ is the pull-back of $L_i(d)$ to the
projective bundle, see \eqref{divisorsSym}.
Hence, we can rewrite \eqref{some eq 1} as
\begin{equation}\label{some eq 2}
R\Gamma_B\bigl(\cO_B(1,i-1)(-\cE)\bigr)\simeq R\Gamma_{\Sym^iC}(L_i(d)),
\end{equation}
where $B$ is the blow-up of $M_i(d+1)$ in $M_{i-1}(d-1)$ and $\cE$ its exceptional divisor.

Recall that the goal is to prove \eqref{eq to prove}. We can do one extra simplification. Let $\sigma=\frac{d}{2}-i$ be the slope on the wall between the moduli spaces $M_i(d)$ and $M_{i-1}(d)$ and let $M_\sigma(d)$ be the corresponding (singular) GIT quotient. The birational morphism $M_i(d)\to M_\sigma(d)$ contracts the projective bundle $\bP W_i^+$ to its base $\Sym^i C$, and in particular proving \eqref{eq to prove} is equivalent to proving that
\begin{equation}\label{some eq 3}
R\Gamma_{M_\sigma(d)}(\cO_i(1,i-1))\simeq R\Gamma_{\Sym^iC}(L_i(d))
\end{equation}
by projection formula and Boutot's theorem \cite{boutot}. To show how \eqref{some eq 2} implies \eqref{some eq 3}, we need a geometric construction, a variant of the Hecke correspondence, relating $B$ to $M_\sigma(d)$.

%Consider the $\bP^1$-bundle $\bP F_p$ over $M_i(d+1)$ that parametrizes triples $(E,\phi,l)$, where  $\phi$ is a non-zero section of $E$ and $l\subset E_p$ is a line, subject to the usual stability condition (see Section~\ref{ThaddeusSpaces}) that for every line subbundle $L\subset E$, one must have 
%\begin{equation}\label{qrwgwRGwg}
%\deg L\le\begin{cases}
%    i+\frac{1}{2}\quad & \text{ if }\phi\in H^0(L),\\
%d-i+\frac{1}{2} \quad & \text{ if }\phi\notin H^0(L).
%\end{cases}
%\end{equation}
%We claim that $B$, the blow-up of $M_i(d+1)$ in $M_{i-1}(d-1)$, is isomorphic to the following locus: 
%$$Z=\{(E,\phi,l)\,:\,\phi(p)\in l\}\subset \bP F_p.$$ %Indeed, the projection of $Z$ onto $M_i(d+1)$ is clearly an isomorphism outside of $M_{i-1}(d-1)$, since the latter is precisely the locus where $\phi(p)=0$. Over $M_{i-1}(d-1)$, the fiber of this projection is $\bP^1$. By the universal property of the blow-up, it suffices to check that $Z$ is the blow-up of $M_i(d+1)$ in $M_{i-1}(d-1)$ locally near $(E,\phi)\in M_{i-1}(d-1)$, where we can trivialize $F_p\simeq\cO\oplus\cO$. Its universal section can be written as $s=(a,b)$, where $a,b\in\cO$ is a regular sequence (its vanishing locus is $M_{i-1}(d-1)$ locally near $(E,\phi)$). Then $Z$ is locally given by the equation $ay-bx=0$, where $[x:y]$ are homogeneous coordinates of the $\bP^1$-bundle $\bP F_p$ given by the trivialization $F_p\simeq\cO\oplus\cO$. Thus $Z$ is indeed isomorphic to the blow-up.

By Lemma \ref{blowup lemma}, $B$ carries a family of parabolic (at $p_0\in C$) rank~$2$ vector bundles $E$ with a section $\phi$. The parabolic line at $p_0$ defines a quotient $E\to\cO_{p_0}$, and we define a rank $2$ vector bundle $E'$ as an elementary transformation, by the formula
\begin{equation}\label{elemtr}
0\to E'\to E\to \cO_{p_0}\to0.
\end{equation}
Our condition $\phi(p_0)\in l$ implies that the section $\phi$ lifts to a section $\phi'$ of $E'$. Elementary transformation is well-known to be a functorial construction \cite[\S4]{narasimhan-ramanan75}, in fact we claim that $(E',\phi')$ is a $\sigma$-semistable pair, i.e. we have a morphism
$$
h:\,B\to M_\sigma(d),\quad (E,\phi,l)\mapsto (E',\phi').
$$
Indeed, we need to check that 
\begin{align*}
\deg L' \leq  
\begin{cases}
i\quad & \text{ if }\phi'\in H^0(L'),\\
d-i \quad & \text{ if }\phi'\notin H^0(L').
\end{cases}
\end{align*}
for every line subbundle $L'\subset E'$, which follows from \eqref{qrwgwRGwg} applied to $L'$.

By the Koll\'ar vanishing theorem \cite[Theorem~7.1]{kollar}, $Rh_*\cO_B=\cO_{M_\sigma(d)}$. Indeed, $B$ is smooth, $M_\sigma(d)$ has rational singularities and a general geometric fiber of $h$ is isomorphic to $\bP^1$ (given by extensions \eqref{elemtr} with fixed $E'$). By projection formula,  \eqref{some eq 2} implies \eqref{some eq 3} if we can show that
$$h^*\cO_i(1,i-1)\simeq\cO_B(1,i-1)(-\cE).$$
Outside of $\cE$ and for any $q\in C$, the bundle $F_q$ over the stack of the $\sigma$-semistable pairs (resp.~its determinant $\Lambda'$),  pulls back to the bundle $F_q$ over $B\setminus\cE$  (resp.~its determinant $\Lambda$), by \eqref{elemtr}.
On the other hand, the divisor $E'_i$ of $\sigma$-semistable stable pairs $(E',\phi')$ such that $\phi'$ has a zero, pulls back to the analogous divisor $E_i$ of $B\setminus\cE$, because the section $\phi$ of $E$ is the same as the section $\phi'$ of $E'$. Since $E$ and $\Lambda$ generate the Picard group of $B\setminus\cE$, it follows that 
$h^*\cO_i(1,i-1)\simeq\cO_B(1,i-1)(-c\cE)$ for some integer $c$. 
It remains to show that $c=1$. To this end, we re-examine the diagram \eqref{blow-upideal}.
Note that the proper transform $\tilde\bP$ of $\bP W_i^+$ in $B$
is isomorphic to its blow-up in $\bP W_{i-1}^+$,
which is the Cartier divisor $\eta$. Therefore, 
$\tilde\bP\simeq\bP W_i^+$. However, the restriction
$h^*\cO_i(1,i-1)|_{\tilde\bP}$ is isomorphic to the pull-back of $L_i(d)$ from $\Sym^iC$, while the restriction $\cO_B(1,i-1)|_{\tilde\bP}$ is isomorphic to the pull-back of $L_i(d+1)$. Since $L_i(d)\simeq L_i(d+1)(-\eta)$, and $\cE$ restricts to $\tilde\bP$ as $\eta$, the claim follows.
\end{proof}

\section{Acyclic vector bundles on $M_i$ -- hard cases}\label{hard section}

The main goal of the present section is to prove the following result.

\begin{theorem}\label{vanishing without Z on Mi}
Suppose $2<d\leq 2g+1$ and $1\leq i\leq v$. Let $D=x_1+\ldots +x_\alpha$, $D'=y_1+\ldots +y_\beta$ (possibly with repetitions) of degrees $\alpha,\ \beta\leq d+g-2i-1$, and let $t$ be an integer satisfying
\begin{align}\label{ineq vanishing without Z}
    \deg D-i-1 < t < d+g-2i-1-\deg D'.
\end{align}
If $t\notin [0,\ \deg D]$, then we have 
$$
R\Gamma_{M_i(d)}\left(\left(\bigotimes_{k=1}^{\alpha}F_{x_k}^\vee\right)\otimes \overline{G}_{D'}\otimes\Lambda_M^t\right)=0.
$$
Equivalently, if $\deg D\notin [t,\ t+\deg D']$, then 
$$
R\Gamma_{M_i(d)}\left(G_D^\vee\otimes \left(\bigotimes_{k=1}^{\beta} F_{y_k}\right)\otimes\Lambda_M^t\right)=0.
$$
%In particular, $R\Gamma_{M_i(d)}(G_D^\vee\otimes G_{D'}\otimes\Lambda_M^t)=0$ in either of these cases.
\end{theorem}

\begin{remark}\label{remark hard vanishing}
In the vanishings of Theorem \ref{vanishing without Z on Mi}, we can write $G_D^\vee$ or $\overline{G}_D^\vee$ in place of $\bigotimes_{k=1}^\alpha F_{x_k}^\vee$, and $G_{D'}$ or $\overline{G}_{D'}$ in place of $\bigotimes_{k=1}^\beta F_{y_k}$. This follows from Corollary \ref{standarddeformation} and semi-continuity.
\end{remark}

\begin{comment}
%We will also prove 
\begin{proposition}\label{GD=C}
Suppose $0<d\leq 2g+1$ and $0\leq i\leq v$. Let $D=x_1 +\ldots + x_\alpha$ (possibly with repetitions), with $\alpha=\deg D < d+g-2i-1$. Then 
\begin{equation}\label{important section}
R\Gamma_{M_i}\left(\bigotimes_{k=1}^\alpha F_{x_k}\right)=R\Gamma_{M_i}(G_D)=R\Gamma_{M_i}(G_D^\vee
\otimes\Lambda_M^{\alpha})=\mathbb{C}.
\end{equation}
Moreover, the unique (up to a scalar) global section of these bundles vanishes precisely along 
the union of codimension~$2$ loci $M_{i-1}(\Lambda(-2x_k))$, for $k\in\{1,\ldots,\alpha\}$.
\end{proposition}
\end{comment}

These computations will allow us to verify both the Bondal-Orlov conditions for the fully faithful embeddings of $D^b(\Sym^\alpha C)$ into $D^b(M_i)$, for $\alpha\leq i$, as well as the vanishings needed in order to show semi-orthogonality between the corresponding subcategories of $D^b(M_i)$ thus defined.

We start with a lemma on $M_0(d)$.

\begin{lemma}\label{lemma i=0 vanishing no Z}
Let $d>0$ and $i=0$. Let $D=x_1+\ldots +x_\alpha$, $D'=y_1+\ldots +y_\beta$ (possibly with repetitions) of degrees $\alpha,\ \beta\leq d+g-1$, and let $t$ be an integer satisfying $\deg D < t < d+g-1-\deg D'$.
Then 
%$R\Gamma_{M_0(d)}(G_D^\vee\otimes G_{D'}\otimes\Lambda_M^t)=$
$R\Gamma_{M_0(d)}\bigl((\bigotimes_{k=1}^{\alpha} F_{x_k}^\vee)\otimes (\bigotimes_{k=1}^{\beta} F_{y_k})\otimes\Lambda_M^t\bigr)=0$.
\end{lemma}

\begin{proof}
The vector bundle 
%$\left. G_D^\vee\otimes G_{\tilde{D}}\otimes\Lambda_M^t\right|_{M_0}$ 
%is a deformation over $\bA^1$ of 
$\left.(\bigotimes_{k=1}^{\alpha} F_{x_k}^\vee)\otimes (\bigotimes_{k=1}^{\beta} F_{y_k})\otimes\Lambda_M^t\right|_{M_0}$
%, which 
has the form $\bigoplus\cO_{\mathbb{P}^{d+g-2}}(s_j-t)$ on $M_0=\mathbb{P}^{d+g-2}$, where $-\beta\leq s_j\leq \alpha$ (see Lemma \ref{Fx=O+O(-1)}). By hypothesis, $\alpha-t<0$ and $-\beta-t \geq -(d+g-2)$, so this bundle is $\Gamma$-acyclic.
\end{proof}

\begin{theorem}\label{Thm:vanishing}
Let $d>2$ and $1\leq i\leq v$.
Let $D=x_1+\ldots +x_\alpha$, $D'=y_1+\ldots +y_\beta$ (possibly with repetitions) of degrees $\alpha,\ \beta\leq d+g-1$, and let $t$ be an integer satisfying 
$$\deg D < t < d+g-1-2i-\deg D'.$$
Then 
%$R\Gamma_{M_i(d)}(G_D^\vee\otimes G_{D'}\otimes\Lambda_M^t)=$
$R\Gamma_{M_i(d)}\bigl((\bigotimes_{k=1}^{\alpha} F_{x_k}^\vee)\otimes (\bigotimes_{k=1}^{\beta} F_{y_k})\otimes\Lambda_M^t\bigr)=0$.
\end{theorem}

\begin{proof}
By Theorem \ref{weights computations of all bundles}, the bundle $(\bigotimes_{k=1}^{\alpha} F_{x_k}^\vee)\otimes (\bigotimes_{k=1}^{\beta} F_{y_k})\otimes\Lambda_M^t$ descends from an object with weights in $[-\beta-t,\alpha-t]$. For every $1< j\leq i$, these weights live in the window between $M_{j-1}$ and $M_j$, since by hypothesis $1+2j-d-g<-\beta-t$ and $\alpha-t<0< j$. Then using Theorem \ref{windows wall-crossing}, 
$R\Gamma_{M_i(d)}\bigl((\bigotimes_{k=1}^{\alpha} F_{x_k}^\vee)\otimes (\bigotimes_{k=1}^{\beta} F_{y_k})\otimes\Lambda_M^t\bigr)=
R\Gamma_{M_1(d)}\bigl((\bigotimes_{k=1}^{\alpha} F_{x_k}^\vee)\otimes (\bigotimes_{k=1}^{\beta} F_{y_k})\otimes\Lambda_M^t\bigr)$,
so it suffices to show 
the theorem for the case $i=1$.

Also, using 
$(\bigotimes_{k=1}^{\alpha} F_{x_k}^\vee)\otimes (\bigotimes_{k=1}^{\beta} F_{y_k})\otimes\Lambda_M^t\simeq(\bigotimes_{k=1}^{\alpha} F_{x_k})\otimes (\bigotimes_{k=1}^{\beta} F_{y_k})\otimes\Lambda_M^{t-\alpha},$ it is easy to see that it suffices to show the theorem for the case $\alpha=0$. So we assume $\alpha=0$ and do induction on $\beta$. If $\beta=0$, then $0< t \leq d+g-4$ and the result follows from Lemma \ref{O(-kH+lE)}. If $\beta>0$, write $D'=\Tilde{D}'+y_\beta$. We use the sequence \eqref{reduction F non-dual} from Lemma \ref{BasicKoszul} with $F_{y_\beta}$ and twist it by $(\bigotimes_{k=1}^{\beta-1} F_{y_k})\otimes\Lambda_M^t$ to obtain an exact sequence
\begin{gather*}
\begin{aligned}
0\to \bigotimes_{k=1}^{\beta-1}F_{y_k}\otimes\Lambda_M^t\to \bigotimes_{k=1}^{\beta}F_{y_k}\otimes\Lambda_M^t \to
\qquad\qquad\qquad\qquad\qquad\qquad\qquad\\
\qquad\qquad\qquad\qquad
\to\bigotimes_{k=1}^{\beta-1}F_{y_k}\otimes\Lambda_M^{t+1}\to \left.\bigotimes_{k=1}^{\beta-1}F_{y_k}\otimes\Lambda_M^{t+1}\right|_{M_0(\Lambda(-2y_{\beta}))}\to 0.
\end{aligned}
\end{gather*}

Of these terms, $R\Gamma_{M_0(d-2)}(\bigotimes_{k=1}^{\beta-1}F_{y_k}\otimes\Lambda_M^{t+1})=0$ by Lemma \ref{lemma i=0 vanishing no Z}, since $0<t+1<(d-2)+g-1-(\beta-1)$, while by induction
$
R\Gamma_{M_1(d)}(\bigotimes_{k=1}^{\beta-1}F_{y_k}\otimes\Lambda_M^t)=
R\Gamma_{M_1(d)}(\bigotimes_{k=1}^{\beta-1}F_{y_k}\otimes\Lambda_M^{t+1})=0
$. Therefore, we obtain $R\Gamma_{M_1(d)}(\bigotimes_{k=1}^{\beta}F_{y_k}\otimes\Lambda_M^t)=0$ as well.
\begin{comment}
Now we do induction on $\alpha$. For $\alpha>0$, write $D=\Tilde{D}+x_\alpha$ and twist \eqref{reduction F dual} by $\bigotimes_{k=1}^{\alpha-1}F_{x_k}^\vee\otimes\bigotimes_{k=1}^{\beta}F_{y_k}\otimes\Lambda_M^t$ to obtain

\begin{gather*}
\begin{aligned}
0\to \bigotimes_{k=1}^{\alpha-1}F^\vee_{x_k}\otimes\bigotimes_{k=1}^{\beta}F_{y_k}\otimes\Lambda_M^{t-1}
\to \bigotimes_{k=1}^{\alpha}F^\vee_{x_k}\otimes\bigotimes_{k=1}^{\beta}F_{y_k}\otimes\Lambda_M^{t-1}
\to
\qquad\qquad\\
\to
\bigotimes_{k=1}^{\alpha-1}F^\vee_{x_k}\otimes\bigotimes_{k=1}^{\beta}F_{y_k}\otimes\Lambda_M^{t}
\to \left.\bigotimes_{k=1}^{\alpha-1}F^\vee_{x_k}\otimes\bigotimes_{k=1}^{\beta}F_{y_k}\otimes\Lambda_M^{t}\right|_{M_0(\Lambda(-2x_{\alpha}))}\to 0.
\end{aligned}
\end{gather*}

Similarly, the first and third terms are $\Gamma$-acyclic by induction, while the fourth one is $\Gamma$-acyclic by Lemma \ref{lemma i=0 vanishing no Z}. Thus, we get
$R\Gamma_{M_1(d)}\bigl((\bigotimes_{k=1}^{\alpha} F_{x_k}^\vee)\otimes (\bigotimes_{k=1}^{\beta} F_{y_k})\otimes\Lambda_M^t\bigr)=0$, as desired.
\end{comment}
\end{proof}

\begin{corollary}\label{Coro:GD=C}
Suppose $0<d\leq 2g+1$ and $0\leq i\leq v$. Let $D=x_1 +\ldots + x_\alpha$ (possibly with repetitions), with $\alpha=\deg D < d+g-2i-1$. Then 
\begin{equation}\label{important section}
R\Gamma_{M_i}\left(\bigotimes_{k=1}^\alpha F_{x_k}\right)=R\Gamma_{M_i}(G_D)=R\Gamma_{M_i}(\overline{G}_D)=\mathbb{C}.
\end{equation}
Moreover, the unique (up to a scalar) global section of these bundles vanishes precisely along 
the union of codimension~$2$ loci $M_{i-1}(\Lambda(-2x_k))$, for $k\in\{1,\ldots,\alpha\}$.
\end{corollary}

%\begin{remark}
%If $\alpha=1$ then  $F_{x_1}=G_D=G_D^\vee\otimes\Lambda_M$. If $\alpha>1$ then the precise relationship between vector bundles  $G_D$ and $G_D^\vee\otimes\Lambda_M^{\alpha}$ is not clear, however both are deformations of $\bigotimes_{k=1}^\alpha F_{x_k}$ by Corollary~\ref{standarddeformation}, see also the proof of Proposition~\ref{bottleneck}.
%\end{remark}

\begin{proof}
When $i=0$, $F_{x_k}=\cO_{\mathbb{P}^r}\oplus\cO_{\mathbb{P}^r}(-1)$ on $M_0=\mathbb{P}^r$, $r=d+g-2$ (see Lemma \ref{Fx=O+O(-1)}), and $\bigotimes F_{x_k}$ splits as a sum of line bundles $\bigoplus\cO_{\mathbb{P}^r}(s_j)$, where $-\alpha\leq s_j\leq 0$ and exactly one of the summands is $\cO_{\mathbb{P}^r}$. Since $\alpha \leq d+g-2$, $R\Gamma_{M_i}(\bigotimes_{k=1}^\alpha F_{x_k})=\mathbb{C}$ in this case. Since $G_D$ and $\overline{G}_D$ are deformations of $\bigotimes_{k=1}^\alpha F_{x_k}$ over $\bA^1$, we have \eqref{important section} by semi-continuity and equality of the Euler characteristic.

Let $i\geq 1$. We see that, using Theorem \ref{windows wall-crossing}, it suffices to prove \eqref{important section} on $M_1(d)$. In fact, by Theorem \ref{weights computations of all bundles}, $\bigotimes_{k=1}^\alpha F_{x_k}$ descends from an object with weights within $[-\alpha,\ 0]$, all of which live in the window $(1+2j-d-g,\ j)$ for $1< j\leq i$, since $1+2j-d-g \leq 1+2i-d-g<-\alpha$ by hypothesis. This way we get $R\Gamma_{M_i}(\bigotimes_{k=1}^\alpha F_{x_k})=R\Gamma_{M_1}(\bigotimes_{k=1}^\alpha F_{x_k})$. Similarly, $R\Gamma_{M_i}(G_D)=R\Gamma_{M_1}(G_D)$ and $R\Gamma_{M_i}(\overline{G}_D)=R\Gamma_{M_1}(\overline{G}_D)$.

Hence, we take $i=1$ and $\alpha<d+g-3$. In this case $d>2$.
Let us show that $R\Gamma_{M_1}(\bigotimes F_{x_k})\simeq \mathbb{C}$ first.
We do induction on $\alpha$. If $D=0$, the result is trivial. Otherwise, use the sequence \eqref{reduction F non-dual} from Lemma \ref{BasicKoszul} on $F_{x_\alpha}$ to obtain an exact sequence
\begin{align*}
    0\to \bigotimes_{k=1}^{\alpha-1} F_{x_k} \to \bigotimes_{k=1}^{\alpha} F_{x_k} \to \bigotimes_{k=1}^{\alpha-1} F_{x_k}\otimes\Lambda_M\to\left. \bigotimes_{k=1}^{\alpha-1} F_{x_k}\otimes\Lambda_M\right|_{M_0(d-2)}\to 0.
\end{align*}
Of these terms, we get $R\Gamma_{M_1(d)}(\bigotimes_{k=1}^{\alpha-1} F_{x_k}\otimes\Lambda_M)=0$ from Theorem \ref{Thm:vanishing}. Also, we have $R\Gamma_{M_0(d-2)}(\bigotimes_{k=1}^{\alpha-1} F_{x_k}\otimes\Lambda_M)=0$ from Lemma \ref{lemma i=0 vanishing no Z}, given that $t=1$ and $0<1<d+g-3-(\alpha-1)$. Using the hypercohomology spectral sequence $E_1^{p,q}=H^q(X,\cF^p)$ and induction, we obtain
\begin{align*}
R\Gamma_{M_1}\left(\bigotimes_{k=1}^\alpha F_{x_k}\right)=R\Gamma_{M_1}\left(\bigotimes_{k=1}^{\alpha-1} F_{x_k}\right)=\mathbb{C}.
\end{align*}

Finally, by Corollary \ref{standarddeformation} both $G_D$ and $\overline{G}_D$ are deformations over $\bA^1$ of $\bigotimes_{k=1}^\alpha F_{x_k}$, so we have \eqref{important section} by semi-continuity and equality of the Euler characteristic. It~also 
follows that the global section of $G_D$ (resp.,~ $\overline{G}_D$) is a deformation of the global section of $\bigotimes_{k=1}^\alpha F_{x_k}$ over $\bA^1$, which does not vanish outside of the union of loci $M_{i-1}(\Lambda(-2x_k))$ for $k=1,\ldots,\alpha$. On the other hand, the tautological sections of these bundles, 
that is,
the descent of the tensor product of tautological sections of $\bigotimes\pi_j^*\cF_k$ 
(resp., this tensor product tensored with the sign representation)
for $G_D$ (resp., $\overline{G}_D$),
vanish precisely along these loci.
\end{proof}

%We will prove both Theorem \ref{vanishing without Z on Mi} and Proposition \ref{GD=C} simultaneously, by a combined induction on the degrees of the divisors. 

A key step in the proof of Theorem \ref{vanishing without Z on Mi} will be the following proposition.

\begin{proposition}\label{bottleneck}
Suppose $2< d\leq 2g+1$ and $1\leq i\leq v$. 
Let $D$ be an effective divisor on $C$ and suppose that $\deg D\leq d+g-2i-1$. %$\deg D\leq g$. 
Then 
\begin{equation}\label{onecopyofLambda}
R\Gamma_{M_i(d)}(G_D^\vee\otimes\Lambda_M^{\deg D-1})=R\Gamma_{M_i(d)}(\overline{G}_{D}\otimes\Lambda_M^{-1})=0.
\end{equation}
\end{proposition}

We will first show how Theorem \ref{vanishing without Z on Mi}
follows from Proposition \ref{bottleneck} 
and then proceed with the proof of Proposition \ref{bottleneck}.

\begin{proof}[Proof of Theorem \ref{vanishing without Z on Mi}]
Note that, by rewriting $\overline{G}_{D'}$ in terms of $G_{D'}^\vee$ using Corollary \ref{relation G bar and G}, both statements can be seen to be equivalent, so we will only prove the first one.

%{\em The case $t\notin [0,\ \deg D]$.} 
We first suppose $D=0$ and do induction on $\deg D'$. 
If $D=D'=0$, we need to show that for $t\neq 0$ with $-i-1 < t < d+g-2i-1$ we have $R\Gamma_{M_i(d)}(\Lambda_M^t)=0$. If $t>0$, Lemma \ref{O(-kH+lE)} ensures $R\Gamma_{M_1(d)}(\Lambda_M^t)=0$, since $i\geq 1$ and so $t\leq d+g-4$. But also for every $1<j\leq i$ we have $1+2j-d-g<-t<0< j$, that is, the weight of $\Lambda_M^t$ lives in the window between $M_{j-1}$ and $M_j$, so we conclude $R\Gamma_{M_i(d)}(\Lambda_M^t)=R\Gamma_{M_1(d)}(\Lambda_M^t)=0$ by Theorem \ref{windows wall-crossing}. Suppose now $t<0$, so that $-i\leq t<0$. By Theorem \ref{full vanishing of Lambda*}, $R\Gamma_{M_i(d)}(\Lambda_M^t)=0$.

Let $D=0$ and $\deg D'\geq 1$. By induction, we may assume the result holds for divisors $\Tilde{D'}$ with $\deg\Tilde{D'}<\deg D'$. By Proposition \ref{bottleneck}, $R\Gamma_{M_i(d)}(\overline{G}_D\otimes\Lambda_M^{-1})=0$, since $\deg D'\leq d+g-2i-1$. We need to show that this implies $R\Gamma_{M_i(d)}(\overline{G}_{D'}\otimes\Lambda_M^t)=0$ for $-i-1<t<d+g-2i-1-\deg D'$ and $t\neq 0$. If $t=-1$, this is \eqref{onecopyofLambda}. If $t< -1$, we write $D'=\tilde{D'}+y$ and use the fact that $\overline{G}_{D'}$ is a stable deformation of $F_y\otimes \overline{G}_{\tilde{D'}}$ over $\bA^1$ (see Proposition \ref{proposition stable deformation}). If we take the second sequence of Lemma \ref{BasicKoszul} twisted by $\overline{G}_{\tilde{D'}}\otimes \Lambda_M^t$, we get an exact sequence
\begin{align*}%\label{some exact seq}
0\to \overline{G}_{\tilde{D'}}\otimes\Lambda_M^{t}\to F_y\otimes \overline{G}_{\tilde{D'}}\otimes\Lambda_M^t\to \overline{G}_{\tilde{D'}}\otimes\Lambda_M^{t+1} \to \left.\overline{G}_{\tilde{D'}}\otimes\Lambda_M^{t+1}\right|_{M_{i-1}}\to 0.
\end{align*}
Observe that this is an acyclic chain complex involving $F_y\otimes \overline{G}_{\tilde{D'}}\otimes\Lambda_M^t$ and where the remaining three terms satisfy the corresponding inequalities from \eqref{ineq vanishing without Z}:
$-i-1 < t < d+g-2i-1-\deg \tilde{D'}$,
$-i-1 < t+1 < d+g-2i-1-\deg \tilde{D'}$,
$-(i-1)-1 < t+1 < d-2+g-2(i-1)-1-\deg \tilde{D'}$.
Notice that the inequality $\deg \Tilde{D'}\leq (d-2)+g-2(i-1)-1$ is preserved too. Given that $t< -1$, we have both $t, \ t+1 \neq 0$ so by induction we see that $R\Gamma_{M_i(d)}(\overline{G}_{\tilde{D'}}\otimes\Lambda_M^{t})=R\Gamma_{M_i(d)}(\overline{G}_{\tilde{D'}}\otimes\Lambda_M^{t+1})=0$. On the other hand, we obtain $R\Gamma_{M_{i-1}(d-2)}(\overline{G}_{\tilde{D'}}\otimes\Lambda_M^{t+1})=0$ either by induction if $i>1$, or from Lemma \ref{lemma i=0 vanishing no Z} if $i=1$. Therefore we get the desired vanishing from the corresponding hypercohomology spectral sequence and semi-continuity.

Next we do induction on $\alpha=\deg D$. If $\alpha\geq 1$, we write $D=\Tilde{D}+x_{\alpha}$ and take the first sequence of Lemma \ref{BasicKoszul} with $F_{x_\alpha}^\vee$, twisted by $(\bigotimes_{k=1}^{\alpha-1}F_{x_k}^\vee)\otimes \overline{G}_{D'}\otimes \Lambda_M^t$. This way we get an exact sequence involving $(\bigotimes_{k=1}^{\alpha}F_{x_k}^\vee)\otimes \overline{G}_{D'}\otimes\Lambda_M^t$, and where the remaining terms are $(\bigotimes_{k=1}^{\alpha-1}F_{x_k}^\vee)\otimes \overline{G}_{D'}\otimes\Lambda_M^{t-1}$ and $(\bigotimes_{k=1}^{\alpha-1}F_{x_k}^\vee)\otimes \overline{G}_{D'}\otimes\Lambda_M^{t}$ on $M_i(d)$, and $(\bigotimes_{k=1}^{\alpha-1}F_{x_k}^\vee)\otimes \overline{G}_{D'}\otimes\Lambda_M^{t}$ on $M_{i-1}(d-2)$. All three still satisfy the inequalities (\ref{ineq vanishing without Z}):
$\deg \tilde{D}-i-1 < t-1 < d+g-2i-1-\deg D'$,
$\deg \tilde{D}-i-1 < t < d+g-2i-1-\deg D'$,
$\deg \tilde{D}-(i-1)-1 < t < d-2+g-2(i-1)-1-\deg D'$.
Further, $t,\ t-1\notin [0,\ \deg \Tilde{D}]$ so by induction $R\Gamma_{M_i(d)}((\bigotimes_{k=1}^{\alpha-1}F_{x_k}^\vee)\otimes \overline{G}_{D'}\otimes\Lambda_M^{t-1})=R\Gamma_{M_i(d)}((\bigotimes_{k=1}^{\alpha-1}F_{x_k}^\vee)\otimes \overline{G}_{D'}\otimes\Lambda_M^{t})=0$, while $R\Gamma_{M_{i-1}(d-2)}((\bigotimes_{k=1}^{\alpha-1}F_{x_k}^\vee)\otimes \overline{G}_{D'}\otimes\Lambda_M^{t})=0$ either by induction when $i>1$ or by Lemma \ref{lemma i=0 vanishing no Z} when $i=1$ (observe that when $i=1$ we must have $t>\deg\Tilde{D}$). By looking at the corresponding hypercohomology spectral sequence we obtain the vanishing $R\Gamma_{M_i(d)}((\bigotimes_{k=1}^{\alpha}F_{x_k}^\vee)\otimes \overline{G}_{{D'}} \otimes\Lambda_M^t)=0$.
\end{proof}

It remains to prove Proposition \ref{bottleneck},
which will take the rest of this section and require several steps.
First, we see that it reduces to showing that $\overline{G}_D\otimes\Lambda_M^{-1}$ has no global sections on $M_1(d)$.

\begin{lemma}
Under the assumptions of Proposition \ref{bottleneck}, \eqref{onecopyofLambda} is equivalent to proving
\begin{equation}\label{kjgc,gcmc}
H^0(M_1(d),\overline{G}_{D}\otimes\Lambda_M^{-1})=0
\end{equation}
for the case that every point in $D$ has multiplicity at least $2$.
\end{lemma}

\begin{proof}
First, we see that \eqref{kjgc,gcmc} is clearly necessary, so we need to show it is sufficient. Note that $G_D^\vee\otimes\Lambda_M^{\deg D-1}\simeq \overline{G}_{{D}}\otimes\Lambda_M^{-1}$ by Corollary \ref{relation G bar and G}.
We know by Theorem \ref{weights computations of all bundles} that for $1<j\leq i$ this bundle descends from an object with weights within $[-\deg D+1,\ 1]$, where $1<j$ and $-\deg D+1>1+2j-d-g$ by hypothesis. Hence, by Theorem \ref{windows wall-crossing}, it suffices to show \eqref{onecopyofLambda} when $i=1$.

We write $D=\alpha_1 x_1+\ldots+\alpha_s x_s$ with $x_k\ne x_j$. If $\deg D=0$ then we are done by Theorem~\ref{full vanishing of Lambda*}. 
Let us now assume that some $\alpha_i = 1$, say, for simplicity, $\alpha_1 = 1$. 
Then we can write $D = \tilde{D} +x_1$ and argue by induction on $\deg D$ as follows. By Lemma~\ref{BasicKoszul}, we obtain an exact sequence
\begin{align*}%\label{reductions2}
    0\to \overline{G}_{\tilde D}\otimes\Lambda_M^{-1}
    \to \overline{G}_{D}\otimes\Lambda_M^{-1}\to 
    \overline{G}_{\tilde D}\to 
    \overline{G}_{\tilde D}|_{M_{0}}\to 0,
\end{align*}
where $M_0=M_{0}(\Lambda(-2x_1))$.
By the induction hypothesis, the first term in each sequence is $\Gamma$-acyclic.
By Corollary \ref{Coro:GD=C}, the last two terms in each sequence have vanishing higher cohomology and $H^0=\bC$ with a global section that does not vanish along $M_{0}(\Lambda(-2x_1))$. Thus 
$$R\Gamma_{M_1(d)}(\overline{G}_{D}\otimes\Lambda_M^{-1})=0
$$
by the hypercohomology spectral sequence $E_1^{p,q}=H^q(X,\cF^p)$ and semi-continuity. So we can assume that $\alpha_k>1$ for all~$k$.
Again, we write $D=\tilde{D}+x_1$ and get
\begin{align}\label{reductions22}
    0\to \overline{G}_{\tilde D}\otimes\Lambda_M^{-1}
    \to \overline{G}_{\tilde{D}}\otimes F_{x_1}\otimes\Lambda_M^{-1}\to 
    \overline{G}_{\tilde D}\to 
    \overline{G}_{\tilde D}|_{M_{0}}\to 0,
\end{align}
The last two terms in \eqref{reductions22} still have $R\Gamma=\bC$, but now the global section vanishes along $M_0(\Lambda(-2x_1))$.
Therefore, applying the same hypercohomology spectral sequence, we conclude that $F_{x_1}\otimes \overline{G}_{\tilde D}\otimes\Lambda_M^{-1}$ has the following cohomology: $h^p=0$ for $p\ge2$ and $h^0=h^1=1$. By Remark \ref{eulerdef}, its stable deformation $\overline{G}_{D}\otimes\Lambda_M^{-1}$ must have $h^p=0$ for $p\ge2$ and $h^0=h^1$.
Hence, it suffices to show that
$H^0(M_1(d),\overline{G}_{D}\otimes\Lambda_M^{-1})=0$, as claimed.
\end{proof}

In what follows, we focus on proving \eqref{kjgc,gcmc}, under the assumptions of Proposition \ref{bottleneck}, and with $D=\alpha_1x_1+\ldots+\alpha_sx_s$, $\alpha_k>1$.
We recall the construction of $\overline{G}_D$ from the proof of Corollary~\ref{standarddeformation} adapted to our case when $D$ is not necessarily a fat point.
Let $M=M_1(d)$.

Let $B_\alpha=\frac{\mathbb{C}[t_1,\ldots,t_\alpha]}{(\sigma_1,\ldots,\sigma_\alpha)}$, the covariant algebra, and $\bB_\alpha=\Spec B_\alpha
%\Spec\frac{\mathbb{C}[t_1,\ldots,t_\alpha]}{(\sigma_1,\ldots,\sigma_\alpha)}
$.
Write the indexing set $\{1,\ldots,\alpha\}$ 
as a disjoint union of sets $A_k$ of cardinality $\alpha_k$ for $k=1,\ldots,s$, and denote $B=B_{\alpha_1}\otimes\ldots\otimes B_{\alpha_s}$
For every $j\in A_k$, we have a diagram of morphisms as in \eqref{diagram tau},
\begin{equation}\label{bigger diagram tau}
\begin{tikzcd}
&\mathbb{B}_{\alpha_1}\times\ldots\times\bB_{\alpha_s}\times M \arrow[r,"\pi_j"]\arrow{dr}[swap]{\tau} &\mathbb{D}_{\alpha_k}\times M \arrow[r,"q_k"]\arrow[d,bend left=20,"\rho"] & C\times M \arrow[dl] \\
& & M\arrow[u,bend left = 20,"\imath"] &
\end{tikzcd}
\end{equation}

We let $\cF_k=q_k^*{F}$, where $F$ is the universal bundle, and therefore $\overline{G}_D=\tau_*^{S_{\alpha_1}\times\ldots\times S_{\alpha_s}}\left(\bigotimes\pi_j^*\cF_k\otimes\sign\right)$.
Here $\tau_*$ does not change local sections of sheaves, but just forgets the $B$-algebra structure.
Thus \eqref{kjgc,gcmc} is equivalent to the following: $\Lambda_M^{-1}\otimes\bigotimes\pi_j^*\cF_k$ does not have
skew-invariant global sections (with respect to each factor of $S_{\alpha_1}\times\ldots\times S_{\alpha_s}$).

The restriction of $\Lambda_M^{-1}\otimes\bigotimes\pi_j^*\cF_k$ to the special fiber $M$ is $\Lambda_M^{-1}\otimes\bigotimes F_{x_k}^{\otimes\alpha_k}$. While the group $S_{\alpha_1}\times\ldots\times S_{\alpha_s}$ acts trivially on the special fiber, the action on the vector bundle is still non-trivial (the action permutes tensor factors within each block).

\begin{lemma}\label{indecomposable}
Suppose $s=1$, that is, $D=\alpha x$ is a fat point. Write $\cF=q_1^*F$ and let $\rho$ be as in \eqref{bigger diagram tau}. Then $\End \rho_*\cF=\mathbb{D}_\alpha$. In particular, $\rho_*\cF$ is indecomposable.
\end{lemma}

\begin{proof}
We see that $\rho_*\cF=\Phi_F(\mathcal{O}_{\alpha x})$, where $\Phi_F$ is the Fourier--Mukai functor with kernel $F$. The result follows from full faithfulness of $\Phi_F$, which is given by Theorem \ref{Poincare is fully faithful on M1}.
\end{proof}

\begin{lemma}\label{sRGSRHSARHR}
As a representation of 
$S_{\alpha_1}\times\ldots\times S_{\alpha_s}$, the space  
$H^0(M,\Lambda_M^{-1}\otimes\bigotimes F_{x_k}^{\otimes\alpha_k})$ 
is isomorphic to the direct sum $V_{\alpha_1}\oplus\ldots\oplus V_{\alpha_s}$  of irreducible representations, where each $V_{\alpha_k}$ is the standard $(\alpha_k-1)$-dimensional irreducible representation of $S_{\alpha_k}$ and the other factors $S_{\alpha_l}$, $l\ne k$, act on $V_{\alpha_k}$ trivially. If we realize the representation $V_{\alpha_k}$ as $\{\sum a_je_j\,|\,\sum a_j=0\}\subset\bC^{\alpha_k}$
then the vector $e_{j'}-e_{j''}\in V_{\alpha_k}$ corresponds to the global section $s_{j'j''}$ of $\Lambda_M^{-1}\otimes\bigotimes F_{x_k}^{\otimes\alpha_k}$ that can be written as a tensor product of the universal sections 
$s_l$ of $F_{x_l}$ with $l\neq k$,  the universal sections 
$s_k$ of $F_{x_k}$ in positions  $j\ne j',j''$ and the  section of 
$\Lambda_M^{-1}\otimes F_{x_k}\otimes F_{x_k}$ 
(in positions $j'$, $j''$)
given by wedging (recall that $\Lambda_M$ is the determinant of $F_{x_k}$).
\end{lemma}

\begin{proof}%[Proof of the Claim]
The sections $s_{j'j''}$ satisfy the same linear relations as the difference vectors $e_{j'}-e_{j''}$, namely that $s_{j_1j_2}+s_{j_2j_3}+\ldots+s_{j_{r-1}j_r}+s_{j_rj_1}=0$ for $j_1,\ldots,j_r\in A_k$. Indeed, choose a basis $\{f_1,f_2\}$ in  a fiber of the rank~$2$ bundle $F_{x_k}$ so that the universal section is equal to $f_2$ and the determinant is given by $f_1\wedge f_2$. 
After reordering of $j_1,\ldots,j_r$,
and ignoring factors of $s_{jj'}$ given by the universl sections $s_l$ of $F_{x_l}$ with $l\neq k$, we have  \begin{align*}
s_{12}+s_{23}+\ldots+s_{r1}=&
(f_1\otimes f_2)\otimes f_2\otimes\ldots \otimes f_2-
(f_2\otimes f_1)\otimes f_2\otimes\ldots \otimes f_2+\\
&f_2\otimes (f_1\otimes f_2)\otimes\ldots \otimes f_2-
f_2\otimes (f_2\otimes f_1)\otimes\ldots \otimes f_2+\\
&\ldots=0.
\end{align*}

Let $j_k = \min (A_k)$ for $k=1,\ldots,s$. 
It suffices to prove that the sections $s_{j_kj}$ for $k=1,\ldots,s$ and $j\in A_k\setminus\{j_k\}$ form a basis of  $H^0(M,\Lambda_M^{-1}\otimes\bigotimes F_{x_k}^{\otimes\alpha_k})$. We~prove this by induction on $\alpha$. This is true if $\alpha=0$ by Lemma \ref{RGamma(lambda-1)=0} and if $\alpha=1$ by Lemma \ref{acyclicty of Fx, Fx*}. Let $\tilde F=F_{x_1}^{\otimes\alpha_1}\otimes\ldots\otimes F_{x_s}^{\otimes(\alpha_s-1)}$.
We have the usual exact sequence 
obtained from Lemma~\ref{BasicKoszul}:
\begin{align}\label{reductioninproduct}
    0\to 
    \Lambda_M^{-1}\otimes
    \tilde F
    \to 
    \Lambda_M^{-1}\otimes
    \tilde F\otimes F_{x_s}
    \to 
    \tilde F\to 
    \tilde F|_{M_{0}}\to 0,
\end{align}
where $M_0=M_{0}(\Lambda(-2x_s))$.
By Corollary \ref{Coro:GD=C}, the last two terms have vanishing higher cohomology and  $H^0=\bC$.
If $\alpha_s=1$ or, equivalently, $A_s=\{\alpha\}$,
then the global section of $\tilde F$
does not vanish along $M_{0}$ and therefore 
$H^0(\Lambda_M^{-1}\otimes\tilde F)=H^0(\Lambda_M^{-1}\otimes
    \tilde F\otimes F_{x_s})$ 
    by the corresponding hypercohomology spectral sequence, and the basis stays the same.
On the other hand, if $\alpha\ne j_s$  then 
the global section of $\tilde F$ (the tensor product of universal sections)
vanishes along $M_{0}$ inducing the zero map
$H^0(\tilde F)\to H^0( 
    \tilde F|_{M_{0}})$.
Moreover, the section $s_{j_s\alpha}\in H^0(\Lambda_M^{-1}\otimes
    \tilde F\otimes F_{x_s})$ maps onto the global section
of~$\tilde F$.
Thus  the claim also follows 
from the hypercohomology spectral sequence.
\end{proof}

The sheaf $\bigotimes\pi_j^*\cF_k$ carries a filtration by $B_{\geq d}\left(\bigotimes\pi_j^*\cF_k\right)$, where $B_{\geq d}$ is the ideal of monomials of degree $\geq d$. The associated graded object is $\gr \left(\bigotimes\pi_j^*\cF_k\right):=\bigotimes_kF_{x_k}^{\otimes\alpha_k}\otimes_{\mathcal{O}_M}B$. If $\Lambda_M^{-1}\otimes\bigotimes\pi_j^*\cF_k$ has a skew-invariant global section, an associated graded section will be a skew-invariant global section of $\Lambda_M^{-1}\otimes\gr\left(\bigotimes\pi_j^*\cF_k\right)$.

By Frobenius reciprocity, the space of skew-invariants in $(V_{\alpha_1}\boxtimes\Id\boxtimes\ldots\boxtimes\Id)\otimes B\subset H^0(M,\Lambda_M^{-1}\otimes\bigotimes F_{x_k}^{\alpha_k})\otimes B$ has dimension $\alpha_1-1$ and basis
\begin{equation}\label{basis}
\sum_{i<j}\left(\frac{\partial^r\Delta_1}{\partial t_i^r}-\frac{\partial^r\Delta_1}{\partial t_j^r}\right)s_{ij}\boxtimes\Delta_2\boxtimes\ldots\boxtimes\Delta_s,
\end{equation}
$r=1,\ldots,\alpha-1$, where $\Delta_i\in\bC[t_1,\ldots,t_{\alpha_i}]$ is the Vandermonde determinant. Global sections of $H^0(M,\Lambda_M^{-1}\otimes\bigotimes F_{x_k}^{\alpha_k})\otimes B$ coming from $V_{\alpha_k}$, $k>1$ are analogous. We will show that these global sections of $\Lambda_M^{-1}\otimes \gr\left(\bigotimes\pi_j^*\cF_k\right)$ do not lift to sections of $\Lambda_M^{-1}\otimes\bigotimes\pi_j^*\cF_k$.

\begin{lemma}
It suffices to prove \eqref{kjgc,gcmc} for $s=1$ and $\alpha=\alpha_1$.
\end{lemma}

\begin{proof}%[Proof of the Claim]
We argue by induction on $s$. Let $\Tilde{D}=\alpha_2x_2+\ldots+\alpha_sx_s$ and suppose $H^0(\Lambda_M^{-1}\otimes \overline{G}_{\Tilde{D}})=0$. Arguing as in the proof of Lemma \ref{sRGSRHSARHR}, using the usual spectral sequences, we get $H^0(\Lambda_M^{-1}\otimes F_{x_1}^{\alpha_1}\otimes \overline{G}_{\Tilde{D}})=V_{\alpha_1}$, with a basis given by \eqref{basis}. Note that $\Delta_i\in B_{\alpha_i}$ is the element of top degree. Therefore, lifting basis elements to sections of $\overline{G}_D$ is equivalent to lifting them to $\overline{G}_{\alpha_1x_1}$.
\end{proof}

From now on, we let $\alpha=\alpha_1$, $x=x_1$ and $\cF=\cF_1$. The space of skew-invariants in $H^0(\Lambda_M^{-1}\otimes F_x^{\otimes\alpha})\otimes B_\alpha$ has basis $I_r=\sum_{i<j}\left(\frac{\partial^r\Delta}{\partial t_i^r}-\frac{\partial^r\Delta}{\partial t_j^r}\right)s_{ij}$, $r=1,\ldots,\alpha-1$. Writing, formally, $s_{ij}=e_i-e_j$, we also have $I_r=\sum_{i}\frac{\partial^r\Delta}{\partial t_i^r}e_i$. We claim that no $I_r$ lifts to a global skew-invariant section $\widetilde{I}_r$ of $\Lambda_M^{-1}\otimes\bigotimes\pi_j^*\cF$. We argue by induction on $\alpha$.

\begin{lemma}\label{lemma E}
Let $D=\alpha x$, $D'=(\alpha-1)x$. Assuming \eqref{kjgc,gcmc} holds for $D'$, we have
$$
H^0\left(\bB_\alpha\times M, \Lambda_M^{-1}\otimes\bigotimes\pi_j^*\cF\otimes\sign\right)^{S_{\alpha-1}}=\bC^{\alpha-1},
$$
where $S_{\alpha-1}\subset S_\alpha$ is the subgroup fixing the last index.
\end{lemma}

\begin{proof}
We start with the Koszul complex on $C\times M$
\begin{equation}\label{some koszul}
0\to\det F^\vee\to F^\vee\to\cO_{C\times M}\to \cO_{\cD'}\to 0,
\end{equation}
where $\cD'\subset C\times M$ is the vanishing locus of the universal section. Recall that $\cD'$ is smooth over $C$ with fibers $M(\Lambda(-2x))\subset M$ of codimension $2$ over $x\in C$. In particular, $\cD'$ is flat over $C$, and so the local generator $t\in\mathfrak{m}_x$ for $x\in C$ is not a zero divisor in $\cO_{\cD'}$. It follows that the pullback of \eqref{some koszul} to $\bD_\alpha\times M$ is also exact:
$$
0\to\Lambda_M^{-1}\to \cF^\vee\to\cO_{\bD_{\alpha}\times M}\to \cO_{\bD_{\alpha}\times M(\Lambda(-2x))}\to 0.
$$
We pullback to $\bB_\alpha\times M$ and tensor with the locally free sheaf $\bigotimes_{j=1}^{\alpha-1}\pi_j^*\cF$ to obtain
\begin{gather}
\begin{aligned}\label{another koszul}
0\to \Lambda_M^{-1}\otimes\bigotimes_{j=1}^{\alpha-1}\pi_j^*\cF\to
\Lambda_M^{-1}\otimes\bigotimes_{j=1}^{\alpha}\pi_j^*\cF\to
\qquad\qquad\qquad\qquad\qquad\\
\qquad\qquad\qquad\qquad\qquad
\to\bigotimes_{j=1}^{\alpha-1}\pi_j^*\cF\to
\left.\bigotimes_{j=1}^{\alpha-1}\pi_j^*\cF\right|_{\bB_\alpha\times M(\Lambda(-2x))}
\to 0.
\end{aligned}
\end{gather}
Next we compute $S_{\alpha-1}$-skew-invariant cohomology of the first, third and fourth terms of \eqref{another koszul}. For each of these terms $U$, we have $H^0(U\otimes\sign)^{S_{\alpha-1}}=\rho_*\pi_{\alpha,*}^{S_\alpha-1}(U\otimes\sign)$, which by Lemma \ref{deformation lemma} is a deformation of $\alpha$ copies of $\rho^*\pi_{\alpha,*}^{S_\alpha-1}(U\otimes\sign)$ over $\mathbb{A}^1$. For the first term $U=\Lambda_M^{-1}\otimes\bigotimes_{j=1}^{\alpha-1}\pi_j^*\cF$ in \eqref{another koszul}, we have $\rho^*\pi_{\alpha,*}^{S_\alpha-1}(U\otimes\sign)$ is isomorphic to $\Lambda_M^{-1}\otimes\overline{G}_{{D'}}$ (see the proof of Proposition \ref{proposition stable deformation}), which is $\Gamma$-acyclic by induction assumption. For the last two terms, $\rho^*\pi_{\alpha,*}^{S_\alpha-1}(U\otimes\sign)$ is isomorphic to $\overline{G}_{D'}$ and $\overline{G}_{D'}|_{M(\Lambda(-2x))}$, respectively, both of which have $R\Gamma=\bC$ by Corollary \ref{Coro:GD=C}.

From this, it follows that $H^0\left(\Lambda_M^{-1}\otimes\bigotimes_{j=1}^{\alpha-1}\pi_j^*\cF\otimes\sign\right)^{S_{\alpha-1}}=0$, while $H^0\left(\bigotimes_{j=1}^{\alpha-1}\pi_j^*\cF\otimes\sign\right)^{S_{\alpha-1}}=H^0\left(\bigotimes_{j=1}^{\alpha-1}\pi_j^*\cF|_{\bB_\alpha\times M(\Lambda(-2x))}\otimes\sign\right)^{S_{\alpha-1}}=\bC^{\alpha}$ and their higher cohomology vanishes. Furthermore, the last two groups are isomorphic to $\bD_\alpha$ as $\bD_\alpha$-modules and generated by the universal section $\left(\bigotimes_{j=1}^{\alpha-1}\pi_j^*\Sigma\right)\otimes\Delta_{\alpha-1}$, which under the restriction map to $\bB_\alpha\times M(\Lambda(-2x))$ goes to $\left(\bigotimes_{j=1}^{\alpha-1}t_\alpha\pi_j^*\Sigma\right)\otimes\Delta_{\alpha-1}$. Therefore, the first page of the spectral sequence $E_1^{p,q}=H^q(X,\cF^p)$ associated with \eqref{another koszul} has the following shape:
\begin{equation*}
\begin{tikzcd}[cramped, sep=scriptsize]
\vdots & \vdots & \vdots & \vdots \phantom{.}\\
0\arrow[r] &H^2\left(\bB_\alpha\times M, \Lambda_M^{-1}\otimes\bigotimes_{j=1}^\alpha\pi_j^*\cF\otimes\sign\right)^{S_{\alpha-1}}\arrow[r] & 0\arrow[r] & 0\phantom{.}\\
0\arrow[r] &H^1\left(\bB_\alpha\times M, \Lambda_M^{-1}\otimes\bigotimes_{j=1}^\alpha\pi_j^*\cF\otimes\sign\right)^{S_{\alpha-1}}\arrow[r] & 0\arrow[r] & 0\phantom{.}\\
0\arrow[r] &H^0\left(\bB_\alpha\times M, \Lambda_M^{-1}\otimes\bigotimes_{j=1}^\alpha\pi_j^*\cF\otimes\sign\right)^{S_{\alpha-1}}\arrow[r] & \bD_\alpha\arrow[r,"\cdot t_\alpha^{\alpha-1}"] & \bD_\alpha.
\end{tikzcd}
\end{equation*}
We conclude that $H^0\left(\bB_\alpha\times M, \Lambda_M^{-1}\otimes\bigotimes_{j=1}^\alpha\pi_j^*\cF\otimes\sign\right)^{S_{\alpha-1}}=\bC^{\alpha-1}$.
\end{proof}

\begin{proof}[Proof of Proposition \ref{bottleneck}]
We need to show that none of the $S_{\alpha-1}$-skew-invariant global sections found in Lemma \ref{lemma E} is $S_\alpha$-skew-invariant. We can explicitly write a basis of $H^0\left(\bB_\alpha\times M, \Lambda_M^{-1}\otimes\bigotimes_{j=1}^\alpha\pi_j^*\cF\otimes\sign\right)^{S_{\alpha-1}}=\Hom\left(\pi_\alpha^*\cF,\bigotimes_{j=1}^{\alpha-1}\pi_j^*\cF\otimes\sign\right)^{S_{\alpha-1}}$. Namely, consider the surjection $\pi_\alpha^*\cF\twoheadrightarrow F_x$, followed by an isomorphism $F_x\xrightarrow{\sim}t_1^{\alpha-1}\cF\simeq F_x$. Then we tensor with $\bigotimes_{j=2}^{\alpha-1}\pi_j^*\Sigma$, multiply by $t_2^{\alpha-3}t_3^{\alpha-4}\cdots t_{\alpha-2}$ and skew-symmetrize over $\{1,2,\ldots,\alpha-1\}$. This way we obtain a morphism 
$$\mu\in \Hom\left(\pi_\alpha^*\cF,\bigotimes_{j=1}^{\alpha-1}\pi_j^*\cF\otimes\sign\right)^{S_{\alpha-1}}$$
and therefore, also morphisms
\begin{equation}\label{mu basis}
\mu,t_\alpha\mu,\ldots,t_\alpha^{\alpha-2}\mu\in \Hom\left(\pi_\alpha^*\cF,\bigotimes_{j=1}^{\alpha-1}\pi_j^*\cF\otimes\sign\right)^{S_{\alpha-1}}.
\end{equation}

We claim that $t_\alpha^{\alpha-2}\mu\neq 0$, and therefore \eqref{mu basis} gives a basis of the space $H^0\left(\bB_\alpha\times M, \Lambda_M^{-1}\otimes\bigotimes_{j=1}^\alpha\pi_j^*\cF\otimes\sign\right)^{S_{\alpha-1}}$ over $\bC$. Indeed, notice that $t_\alpha^{\alpha-2}(t_1^{\alpha-1}t_2^{\alpha-3}t_3^{\alpha-4}\cdots t_{\alpha-2})$ is equal (up to sign) to the Vandermonde determinant $\Delta_\alpha\in \bB_\alpha$, and it is also equal (up to a multiple) to $t_1^{\alpha-1}\Delta_{\alpha-1}$, where $\Delta_{\alpha-1}$ is the Vandermonde determinant in $t_1,\ldots,t_{\alpha-1}$. 
We show that these two expressions are not equal to zero. Let $\bB_\alpha^{\topp}$ be the degree $\alpha\choose 2$ component of $\bB_\alpha$. Being spanned by $\Delta_\alpha$, $\bB_\alpha^{\topp}$ is isomorphic to $\sign$ as an $S_\alpha$-module. Consider a monomial $m=t_1^{d_1}\cdots t_\alpha^{d_\alpha}\in\bB_\alpha^{\topp}$. If $d_j=d_k$, then $m$ is fixed by $(j\:k)\in S_\alpha$, so it must vanish. This leaves only the orbit of $t_1^{\alpha-1}t_2^{\alpha-2}\cdots t_{\alpha-1}$ under $S_\alpha$, which all must be nonzero with
\begin{equation}\label{sigma}
\sigma(t_1^{\alpha-1}t_2^{\alpha-2}\cdots t_{\alpha-1})=(\sign\sigma)t_1^{\alpha-1}t_2^{\alpha-2}\cdots t_{\alpha-1}
\end{equation}
for $\sigma\in S_\alpha$. Monomials in $t_1^{\alpha-1}\Delta_{\alpha-1}$ of the form \eqref{sigma} have $\sigma(1)=1$ and $\sigma(\alpha)=\alpha$. Moreover, they appear with a relative factor of $\sign\sigma$ by anti-symmetry of $\Delta_{\alpha-1}$, so they do not cancel in $\bB_\alpha$, as claimed.

Therefore, $t_\alpha^{\alpha-2}\mu$ can be described as follows: it is the surjection $\pi_\alpha^*\cF\twoheadrightarrow F_x$ followed by an isomorphism $F_x\xrightarrow{\sim}t_1^{\alpha-1}\cF\simeq F_x$, twisted by $\bigotimes_{j=2}^{\alpha-1}\pi_j^*\Sigma$, multiplied by $\Delta_{\alpha-1}$ and then skew-symmetrized over $\{1,2,\ldots,\alpha-1\}$. So the associated graded section of $t_\alpha^{\alpha-2}\mu$ is $\sum_{j=1}^{\alpha-1}s_{j\alpha}\cdot\Delta_\alpha\neq 0$ (cf. Lemma \ref{sRGSRHSARHR}).

Finally, we check that no linear combination of \eqref{mu basis} is $S_\alpha$-skew-invariant. In fact, if $\alpha>2$, the associated graded section does not involve $s_{jk}$ for $j,k<\alpha$, while if $\alpha=2$, the section is $s_{12}(f_1-f_2)$, which is symmetric, not skew-symmetric. This completes the proof.
\end{proof}

\section{Computation of $R\Hom (G_D,G_{D})$}\label{computation section}

Now we will compute some of the $\Ext$ groups between $G_D$ and $G_{D'}$, which will be needed in the proof of our semi-orthogonal decomposition.

\begin{proposition}\label{vanishing beyond alpha-t}
Let $d\leq 2g+1$ and $1\leq i\leq v$. Suppose $D$, $D'$ are effective divisors and let $t$ be an integer satisfying 
$$
\deg D-i-1<t<d+g-1-2i-\deg D'.
$$ 
Then
$$
H^p(M_i(d),G_D^\vee\otimes G_{D'}\otimes\Lambda_M^t)=0
$$
for every $p>\deg D-t$.
\end{proposition}

\begin{proof}
Let $\alpha=\deg D$, $\beta=\deg D'$. We first do the case $\alpha=\beta=0$, for which we need to show vanishing of $H^p(M_i(d),\Lambda_M^t)$ for $p>-t$. If $t=0$, this is trivial. If $t<0$, observe that $i\geq -t$, so Theorem \ref{full vanishing of Lambda*} gives $R\Gamma_{M_i} (\Lambda_M^t)=0$. If $t>0$, we notice $\Lambda_M^t$ has weight $-t$, with $1+2j-d-g<-t<j$ for every $1<j\leq i$, so by Theorem \ref{windows wall-crossing} we must have $R\Gamma_{M_i}(\Lambda_M^t)=R\Gamma_{M_1}(\Lambda_M^t)$. But the latter is $0$ by Lemma \ref{O(-kH+lE)}, since $t\leq d+g-4$.

Now we prove the result for $\beta=0$ and $\alpha\geq 1$ by induction on $\alpha$. Write $D=\Tilde{D}+x$ and twist \eqref{reduction F dual} by $G_{\Tilde{D}}^\vee\otimes\Lambda_M^t$ to get an exact sequence
\begin{gather}
\begin{aligned}\label{SomeExSeq1}
0\to G_{\Tilde{D}}^\vee\otimes\Lambda_M^{t-1}\to F_x^\vee\otimes G_{\Tilde{D}}^\vee\otimes\Lambda_M^t\to
\qquad\qquad\qquad\qquad\qquad\\
\qquad\qquad\qquad\qquad\qquad
\to G_{\Tilde{D}}^\vee\otimes\Lambda_M^t \to \left.G_{\Tilde{D}}^\vee\otimes\Lambda_M^t\right|_{M_{i-1}(d-2)}\to 0.
\end{aligned}
\end{gather}
By induction, the first term has $H^p(M_i(d),G_{\Tilde{D}}^\vee\otimes\Lambda_M^{t-1})=0$ for $p>\alpha-t$, and the third term has $H^p(M_{i}(d),G_{\Tilde{D}}^\vee\otimes\Lambda_M^t)=0$ for $p>\alpha-t-1$. We see that on the last term we also have $H^p(M_{i-1}(d-2),G_{\Tilde{D}}^\vee\otimes\Lambda_M^t)=0$ for $p>\alpha-t-1$. Indeed, if $i>1$ this follows by induction, while if $i=1$, we see that $t>\alpha$ and the restriction of $G_{\Tilde{D}}^\vee\otimes\Lambda_M^t$ to $M_0(d-2)=\mathbb{P}^{d+g-4}$ is a deformation of a sum of line bundles $\bigoplus\mathcal{O}_{\mathbb{P}^{d+g-4}}(s_j)$ with $-(d+g-4)\leq-t\leq s_j \leq \alpha-t-1\leq 0$ (see Corollary \ref{standarddeformation} and Remark \ref{restriction of O(m,n)}). If $\alpha-t-1=0$, this sum of line bundles is $\Gamma$-acyclic, and if $\alpha-t-1=0$, this has vanishing cohomology $H^p$ for $p>0=\alpha-t-1$. In either case, we conclude that the last term has vanishing $H^p$ for $p>\alpha-t-1$ by semi-continuity. 
Taking the hypercohomology spectral sequence $E_1^{p,q}=H^q(X,\mathcal{F}^p)$ of \eqref{SomeExSeq1}, we conclude that $H^p(M_i(d),F_x^\vee\otimes G_{\Tilde{D}}^\vee\otimes \Lambda_M^t)=0$ for $p>\alpha-t$. Since $G_{D}^\vee\otimes \Lambda_M^t$ is a stable deformation over $\mathbb{A}^1$ of $F_x^\vee\otimes G_{\Tilde{D}}^\vee\otimes \Lambda_M^t$ by Proposition \ref{proposition stable deformation}, then by semi-continuity we also have $H^p(M_i(d),G_{D}^\vee\otimes \Lambda_M^t)$ for $p>t-\alpha$.

Finally, we do induction on $\beta\geq 1$. Similarly, write $D'=\Tilde{D}'+y$ and twist \eqref{reduction F non-dual} by $G_D^\vee\otimes G_{\Tilde{D}'}\otimes\Lambda_M^{t}$ to get an exact sequence
\begin{align*}
0\to G_D^\vee\otimes G_{\Tilde{D}'}\otimes\Lambda_M^t\to G_D^\vee\otimes G_{\Tilde{D}'}\otimes F_y\otimes \Lambda_M^t\to
\qquad\qquad\qquad\qquad\qquad\\
\qquad\qquad
\to G_D^\vee\otimes G_{\Tilde{D}'}\otimes\Lambda_M^{t+1}\to \left. G_D^\vee\otimes G_{\Tilde{D}'}\otimes\Lambda_M^{t+1}\right|_{M_{i-1}(d-2)}\to 0.
\end{align*}
By induction, the first term has $H^p=0$ for $p>\alpha-t$ and the third one has $H^p=0$ for $p>\alpha-t-1$. The last term has vanishing $p$-th cohomology for $p>\alpha-t-1$, which follows by induction when $i>1$. It remains to check the case $i=1$. In this case, the restriction $G_D^\vee\otimes G_{\Tilde{D}'}\otimes\Lambda_M^{t+1}|_{M_{i-1}(d-2)}$ is a deformation of a sum $\bigoplus\mathcal{O}_{\mathbb{P}^{d+g-4}}(s_j)$, with $-(d+g-4)\leq -t-\beta\leq \alpha-t-1\leq 0$. As before, we see that this has vanishing $H^p$ for $p>\alpha-t-1$ and the same is true for $G_D^\vee\otimes G_{\Tilde{D}'}\otimes\Lambda_M^{t+1}|_{M_{i-1}(d-2)}$ by semi-continuity. The result then follows from taking the spectral sequence $E_1^{p,q}=H^q(X,\mathcal{F}^p)$ and semi-continuity.
\end{proof}

\begin{corollary}\label{coro:vanishing beyond alpha}
Let $d\leq 2g+1$ and $0\leq i\leq v$. If $\deg D\leq i$ and $\deg D'<d+g-1-2i$, we have
$$
H^p(M_i(d),G_D^\vee\otimes G_{D'})=0
$$
for every $p>\deg D$.
\end{corollary}

\begin{proof}
If $i=0$ then $D$ must be zero and the result follows from Corollary \ref{Coro:GD=C}. For $i\geq 1$, this follows from taking $t=0$ in Proposition \ref{vanishing beyond alpha-t}.
\end{proof}

Using the previous results we can show that $G_D^\vee\otimes G_D$ has exactly one nontrivial global section, up to scalar multiplication. We need a lemma first.

\begin{lemma}\label{arfsGWEGWeg}
Let $d\leq 2g+1$ and let $D$, $D'$ be two effective divisors on $C$ of $\deg D=\alpha\leq i$, $\deg D'<d+g-2i-1$. Write $D=x_1+\ldots+x_\alpha$, in arbitrary order and possibly with repetitions. Then for every $k\leq \alpha$ we have $h^0\bigl(M_i(d),(\bigotimes_{j=1}^k F_{x_j}^\vee)\otimes \overline{G}_{D'}\bigr)\leq 1$.
\end{lemma}

\begin{proof}
If $i=0$ then $\alpha=k=0$ and this is given by Corollary \ref{Coro:GD=C}. Let $i\geq 1$, so $d>2$. We do induction on $k$. If $k=0$, this still follows from Corollary \ref{Coro:GD=C}. Otherwise, we use Lemma \ref{BasicKoszul} to get an exact sequence
\begin{align*}
    0\to\bigotimes_{j=1}^{k-1} F_{x_j}^\vee \otimes \overline{G}_{D'}\otimes \Lambda_M^{-1} \to \bigotimes_{j=1}^kF_{x_j}^\vee\otimes \overline{G}_{D'} \to 
    \qquad\qquad\qquad\qquad\qquad \\
    \qquad\qquad\qquad\qquad\qquad
    \to\bigotimes_{j=1}^{k-1}F_{x_j}^\vee\otimes \overline{G}_{D'} \to \left.\bigotimes_{j=1}^{k-1}F_{x_j}^\vee\otimes \overline{G}_{D'}\right|_{M_{i-1}}\to 0
\end{align*}
where $M_{i-1}=M_{i-1}(\Lambda(-2x_k))$. The first term can be seen to be $\Gamma$-acyclic using Theorem \ref{vanishing without Z on Mi}. Indeed, here $t=-1\notin [0,\ k-1]$ and the inequalities $(k-1)-i-1<-1<d+g-2i-1-\deg D'$ are satisfied since $k\leq\alpha\leq i$ and $\deg D' < d+g-2i$. On the other hand, $h^0\bigl(M_i(d),(\bigotimes_{j=1}^{k-1} F_{x_j}^\vee)\otimes \overline{G}_{D'}\bigr)\leq 1$ by induction. Therefore, taking the hypercohomology spectral sequence $E_1^{p,q}=H^q(X,\cF^p)$ of the $\Gamma$-acyclic complex above, we conclude that $h^0\bigl(M_i(d),(\bigotimes_{j=1}^{k} F_{x_j}^\vee)\otimes \overline{G}_{D'}\bigr)\leq 1$ as well.
\end{proof}

\begin{corollary}\label{Hom(D,D)=1}
Suppose $d\leq 2g+1$ and let $0\leq i\leq v$. If $\deg D\leq i$, then $$\Hom_{M_i(d)}(G_{D},G_{D})=\Hom_{M_i(d)}(\overline{G}_{D},\overline{G}_{D})=\mathbb{C}.$$
\end{corollary}

\begin{proof} 
We have
$\Hom_{M_i(d)}(G_D,G_D)=H^0(M_i(d),G_D^\vee\otimes G_D)$. But by Corollary \ref{relation G bar and G}, $G_D^\vee\otimes G_D\simeq \overline{G}_D^\vee\otimes\overline{G}_D$, so $\Hom_{M_i(d)}(G_D,G_D)=\Hom_{M_i(d)}(\overline{G}_D,\overline{G}_D)$ has dimension $h^0(G_D^\vee\otimes G_D)=h^0(\overline{G}_D^\vee\otimes \overline{G}_D)$, which by Corollary \ref{standarddeformation} and semi-continuity, is at most $h^0\bigl(M_i(d),(\bigotimes_{j=1}^{\deg D} F_{x_j}^\vee)\otimes \overline{G}_{D}\bigr)$. But by Lemma~\ref{arfsGWEGWeg}, this dimension is at most $1$, since by hypothesis $\deg D\leq i<d+g-2i-1$. On the other hand, the identity provides a nontrivial map $G_D\to G_D$, so $\dim\Hom_{M_i(d)}(G_D,G_D)$ must be exactly $1$.
\end{proof}

\section{Full faithfulness}\label{fully section}

In this section we construct fully faithful embeddings from $D^b(\Sym^\alpha C)$ to $D^b(M_{i}(\Lambda))$, for $1\leq\alpha\leq i\leq v$ and $d\leq 2g-1$.

\begin{definition}
    For $1\leq \alpha\leq i$, let $\Phi_\alpha^i:D^b(\Sym^\alpha C)\to D^b(M_{i}(\Lambda))$ be the Fourier--Mukai functor determined by $F^{\boxtimes\alpha}\in D^b(\Sym^\alpha C\times M_{i}(\Lambda))$, where $F$ is the universal bundle on $C\times M_{i}(\Lambda)$. Similarly, let $\overline{\Phi}_\alpha^i:D^b(\Sym^\alpha C)\to D^b(M_{i}(\Lambda))$ be the Fourier--Mukai functor given by $\overline{F}^{\boxtimes\alpha}\in D^b(\Sym^\alpha C\times M_{i}(\Lambda))$ (see Definition \ref{tensor products and sums} for  $F^{\boxtimes\alpha}$ and $\overline{F}^{\boxtimes\alpha}$).
\end{definition}

We have already proved in Theorem \ref{Poincare is fully faithful on M1} that $\Phi_1^1=\Phi_F$ is fully faithful. The main result of the present section is a generalization of that result.

\begin{theorem}\label{fully faithfulness}
Suppose $d\leq 2g-1$. For $1\le i\le v$, $1\leq\alpha\leq i$, both $\Phi_\alpha^i$ and $\overline{\Phi}_\alpha^i$ are fully faithful functors. 
\end{theorem}

We will use induction to prove Theorem \ref{fully faithfulness}. First we need to investigate $R\Hom(G_D,G_{D'})$ between different divisors. We want to obtain $\Gamma$-acyclicity of $G_D^\vee\otimes G_{D'}$, for which we need some preliminary computations.

\begin{lemma}\label{generalized bottleneck}
Suppose $d\leq 2g+1$ and let $0\leq i\leq v$. Let $D$, $D'$ be effective divisors on $C$ with $D=\alpha x$ and $x\notin D'$. 
%\begin{enumerate}
%    \item Let $1\leq i\leq v$. If $\alpha+\deg D'\leq d+g-2i-1$, then
%$$
%R\Gamma_{M_i}(G_D^\vee\otimes G_{D'}\otimes\Lambda_M^{\alpha-1})=0.
%$$
%\item Let $0\leq i\leq v$. 
If $\alpha+\deg D'<d+g-2i-1$, then
$$
R\Gamma_{M_i}(G_D^\vee\otimes G_{D'}\otimes\Lambda_M^{\alpha})=\mathbb{C}.
$$
Moreover, the unique (up to a scalar) global section of $G_D^\vee\otimes G_{D'}\otimes\Lambda_M^{\alpha}$ vanishes precisely along the union of codimension~$2$ loci $M_0(\Lambda(-2x))$ and $M_{0}(\Lambda(-2y))$ for $y\in \supp (D')$.
%\end{enumerate}
\end{lemma}

\begin{proof}%[Proof of the bottleneck case]
%We first show the second claim, for which 
We use the fact that $G_D^\vee\otimes G_{D'}\otimes\Lambda_M^\alpha$ is a deformation over $\bA^1$ of $(F_x^\vee)^{\otimes \alpha}\otimes\bigotimes_{k=1}^{\deg D'}F_{y_k}\otimes\Lambda_M^\alpha\simeq F_x^{\otimes \alpha}\otimes\bigotimes_{k=1}^{\deg D'}F_{y_k}$, where $D'=\sum y_k$. By Corollary \ref{Coro:GD=C}, we see that $R\Gamma_{M_i}(F_x^{\otimes \alpha}\otimes\bigotimes_{k=1}^{\deg D'}F_{y_k})=\mathbb{C}$, so by semi-continuity and equality of the Euler characteristic, we must have $R\Gamma_{M_i}(G_D^\vee\otimes G_{{D'}}\otimes\Lambda_M^\alpha)=\mathbb{C}$ as well. Furthermore, the global section of $G_D^\vee\otimes G_{D'}\otimes\Lambda_M^\alpha$ is a deformation of the global section of $F_x^{\otimes \alpha}\otimes\bigotimes_{k=1}^{\deg D'}F_{y_k}$ over $\bA^1$, which does not vanish outside of the union of loci $M_{0}(\Lambda(-2x))$ and $M_0(\Lambda(-2y_k))$. On the other hand, the tautological section of this bundle vanishes precisely along these loci.
\end{proof}

\begin{lemma}\label{lemma Z}
Suppose $2<d\leq 2g+1$ and $1\leq i\leq v$. Let $D$, $D'$ be effective divisors with $D=\alpha x$ and $D'=\beta x+\Tilde{D}'$, $x\notin \Tilde{D}'$. Suppose $\alpha=\deg D\leq i$ and $\deg D'<d+g-2i-1$. Then the map $R\Gamma_{M_i(d)}(\overline{G}_{\alpha x}^\vee\otimes \overline{G}_{\beta x})\to R\Gamma_{M_i(d)}(\overline{G}_{\alpha x}^\vee\otimes \overline{G}_{\beta x}\otimes\overline{G}_{\Tilde{D}'})$ given by tensoring with the universal section of $\overline{G}_{\Tilde{D}'}$ (cf. Corollary \ref{Coro:GD=C}) is an isomorphism.
\end{lemma}

\begin{proof}
We argue by induction on $\alpha$. If $\alpha=0$, this is clear, as the map $R\Gamma_{M_i(d)}(\overline{G}_{\beta x})\to R\Gamma_{M_i(d)}(\overline{G}_{\beta x}\otimes\overline{G}_{\Tilde{D}'})$ is $\mathbb{C}\xrightarrow{\sim}\mathbb{C}$ (cf. Corollary \ref{Coro:GD=C}).

For the inductive step, we argue as in the proof of Proposition \ref{bottleneck}, specifically as in Lemma \ref{lemma E}:
$
\overline{G}_{\alpha x}^\vee\simeq \Lambda_M^{-\alpha}\otimes G_{\alpha x} = \Lambda_M^{-\alpha}\otimes\tau_*^{S_{\alpha}}\left(\bigotimes_{j=1}^\alpha\pi_j^*\cF\right),
$
which is a direct summand in
\begin{equation}\label{Sa-1}
\tau_*^{S_{\alpha-1}}\left(\bigotimes_{j=1}^\alpha\pi_j^*\cF\right).
\end{equation}
Here $\cF=q^*F=q_1^*F$ from \eqref{bigger diagram tau}.
So it suffices to prove our claim for the bundle \eqref{Sa-1}. As in the proof of Lemma \ref{lemma E}, we have an exact sequence
\begin{gather}
\begin{aligned}\label{Sequence lemma Z}
0\to\Lambda_M^{-1}\otimes\bigotimes_{j=1}^{\alpha-1}\pi_j^*\cF\to\Lambda_M^{-1}\otimes\bigotimes_{j=1}^{\alpha}\pi_j^*\cF\to
\qquad\qquad\qquad\qquad\qquad\\
\qquad\qquad\qquad\qquad\qquad
\to\bigotimes_{j=1}^{\alpha-1}\pi_j^*\cF\to \left.\bigotimes_{j=1}^{\alpha-1}\pi_j^*\cF\right|_{\mathbb{B}_\alpha\times M_i(\Lambda(-2x))}\to 0,
\end{aligned}
\end{gather}
to which we apply $\tau_{*}^{S_{\alpha-1}}$, then tensor with $\Lambda_M^{1-\alpha}\otimes\overline{G}_{\beta x}$ (resp. with $\Lambda_M^{1-\alpha}\otimes\overline{G}_{\beta x}\otimes\overline{G}_{\Tilde{D}'}$) and then compute $R\Gamma$. The resulting left term is a deformation of $\alpha$ copies of $\Lambda_M^{-1}\otimes\overline{G}_{(\alpha-1)x}^\vee\otimes \overline{G}_{\beta x}$ (resp. $\Lambda_M^{-1}\otimes\overline{G}_{(\alpha-1)x}^\vee\otimes \overline{G}_{\beta x}\otimes \overline{G}_{\Tilde{D}'}$), both of which are $\Gamma$-acyclic by Theorem \ref{vanishing without Z on Mi}.

Therefore, we have two exact triangles related by a commutative diagram:
\begin{equation}\label{exact triangles}
\begin{tikzcd}[cramped,row sep=scriptsize, column sep=tiny]
R\Gamma(\Lambda_M^{-\alpha}\otimes U\otimes \pi_\alpha^*\cF\otimes\overline{G}_{\beta x})^{S_{\alpha-1}}\arrow[r]\arrow[d] 
& R\Gamma(\Lambda_M^{-\alpha}\otimes U\otimes \pi_\alpha^*\cF\otimes\overline{G}_{\beta x}\otimes\overline{G}_{\Tilde{D}'})^{S_{\alpha-1}} \arrow[d]\\
R\Gamma(\Lambda_M^{1-\alpha}\otimes U\otimes \overline{G}_{\beta x})^{S_{\alpha-1}} \arrow[d]\arrow[r] 
&R\Gamma(\Lambda_M^{1-\alpha}\otimes U\otimes \overline{G}_{\beta x}\otimes\overline{G}_{\Tilde{D}'})^{S_{\alpha-1}}\arrow[d]\\ 
R\Gamma(\left.\Lambda_M^{1-\alpha}\otimes U\otimes \overline{G}_{\beta x}\right|_{\mathbb{B}_\alpha\times M'})^{S_{\alpha-1}} \arrow[r]\arrow[d]& R\Gamma(\left.\Lambda_M^{1-\alpha}\otimes U\otimes \overline{G}_{\beta x}\otimes\overline{G}_{\Tilde{D}'}\right|_{\mathbb{B}_\alpha\times M'})^{S_{\alpha-1}}\arrow[d]\\
\phantom{.} &\phantom{.}\\
\end{tikzcd}
\end{equation}
where $U=\bigotimes_{j=1}^{\alpha-1}\pi_j^*\cF$, $M'=M_i(\Lambda(-2x))$ and the horizontal maps are multiplication by the universal section of $\overline{G}_{\Tilde{D}'}$. The middle row of \eqref{exact triangles} is a deformation of $\alpha$ copies of the map $R\Gamma_{M_i(d)}(\overline{G}_{(\alpha-1) x}^\vee\otimes \overline{G}_{\beta x})\to R\Gamma_{M_i(d)}(\overline{G}_{(\alpha-1) x}^\vee\otimes \overline{G}_{\beta x}\otimes\overline{G}_{\Tilde{D}'})$, which is an isomorphism by the induction assumption. The same is true for the third row, on the moduli space $M_i(\Lambda(-2x))$. We conclude that the first row of \eqref{exact triangles} must also be an isomorphism, which completes the proof.
\end{proof}

\begin{lemma}\label{lemma fully faithful}
Suppose $2<d\leq 2g+1$ and $1\leq i\leq v$. Let $D$, $D'$ be effective divisors with $D=\alpha x$ and $\mult_x(D')\leq \alpha-1$. Suppose $\alpha=\deg D\leq i$ and $\deg D'<d+g-2i-1$. If we assume that $\overline{\Phi}_{\alpha'}^i$ 
and $\overline{\Phi}^{i-1}_{\alpha'}$
are fully faithful for every $\alpha'<\alpha$, then $R\Gamma_{M_i(d)}(\overline{G}_D^\vee\otimes\overline{G}_{D'})=0$.
\end{lemma}

\begin{proof}
By Lemma \ref{lemma Z}, it suffices to consider the case $D'=\beta x$, where $\beta<\alpha$. Moreover, arguing as in Lemma \ref{lemma Z}, we can assume that $\alpha=\beta+1$, so it suffices to show that $R\Gamma_{M_i(d)}(\overline{G}_{\alpha x}^\vee,\overline{G}_{(\alpha-1)x})=0$ under the assumptions $\alpha\leq i$, $\alpha<d+g-2i$. As in Lemma \ref{lemma Z}, we consider the exact sequence \eqref{Sequence lemma Z}, twist it by $\Lambda_M^{1-\alpha}\otimes\overline{G}_{(\alpha-1) x}$ and take $S_{\alpha-1}$-invariant global sections. The resulting term on the left vanishes by semi-continuity and Theorem \ref{vanishing without Z on Mi}. It suffices to show that the second term vanishes, because it contains $R\Gamma_{M_i(d)}(\overline{G}_{\alpha x}^\vee\otimes\overline{G}_{(\alpha-1)x})$ as a direct summand. But the last two terms are deformations over $\mathbb{A}^1$ of $\alpha$ copies of the map $R\Hom_{M_i(d)}(\overline{G}_{(\alpha-1)x},\overline{G}_{(\alpha-1)x})\to R\Hom_{M_{i-1}(d-2)}(\overline{G}_{(\alpha-1)x},\overline{G}_{(\alpha-1)x})$, which is an isomorphism by our assumption that $\overline{\Phi}^i_{\alpha-1}$ 
and $\overline{\Phi}^{i-1}_{\alpha-1}$
are fully faithful. This completes the proof.
\end{proof}

\begin{theorem}\label{proof of fully faithfulness}
Suppose $2<d\leq 2g+1$ and $1\leq i\leq v$. Let $D$, $D'$ be effective divisors on~$C$, with $D\not\leq D'$ and satisfying $\deg D\leq i$ and $\deg D'< d+ g-2i-1$. If we assume that $\overline{\Phi}_{\alpha'}^i$ is fully faithful for every $\alpha'<\alpha$, then $R\Gamma_{M_i(d)}(\overline{G}_D^\vee\otimes\overline{G}_{D'})=0$.
\end{theorem}

\begin{proof}
We do induction on $\deg D$. If $\deg D=1$, then we have $D=x$ and $\mult_x(D')=0$, so the result follows from Lemma \ref{lemma fully faithful} with $\alpha=1$.

Let $\deg D>1$, and so $i>1$ as well. Since $D\not\leq D'$, there is a point $x\in D$ with $\mult_x(D)=\alpha$, $\mult_x(D')\leq \alpha-1$. If $\supp (D)=\{x\}$, then $D=\alpha x$ is a fat point and the result follows  from Lemma \ref{lemma fully faithful}. Otherwise, we can find a point $y\neq x$ such that $\Tilde{D}=D-y$ is effective. From \eqref{reduction F dual}, we get an exact sequence
$$
0\to \overline{G}_{\Tilde{D}}^\vee\otimes \overline{G}_{D'}\otimes\Lambda_M^{-1}\to F_y^\vee\otimes \overline{G}_{\Tilde{D}}^\vee\otimes \overline{G}_{D'}\to \qquad\qquad\qquad\qquad\qquad\qquad
$$
$$
\qquad\qquad\qquad\qquad\qquad\qquad
\to \overline{G}_{\Tilde{D}}^\vee\otimes \overline{G}_{D'}\to \left.\overline{G}_{\tilde{D}}^\vee\otimes \overline{G}_{D'}\right|_{M_{i-1}(d-2)}\to 0.
$$
By induction, $R\Gamma_{M_i(d)}(\overline{G}_{\Tilde{D}}^\vee\otimes \overline{G}_{D'})=R\Gamma_{M_{i-1}(d-2)}(\overline{G}_{\Tilde{D}}^\vee\otimes \overline{G}_{D'})=0$. On the other hand, the term $\overline{G}_{\Tilde{D}}^\vee\otimes \overline{G}_{D'}\otimes\Lambda_M^{-1}$ satisfies the inequalities \eqref{ineq vanishing without Z} with $t=-1\notin [0, \ \deg \Tilde{D}]$, so by Theorem \ref{vanishing without Z on Mi} it is $\Gamma$-acyclic. As usual, the result follows from the hypercohomology spectral sequence and semi-continuity.
\end{proof}

Now we can prove the main result of this section.

\begin{proof}[Proof of Theorem \ref{fully faithfulness}]
By Bondal-Orlov's criterion \cite{bondal-orlov}, we only need to consider the images of skyscraper sheaves, $\Phi_\alpha^i(\mathcal{O}_{\{D\}})=G_D$ and $\overline{\Phi}_\alpha^i(\mathcal{O}_{\{D\}})=\overline{G}_D$. Namely, we need to show that for two divisors $D,\ D'\in \Sym^\alpha C$ we have
\begin{align}\label{cohomologies fully faithfulness}
    R^k\Gamma_{M_{i}(\Lambda)}(\overline{G}_{D}^\vee\otimes \overline{G}_{D'})=
    \begin{cases}
    0 \quad & \text{if $D\neq D'$ or $k<0$ or $k>\alpha$}\\
    \mathbb{C} &\text{if $k=0$ and $D=D'$}
    \end{cases}
\end{align}
and similarly for $R\Gamma_{M_i(\Lambda)}({G}_{D}^\vee\otimes{G}_{D'})$. Observe that since $R\Gamma_{M_i(\Lambda)}(\overline{G}_D^\vee\otimes\overline{G}_{D'})=R\Gamma_{M_i(\Lambda)}({G}_{D'}^\vee\otimes{G}_{D})$, full faithfulness of $\Phi_\alpha^i$ is equivalent to that of $\overline{\Phi}_\alpha^i$, and it suffices to prove \eqref{cohomologies fully faithfulness}. We prove it by induction on $\alpha$, the case $\alpha=0$ being trivial, and $\alpha=1$ is Theorem \ref{Poincare is fully faithful on M1}. So we assume \eqref{cohomologies fully faithfulness} holds for $\alpha'<\alpha$. 
If $D=D'$ then \eqref{cohomologies fully faithfulness} follows directly from Corollary \ref{coro:vanishing beyond alpha} and Corollary \ref{Hom(D,D)=1}. Now let $D\neq D'$ be different divisors of degree $\alpha\leq i$. Notice $i\leq (d-1)/2\leq g-1$, so the inequality $\alpha \leq d+g-2i-2$ holds. Therefore, in this case \eqref{cohomologies fully faithfulness} follows from Theorem \ref{proof of fully faithfulness} by our induction hypothesis. We conclude that $\Phi_\alpha^i$ and $\overline{\Phi}_\alpha^i$ are fully faithful functors.
\end{proof}

\section{Proof of the semi-orthogonal decomposition}\label{ProofOfTheSOD}

Throughout this section we fix $d=\deg\Lambda=2g-1$, so that $v=(d-1)/2=g-1$. We are interested in the moduli spaces $M_i=M_{i}(\Lambda)$, where $i$ will always be assumed to satisfy $1\leq i\leq g-1$. Note that when $d=2g-1$, the canonical bundle is $\omega_{M_i}=\mathcal{O}_i(-3,3-3g)=\Lambda_M^{-1}\otimes\zeta^{-1}\otimes\theta^{-1}$ (see \cite[6.1]{thaddeus} and Definition \ref{notation important line bundles}).

By abuse of notation, we will denote the essential image $\Phi_\alpha^i (\Sym^\alpha C)$ simply by $\Phi_\alpha^i$, and the image $\overline{\Phi}_\alpha^i (\Sym^\alpha C)$ by $\overline{\Phi}_\alpha^i$, which by Theorem \ref{fully faithfulness} are admissible subcategories of $D^b(M_i)$ equivalent to $D^b(\Sym^\alpha C)$. Similarly, we will denote by $\Phi_0^i$ the full triangulated subcategory generated by $\mathcal{O}_{M_i}$, which is an admissible subcategory equivalent to $D^b(\pt)$, since $M_i$ is a rational variety. It can be described as the image of the (derived) pullback functor from a point, $\Phi_0^i=q^*$, $q:M_i\to\pt=\Sym^0C$.

\begin{definition}\label{definition ABCD}
We define the following full triangulated subcategories of $D^b(M_i)$:
\begin{align*}
    \mathcal{A}_{2k}&:=\Phi_{2k}^i\otimes\Lambda_M^{-k}\otimes\theta^{-1},\quad &0\leq 2k\leq i\phantom{.}\\
    \mathcal{B}_{2k}&:=\Phi_{2k}^i\otimes\Lambda_M^{-k},&0\leq 2k\leq i\phantom.{}\\
    \mathcal{C}_{2k+1}&:=\overline{\Phi}_{2k+1}^i\otimes\Lambda_M^{-k}\otimes\zeta\otimes\theta^{-1},&0\leq 2k+1\leq i\phantom{.}\\
    \mathcal{D}_{2k+1}&:=\overline{\Phi}_{2k+1}^i\otimes\Lambda_M^{-k}\otimes\zeta,&0\leq 2k+1\leq i.\\
\end{align*}
\end{definition}

Each of these subcategories is equivalent to some $D^b(\Sym^\alpha C)$ with either $\alpha=2k$ or $\alpha=2k+1$. These four families of subcategories constitute the building blocks of our semi-orthogonal decomposition on $D^b(M_i)$. We will see that different subcategories of the form $\mathcal{A}_{2k}$ are semi-orthogonal to each other, and the same is true for subcategories within the other three blocks. We need the following lemma.

\begin{lemma}\label{lemma spanning class}
Let
$\mathcal{D}_1$, $\mathcal{D}_2$ be admissible subcategories of
a  triangulated category~$\cD$ and  $\Omega_1$, $\Omega_2$ spanning classes  \cite[\S 3.2]{huybrechts} of $\mathcal{D}_1$, $\mathcal{D}_2$. 

If $\Hom_\mathcal{D}(A,B[k])=0$ for every $A\in\Omega_1$, $B\in\Omega_2$ and $k\in\mathbb{Z}$, then  $\Hom_\mathcal{D}(F,G)=0$ for every $F\in\mathcal{D}_1$, $G\in\mathcal{D}_2$.
\end{lemma}

\begin{proof}
We need to show that $\mathcal{D}_1\subset {}^{\perp}\mathcal{D}_2$ or, equivalently, $\mathcal{D}_2\subset\mathcal{D}_1^\perp$. 

First we see that $\Omega_1\subset {}^{\perp}\mathcal{D}_2$. Let $A\in\Omega_1$. Since $\mathcal{D}=\langle\mathcal{D}_2,{}^{\perp}\mathcal{D}_2\rangle$, we can fit $A$ in a exact triangle $D\to A\to D'\to D[1]$ where $D\in{}^{\perp}\mathcal{D}_2$ and $D'\in\mathcal{D}_2$. Applying $\Hom(\cdot,B)$ for $B\in\Omega_2$ we get a long exact sequence where $\Hom (D,B[k])=0$ by definition and $\Hom (A,B[k])=0$ by hypothesis. Therefore $\Hom (D',B[k])=0$ for every $k$ and every $B\in\Omega_2$, so $D'\simeq 0$ since $\Omega_2$ is a spanning class of $\mathcal{D}_2$. As a consequence, $A\simeq D\in {}^{\perp}\mathcal{D}_2$.

Now let $G\in\mathcal{D}_2$. Similarly, there is an exact triangle $D\to G\to D'\to D[1]$ with $D\in\mathcal{D}_1$, $D'\in\mathcal{D}_1^\perp$. Applying $\Hom (A,\cdot)$ with $A\in \Omega_1$ we now see that $\Hom (A,D[k])=\Hom (A,G[k])=0$ by the previous discussion and therefore $D'\simeq 0$. This implies $G\simeq D\in\mathcal{D}_1^\perp$, as desired.
\end{proof}

\begin{proposition}\label{orhtogonal within blocks}
Let $k> l$ and $0\leq 2l< 2k \leq i$. Then
\begin{align*}
    \Hom_{D^b(M_i)}(\mathcal{A}_{2k},\mathcal{A}_{2l})=0,\quad &\Hom_{D^b(M_i)}(\mathcal{B}_{2k},\mathcal{B}_{2l})=0.
\end{align*}
Similarly, if $k< l$ and $0\leq 2k+1<2l+1 \leq i$, we have
\begin{align*}
    \Hom_{D^b(M_i)}(\mathcal{C}_{2k+1},\mathcal{C}_{2l+1})=0,\quad &\Hom_{D^b(M_i)}(\mathcal{D}_{2k+1},\mathcal{D}_{2l+1})=0.
\end{align*}
\end{proposition}

\begin{proof}
Let us first show semi-orthogonality between subcategories of the form $\mathcal{A}_{2k}$, $\mathcal{A}_{2l}$, $k> l$, as well as semi-orthogonality between those of the form $\mathcal{B}_{2k}$, $\mathcal{B}_{2l}$, $k> l$. Since skyscraper sheaves $\mathcal{O}_{\{D\}}$ of closed points $D\in \Sym^\alpha C$ are a spanning class of $D^b(\Sym^\alpha C)$ (see  \cite[Proposition 3.17]{huybrechts}), Lemma \ref{lemma spanning class} says that semi-orthogonality can be checked on closed points. That is, it suffices to show that for $D\in\Sym^{2k}C$, $D'\in\Sym^{2l}C$, with $0\leq 2l< 2k\leq i\leq g-1$, we have
$
    R\Gamma_{M_i}(G_D^\vee\otimes G_{D'}\otimes\Lambda_M^{k-l})=0.
$
But this follows from Theorem \ref{vanishing without Z on Mi} (and Remark \ref{remark hard vanishing}). Indeed, the inequalities
$$
2k-i-1<k-l<d+g-2i-1-2l
$$
are equivalent to $k+l<i+1$ and $k+l+2i<d-1+g$, which are guaranteed by the fact that $k+l<i\leq (d-1)/2<g$ in this case. Also, since $k>l$ we have $2k\notin [k-l,\ k+l]$. Notice that all divisors involved have degree $\leq g-1<d+g-2i-1$. This proves the first two semi-orthogonality statements.

Similarly, in order to prove semi-orthogonality between subcategories $\mathcal{C}_{2k+1}$, $\mathcal{C}_{2l+1}$, $k< l$, as well as between $\mathcal{D}_{2k+1}$, $\mathcal{D}_{2l+1}$, $k< l$, we need to prove that for $D\in\Sym^{2k+1}C$, $D'\in\Sym^{2l+1}C$, with $0\leq 2k+1<2l+1\leq i\leq g-1$, we must have
\begin{align*}
    R\Gamma_{M_{i}}(\overline{G}_D^\vee\otimes \overline{G}_{D'}\otimes\Lambda_M^{k-l})=0.
\end{align*}
Again, this can be proved using Theorem \ref{vanishing without Z on Mi}: the inequalities
$$
2k+1-i-1<k-l<d+g-1-2i-(2l+1)
$$
are equivalent to $k+l<i$ and $k+l+2i<d+g-2$, both of which follow from the fact that $k+l+1<i\leq (d-1)/2<g$ in this case. Similarly, 
$k<l$ implies $k-l\notin [0,\ 2k+1]$.
%$k>l$ implies $2k+1\notin [k-l,\ k+l+1]$. 
This proves the required vanishing.
\end{proof}

\begin{theorem}\label{the sod on Mi}
Let $d=2g-1$ and $1\leq i\leq g-1$. On $D^b(M_i)$, we have a semi-orthogonal list of admissible subcategories arranged in four blocks
\begin{align}\label{four blocks}
\mathcal{A},\mathcal{C},\mathcal{B},\mathcal{D}    
\end{align}
where
\begin{align*}
\mathcal{A}&=\langle\mathcal{A}_{2k}\rangle_{0\leq 2k\leq i}
&\mathcal{C}&=\langle\mathcal{C}_{2k+1}\rangle_{1\leq 2k+1\leq i}\\
\mathcal{B}&=\langle\mathcal{B}_{2k}\rangle_{0\leq 2k\leq \min (i,\ g-2)}
&\mathcal{D}&=\langle\mathcal{D}_{2k+1}\rangle_{1\leq 2k+1\leq \min (i,\ g-2)}
\end{align*}
%\begin{align*}
%\mathcal{A}&=\mathcal{A}_0,\mathcal{A}_2,\mathcal{A}_4,\ldots &\text{for $0\leq 2k\leq i$\phantom{.}}\\
%\mathcal{C}&=\mathcal{C}_1,\mathcal{C}_3,\mathcal{C}_5,\ldots &\text{for $1\leq 2k+1\leq i$\phantom{.}}\\
%\mathcal{B}&=\mathcal{B}_0,\mathcal{B}_2,\mathcal{B}_4,\ldots &\text{for $0\leq 2k\leq \min (i,\ g-2)$\phantom{.}}\\
%\mathcal{D}&=\mathcal{D}_1,\mathcal{D}_3,\mathcal{D}_5,\ldots &\text{for $1\leq 2k+1\leq \min (i,\ g-2)$,}
%\end{align*}
as given in Definition \ref{definition ABCD}. 
Within the blocks $\mathcal{A}$ and $\mathcal{B}$, the subcategories are arranged in increasing order of $k$. Within the blocks $\mathcal{C}$ and $\mathcal{D}$, the subcategories are arranged in decreasing order of $k$.
%Within each of the four blocks in \eqref{four blocks}, the corresponding admissible subcategories can be rearranged in any order.
\end{theorem}

\begin{proof}
All of these are admissible subcategories of $D^b(M_i)$ by Theorem \ref{fully faithfulness}, and we have already shown in Proposition \ref{orhtogonal within blocks} that, within each of the four blocks in \eqref{four blocks}, the corresponding subcategories are semi-orthogonal in the given order. It remains to prove semi-orthogonality between different blocks.

\begin{step}
Between $\mathcal{A}$ and $\mathcal{C}$: we show that $\Hom_{D^b(M_i)}(\mathcal{C}_{2k+1},\mathcal{A}_{2l})=0$. By Lemma \ref{lemma spanning class}, this amounts to showing that 
$$
R\Gamma_{M_i} (\overline{G}_D^\vee\otimes G_{D'}\otimes\Lambda_M^{k-l}\otimes\zeta^{-1})=0
$$
for $D\in\Sym^{2k+1}C$, $D'\in\Sym^{2l}C$, with $0\leq 2k+1,2l\leq i\leq (d-1)/2=g-1$. We can apply Theorem \ref{vanishing with Z on Mi} (and Remark \ref{remark easy vanishing}) since the inequalities
$$
2k+1-g<k-l<d-2l-i-1
$$
are equivalent to $k+l<g-1$ and $k+l+i<d-1$, which hold in this case as $k+l<i\leq (d-1)/2 = g-1$. This gives the corresponding semi-orthogonality.
\end{step}

\begin{step}
Between $\mathcal{A}$ and $\mathcal{B}$: let us show $\Hom_{D^b(M_i)}(\mathcal{B}_{2k},\mathcal{A}_{2l})=0$. Again by Lemma \ref{lemma spanning class}, we need to show
$
R\Gamma_{M_i} (G_D^\vee\otimes G_{D'}\otimes\Lambda_M^{k-l}\otimes\theta^{-1})=0
$
when $D\in\Sym^{2k}C$, $D'\in\Sym^{2l}C$, $0\leq 2k,2l\leq i\leq(d-1)/2=g-1$ and $2k\leq g-2$. By Serre duality, given that $\omega_{M_i}=\Lambda_M^{-1}\otimes\zeta^{-1}\otimes\theta^{-1}$, this is equivalent to showing that
$
G_{D'}^\vee\otimes G_{D}\otimes\Lambda_M^{l-k-1}\otimes\zeta
^{-1}
$
is $\Gamma$-acyclic on $M_i$ under the conditions above. This is given by Theorem \ref{vanishing with Z on Mi}, because
$$
2l-g<l-k-1<d-2k-i-1
$$
is equivalent to $l+k<g-1$ and $l+k+i<d$, and these inequalities hold since $l+k+i\leq 2i \leq d-1$ and $2l+2k\leq g-1+g-2$ in this case.
\end{step}

\begin{step}
Between $\mathcal{A}$ and $\mathcal{D}$.
For $\Hom_{D^b(M_i)}(\mathcal{D}_{2k+1},\mathcal{A}_{2l})$, we need to show that
$
R\Gamma_{M_i} (\overline{G}_D^\vee\otimes G_{D'}\otimes \Lambda_M^{k-l}\otimes\zeta^{-1}\otimes\theta^{-1})=0
$
whenever $D\in\Sym^{2k+1}C$, $D'\in\Sym^{2l}C$, $0\leq 2l,2k+1\leq i\leq (d-1)/2=g-1$. Again by Serre duality, this is equivalent to $\Gamma$-acyclicity of 
$
G_{D'}^\vee\otimes \overline{G}_{D}\otimes\Lambda_M^{l-k-1}.
$

If $l\leq k$, we check that this is given by Theorem \ref{vanishing without Z on Mi}. Indeed, the corresponding inequalities
$$
2l-i-1<l-k-1<d+g-2i-1-(2k+1)
$$
are equivalent to $k+l<i$ and $l+k+2i<d+g-1$. The former follows from $2l,2k+1\leq i$ and the latter follows from $l+k <i<g$ and $2i\leq d-1$.
Also, the fact that $k\geq l$ implies $l-k-1\notin [0,\ 2l]$.

On the other hand, if $l>k$, we rewrite $
G_{D'}^\vee\otimes \overline{G}_{D}\otimes\Lambda_M^{l-k-1}\simeq G_D^\vee\otimes \overline{G}_{D'}\otimes\Lambda_M^{k-l}
$
using Corollary \ref{relation G bar and G}. Again, we can use Theorem \ref{vanishing without Z on Mi}. Indeed, we see that the inequalities
$$
(2k+1)-i-1<k-l<d+g-2i-1-2l
$$
are equivalent to the ones above and hence are satisfied, while now $l>k$ guarantees $k-l\notin [0,2k+1]$.
Thus, Theorem \ref{vanishing without Z on Mi} gives the required $\Gamma$-acyclicity.
\end{step}

\begin{step}
Next we show semi-orthogonality between $\mathcal{C}$ and $\mathcal{B}$. This amounts to $\Gamma$-acyclicity of
$
G_D^\vee\otimes \overline{G}_{D'}\otimes\Lambda_M^{k-l}\otimes\zeta\otimes\theta^{-1}=G_D^\vee\otimes \overline{G}_{D'}\otimes\Lambda_M^{k-l-1}\otimes\zeta^{-1}
$
(cf. Definition \ref{notation important line bundles}) for $D\in \Sym^{2k}C$, $D'\in\Sym^{2l+1}C$, where $0\leq 2k,2l+1\leq i\leq (d-1)/2=g-1$. We check that Theorem \ref{vanishing with Z on Mi} can be applied in this case:
$$
2k-g<k-l-1<d-(2l+1)-i-1
$$
is equivalent to $k+l<g-1$ and $k+l+i<d-1$, both of which hold in our case. This proves $\Hom_{D^b(M_i)}(\mathcal{B}_{2k},\mathcal{C}_{2l+1})=0$.
\end{step}

\begin{step}
To show that $\Hom_{D^b(M_i)}(\mathcal{D}_{2k+1},\mathcal{C}_{2l+1})=0$, we need to check that
$
\overline{G}_D^\vee\otimes \overline{G}_{D'}\otimes\Lambda_M^{k-l}\otimes\theta^{-1}
$
is $\Gamma$-acyclic on $M_i$, where $D\in\Sym^{2k+1}C$, $D'\in\Sym^{2l+1}C$, $1\leq 2k+1,2l+1 \leq i\leq (d-1)/2=g-1$ and $2k+1\leq g-2$. By Serre duality, this is equivalent to $\Gamma$-acyclicity of
$
\overline{G}_{D'}^\vee\otimes \overline{G}_D\otimes\Lambda_M^{l-k-1}\otimes\zeta^{-1}
$
and this follows from Theorem \ref{vanishing with Z on Mi} since
$$
2l+1-g<l-k-1<d-(2k+1)-i-1
$$
is equivalent to $l+k+1<g-1$ and $l+k+i<d-1$, both of which hold given the conditions above.
\end{step}

\begin{step}
Finally, we show semi-orthogonality between blocks from $\mathcal{B}$ and $\mathcal{D}$. We need to show that if $D\in\Sym^{2k+1}C$, $D'\in\Sym^{2l}C$, $0\leq 2k+1,2l\leq i\leq (d-1)/2=g-1$, we have
$
R\Gamma_{M_i}(\overline{G}_D^\vee\otimes G_{D'}\otimes\Lambda_M^{k-l}\otimes\zeta^{-1})=0.
$ 
We can use Theorem \ref{vanishing with Z on Mi} since
$$
2k+1-g<k-l<d-2l-i-1
$$
is equivalent to the inequalities $k+l<g-1$ and $k+l+i<d-1$, again both of which hold in our situation. We conclude $\Hom_{D^b(M_i)}(\mathcal{D}_{2k+1},\mathcal{B}_{2l})=0$. 
\end{step}
This completes the proof of the theorem.
\end{proof}

\begin{remark}
On $D^b(M_{g-1})$, this defines a semi-orthogonal list of admissible subcategories
$
\mathcal{A}_0,\mathcal{A}_2,\ldots,\ldots\mathcal{C}_3,\mathcal{C}_1,
\mathcal{B}_0,\mathcal{B}_2,\ldots,\ldots\mathcal{D}_3,\mathcal{D}_1
$
where we have two copies of $D^b(\Sym^\alpha C)$ for $0\leq\alpha\leq g-2$ and one copy of $D^b(\Sym^{g-1}C)$. We have chosen $D^b(\Sym^{g-1}C)$ to appear in the block $\mathcal{A}$ when $g-1$ is even and in $\mathcal{C}$ when $g-1$ is odd, but in fact any other choice of even and odd blocks would be valid too. Indeed, a similar computation in the proof of Theorem \ref{the sod on Mi} still gives the required semi-orthogonalities.
\end{remark}

Now let $i=g-1$, and call $\xi:M_{g-1}\to N$ the last map in (\ref{diagram Mi}), where $N=M_C(2,\Lambda)$ is the space of stable rank-two vector bundles  of odd degree. The Picard group of $N$ is generated by an ample line bundle $\theta_N$, such that $\xi^*\theta_N=\theta$ (see \cite[5.8, 5.9]{thaddeus} and \cite[Proposition 2.1]{narasimhan1}). Then we have the following corollary.

\begin{corollary}\label{rearrangement}
Let $\mathcal{E}$ be the Poincar\'e bundle of the moduli space $N=M_C(2,\Lambda)$ over a curve of genus $\geq 3$, normalized so that $\det\pi_!\mathcal{E}=\mathcal{O}_N$ and $\det\mathcal{E}_x=\theta_N$, and where $\Lambda$ is a line bundle on $C$ of arbitrary odd degree.
%(cf. \cite{narasimhan1}).
For $i=0,\ldots, g-1$, let $\cG_i\subset D^b(N)$ (resp. $\overline{\cG}_i$) be the essential image of the Fourier--Mukai functor with kernel $\cE^{\boxtimes i}$ (resp. $\overline{\cE}^{\boxtimes i}$).
Then 
\begin{equation}\label{SODonNrearr}
\begin{matrix}
&\theta_N^*\otimes\cG_0,
&(\theta_N^*)^2\otimes\cG_2,
&(\theta_N^*)^3\otimes\cG_4,
&(\theta_N^*)^4\otimes\cG_6,
&\ldots&&
\\
&\ldots,
&(\theta_N^*)^4\otimes\overline{\cG}_7,
&(\theta_N^*)^3\otimes\overline{\cG}_5,
&(\theta_N^*)^2\otimes\overline{\cG}_3,
&\theta_N^*\otimes\overline{\cG}_1&&
\\
&\cG_0,
&\theta_N^*\otimes\cG_2,
&(\theta_N^*)^2\otimes\cG_4,
&(\theta_N^*)^{3}\otimes\cG_6,
&\ldots&&
\\
&\ldots,
&(\theta_N^*)^3\otimes\overline{\cG}_7,
&(\theta_N^*)^2\otimes\overline{\cG}_5,
&\theta_N^*\otimes\overline{\cG}_3,
&\overline{\cG}_1&&
\\
\end{matrix}
\end{equation}
is a semi-orthogonal sequence of admissible subcategories of $D^b(N)$.
There are two blocks isomorphic to
$D^b(\Sym^i C)$ for each $i=0,\ldots,g-2$ and one block isomorphic to $D^b(\Sym^{g-1}C)$.
%Within each of the four lines in \eqref{SODonNrearr}, the corresponding admissible subcategories can be rearranged in any order.
\end{corollary}

\begin{proof}
If $\Lambda$, $\Lambda'$ are two line bundles of odd degree, it is easy to see that $M_C(2,\Lambda)\simeq M_C(2,\Lambda')$, so we can assume $d=\deg\Lambda=2g-1$, as before.
Observe that $\xi^*$ is fully faithful. Indeed, $\xi$ is a projective birational morphism of nonsingular varieties, so we have $R\xi_*(\mathcal{O}_{M_{g-1}})=\mathcal{O}_N$ by \cite[5.12]{thaddeus} and \cite[(2), pp.144-145]{hironaka}. Then by adjointness $$\Hom_{D^b(M_{g-1})}(\xi^*A,\xi^*B)=\Hom_{D^b(N)}(A,R\xi_*\xi^*B)=\Hom_{D^b(N)}(A,B).$$ The pullback $\xi^*(\mathcal{E})$ is a vector bundle on $C\times M_{g-1}$ whose restriction to each $C\times\{(E,\phi)\}\subset C\times M_{g-1}$ is exactly $\cE$. Thus, it has to coincide with the universal bundle $F$ up to twist by a line bundle on $M_{g-1}$, so that $\xi^*\mathcal{E}=F\otimes L$. Then $\xi^*\det \mathcal{E}_x=\Lambda_M\otimes L^2$, which by the normalization chosen must be $\xi^*\theta_N=\theta$, so $L=\zeta$. Thus $\xi^*(\mathcal{E})=F\otimes \zeta$ and the result follows from Theorem \ref{the sod on Mi}, together with the fact that $\zeta^{2k} \otimes \theta^{-k} \simeq \Lambda_M^{-k}$ under our assumption $d = 2g - 1$.
\end{proof}

\end{document}